\documentclass[11pt,reqno]{amsart}
\usepackage[tt=false]{libertine}
\usepackage{amssymb}
\usepackage{mathtools}
\usepackage[varbb]{newpxmath}
\usepackage[foot]{amsaddr}

\let\savedbigtimes\bigtimes
\let\bigtimes\relax
\usepackage{mathabx}
\let\bigtimes\savedbigtimes

\usepackage[margin=1in, bottom=1in]{geometry}
\usepackage{graphicx}
\usepackage{caption}
\usepackage{subcaption}
\usepackage{enumitem}
\usepackage{comment}
\usepackage{bbm}
\usepackage{booktabs}
\usepackage{array}
\usepackage{xspace}
\usepackage{algorithm}
\usepackage{algpseudocode}
\usepackage{needspace} 

\usepackage[dvipsnames]{xcolor}
\usepackage{tikz}
\usetikzlibrary{arrows.meta,positioning,fit,calc}
\usepackage[colorlinks=true,
  linkcolor=teal!60!black,
  citecolor=PineGreen,
  urlcolor=RedViolet]{hyperref}
\hypersetup{bookmarksopen=true,
  pdftitle={Graph Matching in Random Geometric Graphs I: Efficient Algorithms},
  pdfauthor={Dong Huang, Ankang Jiao, Shuangping Li, Zhangsong Li}}

\usepackage[yyyymmdd,hhmmss]{datetime}

\usepackage{float}

\newtheorem{thm}{Theorem}[section]
\newtheorem{lem}[thm]{Lemma}
\newtheorem{ppn}[thm]{Proposition}
\newtheorem{prop}[thm]{Proposition}
\newtheorem{cor}[thm]{Corollary}

\theoremstyle{definition}
\newtheorem{dfn}[thm]{Definition}

\theoremstyle{remark}
\newtheorem{rmk}[thm]{Remark}

\numberwithin{equation}{section}
\def\beq#1\eeq{%
    \begin{equation}%
    #1%
    \end{equation}%
}
\def\baln#1\ealn{%
    \begin{align}%
    #1%
    \end{align}%
}
\def\balnn#1\ealnn{%
    \begin{align*}%
    #1%
    \end{align*}%
}

\def\bbN{{\mathbb{N}}}

\def\bbP{{\mathbb{P}}}
\def\bbQ{{\mathbb{Q}}}
\def\bbR{{\mathbb{R}}}

\def\bbZ{{\mathbb{Z}}}

\def\cB{{\mathcal{B}}}

\def\cF{{\mathcal{F}}}

\def\cP{{\mathcal{P}}}

\definecolor{rggA}{HTML}{317CA3}
\definecolor{rggB}{HTML}{BB4B48}
\definecolor{ink}{HTML}{303943}
\definecolor{soft}{HTML}{AAB3BC}

\newcommand{\ER}{Erd\H{o}s--R\'enyi\xspace}
\newcommand{\E}{\mathbb{E}}
\renewcommand{\P}{\mathbb{P}}
\def\de{{\mathsf{d}}}

\def\TV{{\mathsf{TV}}}
\def\gram{{\mathsf{Gram}}}
\def\adv{{\mathsf{Adv}}}

\def\ERG{{\mathsf{ER}}}
\def\RGG{{\mathsf{RGG}}}
\def\norm{{\mathsf{norm}}}
\def\cl{{\mathsf{cl}}}
\def\clean{{\mathsf{clean}}}
\def\err{{\mathsf{err}}}

\makeatletter
\renewcommand\paragraph{\@startsection{paragraph}{4}{\z@}{\z@}{-\fontdimen2\font}{\normalfont\itshape}}
\newenvironment{breakablealgorithm}
{
		\Needspace{8\baselineskip}
		\begin{center}
			\refstepcounter{algorithm}
			\hrule height.8pt depth0pt \kern2pt
			\renewcommand{\caption}[2][\relax]{
				{\raggedright\textbf{\ALG@name~\thealgorithm} ##2\par}%
				\ifx\relax##1\relax 
				\addcontentsline{loa}{algorithm}{\protect\numberline{\thealgorithm}##2}%
				\else 
				\addcontentsline{loa}{algorithm}{\protect\numberline{\thealgorithm}##1}%
				\fi
				\kern2pt\hrule\kern2pt
			}
		}{
		\kern2pt\hrule\relax
	\end{center}
}
\makeatother

\title{Graph Matching in Random Geometric Graphs I: Efficient Algorithms}

\author{Dong Huang \and Ankang Jiao \and Shuangping Li \and Zhangsong Li}

\address[D.~Huang]{Department of Statistics and Data Science, Tsinghua University}
\email[D.~Huang]{hd23@mails.tsinghua.edu.cn}

\address[A.~Jiao]{School of Mathematical Sciences, Peking University}
\email[A.~Jiao]{jiaoankang@proton.me}

\address[S.~Li]{Department of Statistics and Data Science, Yale University}
\email[S.~Li]{shuangping.li@yale.edu}

\address[Z.~Li]{School of Mathematical Sciences, Peking University}
\email[Z.~Li]{ramblerlzs@pku.edu.cn}

\date{\today}

\begin{document}

\begin{abstract}
We study the recovery of the hidden vertex correspondence between two correlated dense random geometric graphs. Each graph is generated from $n$ latent unit vectors in $\bbR^d$; the vectors $u_i^A$ and $u_i^B$ of corresponding vertices satisfy $\langle u_i^A,u_i^B\rangle=\rho$, and the vertex labels of the second graph are permuted by an unknown permutation $\pi$. For every fixed correlation $\rho\in(0,1]$ and every edge density bounded away from zero and one, we give a randomized polynomial time algorithm that recovers $\pi$ exactly with probability tending to one whenever $d\ge(\log n)^L$, where $L$ depends only on the correlation and the density bounds. The algorithm assigns to each vertex a signature formed by counting decorated paths with a constant number of seeds. The number of such seeded templates grows exponentially with the path length, with an arbitrarily large constant base, so the resulting statistics amplify an arbitrarily small constant correlation. We control the dependencies created by the latent geometry in two ways: when $d$ is larger than $n/(\log n)^{L_0}$ for a suitable constant $L_0$, by comparing local moments with those of correlated \ER graphs; for smaller $d$, by spectral estimation of the latent geometry combined with a harmonic correction. We also provide low-degree evidence for a tradeoff between the correlation and the computational cost.
\end{abstract}

\maketitle

\setcounter{tocdepth}{1}
\tableofcontents

\section{Introduction}\label{s:intro}

\subsection{Background and our results}

Graph matching, also known as graph alignment, seeks a correspondence between the vertices of two graphs. The problem is motivated by applications in computational biology \cite{singh2008global}, social networks \cite{narayanan2008robust,narayanan2009anonymizing}, and natural language processing \cite{haghighi2005robust}. A standard formulation asks for a bijection that maximizes the number of common edges after relabeling. When the graphs are isomorphic, finding such a correspondence reduces to the graph isomorphism problem, which admits a quasi-polynomial time algorithm \cite{babai2016graph}. In general, graph matching is a quadratic assignment problem and is computationally intractable in the worst case, with hardness results even for approximation \cite{makarychev2014maximum}.

The problem remains challenging even when the two graphs are generated by simple random models. The most extensively studied example is the correlated \ER model \cite{pedarsani2011privacy}, in which two \ER graphs are correlated through a hidden vertex correspondence; although idealized, this model has taken many years of work to understand and exhibits a rich range of phenomena. These efforts have led to substantial progress on two central questions for this model and closely related ones, namely the information-theoretic thresholds and the design of computationally efficient algorithms. Sharp information-theoretic thresholds for detection and various notions of recovery have been established in many regimes \cite{cullina2016improved,cullina2017exact,hall2023partial,ganassali2021impossibility,wu2022settling,wu2023testing,ding2023detection,ding2023matching,huang2025information,huang2026information}. On the algorithmic side, a sequence of works has developed increasingly effective procedures for matching and related testing problems \cite{barak2019nearly,ganassali2020tree,mossel2020seeded,ding2021efficient,fan2023spectral,fan2023spectral2,mao2023exact,mao2023random,ganassali2024correlation,ganassali2024statistical,mao2024testing,ding2025polynomial,ding2025polynomial2}.
A useful distinction is between sparse and dense graphs, according to whether the average degree grows subpolynomially or polynomially in the number of vertices. In sparse models, the threshold $\sqrt{\alpha}$ appears in known guarantees for matching and related correlation testing problems \cite{ganassali2024correlation,ganassali2024statistical,mao2024testing,mao2023random}, where $\alpha\approx0.338$ is Otter's constant~\cite{otter1948number}. In the dense regime, polynomial time algorithms succeed at every fixed positive edge correlation \cite{ding2025polynomial2}, with related results for Gaussian matrix matching \cite{ding2025polynomial}. Complementary low-degree analyses provide evidence for computational barriers in parts of the parameter space not covered by these algorithms \cite{ding2025low}.

The analyses of these algorithms are powerful, but they rely on the independence of the edge pairs associated with distinct vertex pairs under the true alignment, so they do not automatically extend to graph distributions with latent structure, which creates dependence among the edges. Such structure is important both for modeling realistic networks and for understanding which features determine algorithmic thresholds. Examples of such models include inhomogeneous random graphs \cite{ding2025efficiently,vassaux2026graph}, random growing graphs \cite{racz2022correlated,baumler2026correlated}, and stochastic block models \cite{racz2021correlated,gaudio2022exact,chen2026computational,chen2025detecting}.

In this paper, we study matching in a pair of correlated spherical random geometric graphs, one of the most canonical models with latent structure. This random graph model describes networks whose vertices carry latent spatial or geometric features \cite{penrose2003random, dall2002random}, and its edges encode proximity between the unobserved features. Unlike in the \ER model, the edges are dependent (for instance, two vertices with a common neighbor are more likely to be adjacent), which makes the analysis considerably more involved. Random geometric graphs have been studied extensively from the perspectives of testing \cite{bubeck2016testing,brennan2020phase,liu2022testing,mao2026random,du2026resolution,du2026sharp} and of recovering the latent geometry \cite{valdivia2019latent,eldan2022community,li2023spectral,mao2024impossibility,fernandez2026spectral}, but the matching problem remains much less understood.

\begin{dfn}[Correlated spherical random geometric graphs]\label{def:correlated-RGG}
Let $[n]=\{1,\ldots,n\}$ be the vertex set. For each vertex $i\in[n]$, independently sample a pair of unit vectors $u_i^A,u_i^B\in\bbR^d$ such that $\|u_i^A\|_2=\|u_i^B\|_2=1$ and
\[
    \langle u_i^A,u_i^B\rangle=\rho,\quad 0\le\rho\le1.
\]
That is, $u_i^A$ is uniform on the unit sphere and, conditional on $u_i^A$, the vector $u_i^B$ is uniform among unit vectors with inner product $\rho$ with $u_i^A$. Fix $p_0\in(0,1/2]$ and $p\in[p_0,1-p_0]$, and choose $\tau_{p,d}$ so that two independent uniform unit vectors have inner product at least $\tau_{p,d}$ with probability $p$. For a hidden permutation $\pi:[n]\to[n]$, the observations are the adjacency matrices $A,B$ with zero diagonal defined, for $i\ne j$, by
\[
    A_{ij}=\mathbf1\{\langle u_i^A,u_j^A\rangle\ge \tau_{p,d}\},\quad
    B_{\pi(i),\pi(j)}=\mathbf1\{\langle u_i^B,u_j^B\rangle\ge \tau_{p,d}\}.
\]
\end{dfn}

The parameters $n,d,p,\rho$ are known, and the goal is to estimate $\pi$. Exact recovery means that an estimator $\widehat\pi$ satisfies $\P[\widehat\pi=\pi]\to1$ as $n\to\infty$; almost exact recovery allows $o(n)$ mismatched vertices. The parameter $\rho$ is the correlation of corresponding latent vectors, which differs from the correlation of the binary edges: for $i\ne j$, the inner products $\langle u_i^A,u_j^A\rangle$ and $\langle u_i^B,u_j^B\rangle$ have correlation $\rho^2$, whereas the edge correlation also depends on $p$ and $d$. At $\rho=0$ the paired vectors are orthogonal, so the model does not in general reduce to two independent graphs.

If the full Gram matrices $A^{\gram}_{ij}=\langle u_i^A,u_j^A\rangle$ and $B^{\gram}_{\pi(i),\pi(j)}=\langle u_i^B,u_j^B\rangle$ were observed, the problem would become a version of Wasserstein--Procrustes matching \cite{wang2022random, gong2024umeyama, even2024aligning, li2026computational, niu2026high}. For a Gaussian version of this problem, \cite{wang2022random} establishes sharp information-theoretic thresholds in the low-dimensional regime $d=o(\log n)$. In its correlation parametrization, the conditions are $1-\rho^2=o(n^{-4/d})$ for exact recovery and $1-\rho^2=o(n^{-2/d})$ for almost exact recovery. Subsequently, \cite{gong2024umeyama} shows that the classical Umeyama algorithm \cite{umeyama1988eigendecomposition} attains these thresholds up to polynomial factors in $d$, in polynomial time. The case $d\gg\log n$ is less well understood, even with access to the Gram matrices. For a Gaussian model with $d\gg\log n$, \cite{niu2026high} gives an algorithm for exact recovery when $\rho^2>\sqrt{\alpha}$, where $\alpha\approx0.338$ is Otter's constant. 

Our observations are far coarser than the Gram matrices and record only whether each inner product exceeds a threshold. Nevertheless, once the dimension is at least a sufficiently large power of $\log n$, we show that exact recovery is possible in polynomial time at every fixed positive latent correlation. 

\begin{thm}\label{thm:main-alg}
Fix $p_0\in(0,1/2]$ and $\rho\in(0,1]$. There exist finite constants $L=L(p_0,\rho)$ and $C=C(p_0,\rho)$ and a randomized algorithm $\mathcal A$ (see Algorithm~\ref{alg:rgg-matching}) such that, uniformly over $p\in[p_0,1-p_0]$, dimensions $d\ge(\log n)^L$, and hidden permutations $\pi$, the algorithm takes $(n,d,p,\rho,A,B)$ as input and returns an estimator $\widehat\pi$ in time $O(n^C)$ with $\P[\widehat\pi=\pi]\to1$ as $n\to\infty$.
\end{thm}

Theorem~\ref{thm:main-alg} shows that exact recovery is possible in polynomial time at every fixed positive correlation even in the presence of geometric dependence, extending a phenomenon previously known only for models with independent edges.
In addition, our algorithm provides a robust and general framework for graph matching. Since the latent geometry enters the analysis only through spectral estimates of the latent vectors and comparisons of local moments with the correlated \ER model, our approach is not specific to random geometric graphs, and we expect it to apply to other dense models. In particular, it should apply to the correlated \ER model itself and give a more streamlined algorithm and analysis than those of \cite{ding2025polynomial,ding2025polynomial2}; see \S\ref{ss:algorithmic-innovations}. It should also apply to models with latent structure, such as stochastic block models and inhomogeneous random graphs, although we do not pursue these extensions here.

We complement the algorithmic result with a lower bound for testing at vanishing correlation in the low-degree polynomial framework \cite{hopkins2018statistical}. This framework uses polynomial statistics of the observations as a proxy for efficient algorithms and measures their performance by the difference between their means under the two models, relative to their standard deviation under the null model. Let $\P$ be the law in Definition~\ref{def:correlated-RGG} with $\pi$ uniform and independent of the latent pairs, and let $\bbQ$ be the law of two independent spherical random geometric graphs with the same $n,d,p$. We consider real polynomial statistics $f=f(A,B)$, where $\deg f$ denotes the total degree in the upper triangular entries of both matrices.

\begin{thm}\label{thm:main-hardness}
Let $D=D_n\ge1$ be an integer. If $D\rho^2\to0$ and $D/d\to0$ as $n\to\infty$, then
\[
\sup_{\substack{\deg f\le D\\
                \operatorname{Var}_{\bbQ}(f)>0}}
\frac{|\E_{\P}[f]-\E_{\bbQ}[f]|}
     {\sqrt{\operatorname{Var}_{\bbQ}(f)}}
=o(1).
\]
\end{thm}


Together, Theorems~\ref{thm:main-alg} and \ref{thm:main-hardness} suggest that polynomial time algorithms for graph matching in high-dimensional random geometric graphs have no sharp correlation threshold. In Theorem~\ref{thm:main-alg}, a smaller fixed $\rho$ is handled by a larger, but still constant, exponent $C$ of the running time; increasing the polynomial cost of the algorithm thus lets the correlation be arbitrarily small. Theorem~\ref{thm:main-hardness} is consistent with this picture: polynomials of degree $D$ fail whenever $D\rho^2\to0$ and $D/d\to0$, so in high dimension the degree needed to detect the correlation grows at least like $\rho^{-2}$. Related smooth computational transitions have been studied for tensor PCA \cite{li2025smooth, kothari2026smooth}, sparse PCA \cite{ding2024subexponential}, and dense random graph matching and correlation testing \cite{ding2025polynomial2, ding2025low, huang2025testing}.

Finally, our companion paper~\cite{huang2026matching} studies information-theoretic recovery in the same model. For sufficiently large polylogarithmic dimension, it shows that the threshold for exact recovery is, up to constant factors,
\[
    \rho^2\asymp\max\Big\{\frac{\log d}{d},\sqrt{\frac{\log n}{n}}\Big\}.
\]
This vanishing information-theoretic threshold, together with the low-degree obstruction in Theorem~\ref{thm:main-hardness}, provides evidence for a substantial \emph{information--computation gap}.

\subsection{Algorithmic ideas}\label{ss:algorithmic-innovations}

A common approach to graph matching replaces the quadratic assignment objective by informative pairwise similarity scores. We follow this approach in three steps:
\begin{itemize}
    \item \emph{Signature construction:} Associate each vertex $i$ in $A$ with a vector $\Phi^A(i)$ and each vertex $j$ in $B$ with a vector $\Phi^B(j)$ in a common signature space.
    \item \emph{Similarity scoring:} Compare the signatures through their inner product $S(i,j)=\langle\Phi^A(i),\Phi^B(j)\rangle$.
    \item \emph{Assignment and refinement:} Retain pairs with sufficiently large scores, construct a candidate matching, and correct its remaining errors to recover the hidden correspondence exactly.
\end{itemize}
The main challenge is to construct signatures whose scores separate true pairs from false pairs and can be computed in polynomial time. 
To motivate our construction, we first briefly explain how the threshold $\sqrt\alpha$ emerges as a natural barrier for the previous tree-counting algorithms \cite{mao2023random, niu2026high}. 
For this heuristic discussion, consider a correlated \ER model in which corresponding standardized edge variables have correlation $r$; this edge correlation is distinct from the correlation $\rho$ of the latent vectors in our geometric model. For a connected rooted graph $H$ with at least one edge, a natural score is the product of its centered rooted subgraph counts,
\[
    g_H(i,j):= \bigg(\sum_{\substack{S\cong H\\\text{root}(S)=i}} \prod_{\{u,v\}\in E(S)}\overline A_{uv}\bigg) \bigg(\sum_{\substack{S\cong H\\\text{root}(S)=j}} \prod_{\{u,v\}\in E(S)}\overline B_{uv}\bigg),
\]
where $\overline A,\overline B$ denote the adjacency entries centered and scaled to unit variance, and the sums run over rooted copies of $H$ in the complete graph on $[n]$. For a true pair $(i,\pi(i))$, normalizing this product by its standard deviation under two independent graphs gives mean $r^{|E(H)|}$. Under this independent null, distinct rooted graph types give orthogonal statistics. This suggests the following heuristic for the \emph{squared} signal-to-noise ratio of a suitably weighted combination:
\begin{equation}\label{eq:informal-snr}
    \sum_{[H]} r^{2|E(H)|},
\end{equation}
where the sum ranges over the chosen rooted isomorphism types. The constructions in \cite{mao2023random, niu2026high} use families of trees whose exponential growth rate approaches that of unlabeled trees, namely $\alpha^{-1}$ \cite{otter1948number}. The contribution of trees with $m$ edges therefore has exponential rate $r^2/\alpha$, which explains the threshold $r=\sqrt\alpha$.

Our first challenge is to amplify an arbitrarily small fixed correlation $r$ below the tree threshold $\sqrt\alpha$. To this end, we construct a family of templates whose exponential growth rate can be made arbitrarily large. A large family alone does not suffice: we also need to control the correlations between the counts and to compute them efficiently. We meet these requirements by decorating a common path with labels from a fixed set of seeds. 
Suppose first that $K$ true seed pairs $(a,\pi(a))$, $a\in[K]$, are available, where $K$ is a sufficiently large constant. Let $V=\lceil C_0\log\log n\rceil$ for a constant $C_0>0$, and define
\[
    \mathcal W=[K]^{V-1},\quad J=|\mathcal W|=K^{V-1}.
\]
A word $w=(w_1,\ldots,w_{V-1})\in\mathcal W$ specifies a rooted template $H_w$ with seed labels, as illustrated in Figure~\ref{fig:word-family}. Start with a path with vertices $\star,r_1,\ldots,r_V$ in this order, rooted at $\star$, and with fixed seed vertices $s_1,\ldots,s_K$ outside the path. For each $t\in[V-1]$, attach $r_t$ to the seed $s_{w_t}$, as in the upper schematic of Figure~\ref{fig:word-family}. Equivalently, 
\[
    E(H_w)=\{\{\star,r_1\}\}
    \cup\{\{r_t,r_{t+1}\}:1\le t<V\}
    \cup\{\{s_{w_t},r_t\}:1\le t<V\}.
\]
The symbols $r_1,\ldots,r_V$ denote positions in the path, which are filled by distinct vertices of the graph, and unused seeds may be omitted. The root and the terminal position $r_V$ carry no seed. Repetitions in $w$ are allowed: if $w_t=w_{t'}$, then $r_t$ and $r_{t'}$ are attached to the \emph{same} seed vertex, so $H_w$ is in general not a tree.

\begin{figure}[htbp]
\centering
\includegraphics[width=0.8\linewidth]{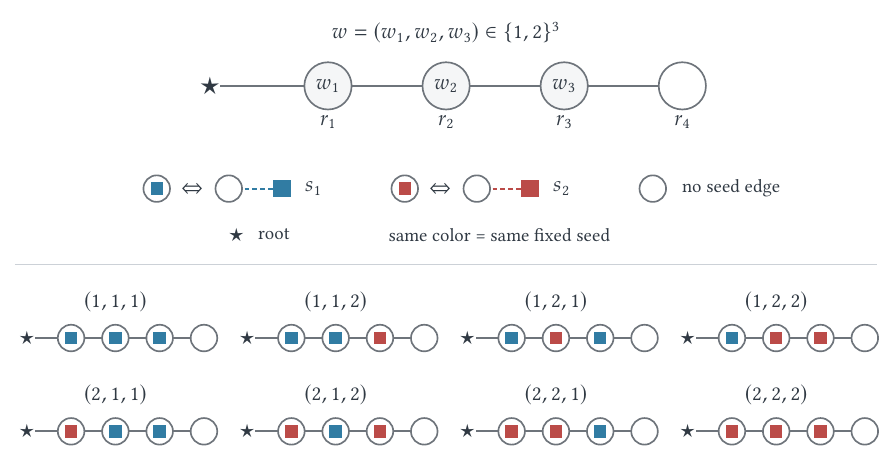}
\caption{The word family $\mathcal W=[2]^3$ for $K=2$ and $V=4$, represented by decorated paths.
The upper schematic gives the general construction, and the lower panels show all eight templates $H_w$.
The star is the query root, and circles are reservoir positions.
The decoration at $r_t$ records the seed edge $(r_t,s_{w_t})$: blue denotes $s_1$ and red denotes $s_2$.
Repeated colors refer to the same fixed seed vertex; $r_4$ has no seed attachment.
These colors encode seeds and are unrelated to the random colors used for color coding in \S\ref{s:color-coding}.}
\label{fig:word-family}
\end{figure}

More precisely, in the high-dimensional regime $d>n/(\log n)^{L_0}$, each rooted template $H_w$ gives a normalized centered subgraph count. For $G\in\{A,B\}$, write $s_a^G$ for the seed with label $a$, so that $s_a^A=a$ and $s_a^B=\pi(a)$ for the true seed pairs. Let $R_G$ be a uniformly sampled reservoir of $N$ nonseed vertices, and take a query vertex $i$ outside $R_G$ and the seeds. Define $Z^G_{uv}=(G_{uv}-p)/\sqrt{p(1-p)}$ for $u\ne v$, and set $Z^G_{uu}=0$. The signature coordinate indexed by $w$ is
\[
    \Phi_w^G(i)
    =\frac{1}{\sqrt{(N)_V}}
    \sum_{\substack{r_1,\ldots,r_V\in R_G\\\text{all distinct}}}
    Z^G_{ir_1}
    \prod_{t=1}^{V-1}Z^G_{r_tr_{t+1}}Z^G_{s^G_{w_t}r_t},
\]
where $(N)_V=N(N-1)\cdots(N-V+1)$. Thus each summand corresponds to an injective embedding of $H_w$, with the root fixed at $i$, the seeds fixed at $s_a^G$, and the path positions filled by distinct reservoir vertices; every edge contributes its centered and normalized adjacency entry. Collecting these counts gives $\Phi^G(i)=(\Phi_w^G(i))_{w\in\mathcal W}$, as in \eqref{eq:high-exact-signature}. Each template $H_w$ has $2V-1$ edges, so by~\eqref{eq:informal-snr} it contributes $r^{2(2V-1)}$ to the \emph{squared} signal-to-noise ratio. Summing over the $|\mathcal W|=K^{V-1}$ words gives the total
\[
    K^{V-1}r^{2(2V-1)}=r^2(Kr^4)^{V-1}.
\]
See Figure~\ref{fig:word-family-snr} for an illustration. The intermediate regime $(\log n)^L\le d\le n/(\log n)^{L_0}$ is more subtle. Geometric contributions to the moments can accumulate over the many path embeddings, so the comparison with independent edge variables used in the high-dimensional regime does not suffice throughout this range. We instead estimate the latent geometry spectrally and evaluate corrected path weights on the resulting coordinates. See \S\ref{ss:low-signature-construction} for more details.

\begin{figure}[htbp]
\centering
\includegraphics[width=0.8\linewidth]{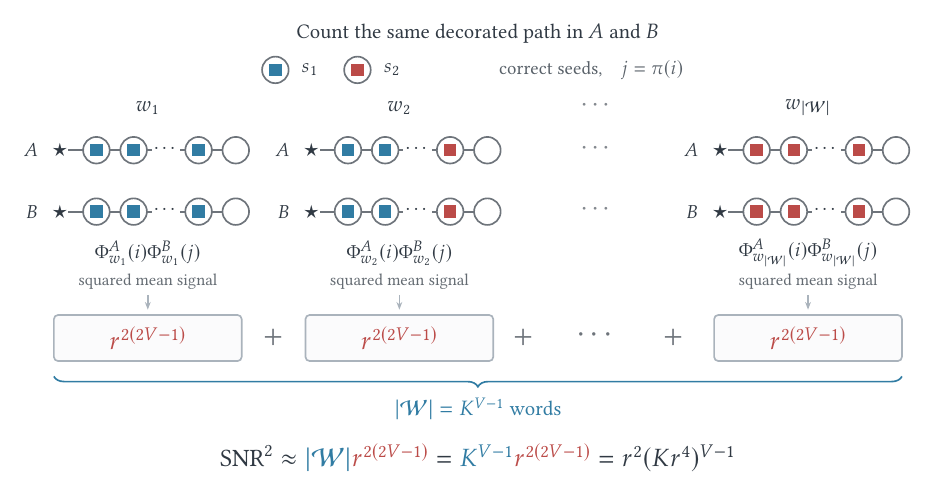}
\caption{Counting the same template in $A$ and $B$. For a true pair $j=\pi(i)$ and correctly matched seeds, each word gives a squared signal of order $r^{2(2V-1)}$. Summing over the $|\mathcal W|=K^{V-1}$ words yields the heuristic squared signal-to-noise ratio $r^2(Kr^4)^{V-1}$, ignoring the losses from the partial overlap of the reservoirs and the geometric errors.}
\label{fig:word-family-snr}
\end{figure}

Unlike the growth rate $\alpha^{-1}$ of unlabeled trees, the base $K$ can be chosen arbitrarily large, so any fixed $r>0$ is compensated by taking $K>r^{-4}$. With $V=\Theta(\log\log n)$, the signature dimension $J$ remains polylogarithmic in $n$. Direct enumeration of all path embeddings takes superpolynomial time, but color coding computes approximate signatures in polynomial time. With $M_{\mathrm{col}}=(\log n)^{O(1)}$ independent colorings, each candidate permutation is computed in time $O(n^{2+o(1)})$ once a common preprocessing step is available. For a single seed guess, the preprocessing, the candidate construction, and the final refinement take time $O(n^{3+o(1)})$. Since the true seeds are not given, we fix $K$ source vertices and enumerate all ordered tuples of $K$ distinct target vertices. This list contains the true seed tuple and multiplies the cost of candidate construction and selection by at most $n^K$. The preprocessing is shared by all seed guesses and only the selected candidate is refined, so the total running time is $O(n^{K+2+o(1)}+n^3)$ by Proposition~\ref{prop:running-time}.

This template construction is also closely related to the iterative algorithms in \cite{ding2025polynomial,ding2025polynomial2}, which amplify the signal in a similar way. The algorithms in \cite{ding2025polynomial,ding2025polynomial2} start from a constant number of guessed seed pairs. Each iteration aggregates features through the observed matrix and applies nonlinear transformations, producing more features with controlled correlations, and the growing number of features compensates for the loss of correlation in each one. Our construction expresses this idea through path counts. Each extension of a word propagates along one path edge and multiplies by one of the $K$ seed features; the seed attachment records this multiplication, and the $K$ choices enlarge the family at each level.
For the construction from adjacency entries used in the high-dimensional regime, the recursion behind these counts is explicit. For $G\in\{A,B\}$, let $R_G$ be the reservoir, the random set of vertices through which the paths run, and let $s_a^G$ be the seed with label $a$. Define $Z^G_{uv}=(G_{uv}-p)/\sqrt{p(1-p)}$ for $u\ne v$, and let $Z^G_{uu}=0$, so that each entry off the diagonal has mean zero and variance one. For now, omit the normalization and allow repeated reservoir vertices. For each $a\in[K]$, update a message $h:R_G\to\bbR$ by
\[
    h\longmapsto
    \Big(Z^G_{s_a^G r}\sum_{r'\in R_G}Z^G_{rr'}h(r')\Big)_{r\in R_G}.
\]
Start from $h\equiv1$, apply the updates in the order $w_{V-1},\ldots,w_1$, and finally form $\sum_{r\in R_G}Z^G_{ir}h(r)$ for the query vertex $i$. Expanding this expression gives the unrestricted walk sum associated with $w$, and keeping only the terms with distinct reservoir vertices gives the path sum in \eqref{eq:high-exact-signature} before normalization.

A related example is the use of non-backtracking path counts in community detection \cite{mossel2018proof,bordenave2015non}. Linearizing belief propagation around its uninformative fixed point gives a non-backtracking message passing operator, and expanding its powers gives sums over non-backtracking walks \cite{krzakala2013spectral}. These counts describe the linearized dynamics without reproducing the full nonlinear updates. Our word counts have a similar relation to the iterations in \cite{ding2025polynomial,ding2025polynomial2}: they give a polynomial version of the amplification mechanism without reproducing each nonlinear update. Our paths also satisfy a stronger constraint than non-backtracking walks, since no reservoir vertex is repeated. This explicit counting form lets us analyze the amplification through the overlaps of templates. 

The second challenge is to preserve this amplification despite the edge dependencies induced by the latent geometry. We use different constructions in the \emph{intermediate regime} $(\log n)^L\le d\le n/(\log n)^{L_0}$ and in the \emph{high-dimensional regime} $d>n/(\log n)^{L_0}$. This split is motivated by the role of dimension in the reconstruction of latent geometry \cite{valdivia2019latent, eldan2022community, li2023spectral, mao2024impossibility,fernandez2026spectral}.
In the intermediate regime, spectral methods approximately recover the latent geometry \cite{li2023spectral}, so we can evaluate the seeded paths on estimated coordinates. To explain the construction, suppose first that the latent vectors are known. For a fixed word $w$ and distinct reservoir vertices $r_1,\ldots,r_V$, the natural path weight in graph $A$ is
\[
    \langle u_i^A,u_{r_1}^A\rangle
    \prod_{t=1}^{V-1}
        \langle u_{r_t}^A,u_{r_{t+1}}^A\rangle
        \langle u_{r_t}^A,u_{w_t}^A\rangle
    =(u_i^A)^\top\bigg(\prod_{t=1}^{V-1}
        \big(\langle u_{r_t}^A,u_{w_t}^A\rangle u_{r_t}^A(u_{r_t}^A)^\top\big)\bigg)u_{r_V}^A.
\]
The corresponding weight in graph $B$ is defined analogously, using the proposed seed counterparts. The difficulty is that the cubic matrix polynomial $(z^\top v)zz^\top$, viewed as a function of $z$ on the sphere, contains spherical harmonics of degree three and of degree one. The component of degree one creates unwanted contributions of lower order to the overlap moments of different words: a reservoir vertex contributes even when it is internal in one path and terminal in another. Figure~\ref{fig:cubic-path-cancellation} in \S\ref{ss:coordinate-path-moments} gives an example and explains its cancellation. We remove this component by replacing each cubic factor with its harmonic part of degree three,
\[
    H_3(z)[v]=(z^\top v)zz^\top-\frac{\|z\|^2}{d+2}
        \big(zv^\top+vz^\top+(z^\top v)I_d\big).
\]
The corrected path weight is
\[
    (u_i^A)^\top\bigg(\prod_{t=1}^{V-1}H_3(u_{r_t}^A)[u_{w_t}^A]\bigg)u_{r_V}^A.
\]
Figure~\ref{fig:h3-correction-overview} compares this construction with the one used in the high-dimensional regime. Summing the corrected weights over distinct reservoir embeddings and applying a deterministic normalization gives the ideal signature $\Phi^{u,A}(i)=(\Phi_w^{u,A}(i))_{w\in\mathcal W}$; the signature $\Phi^{u,B}$ is defined analogously. The harmonic correction provides the cancellations needed to control the correlations between words. If $\mu$ denotes the uniform probability measure on $\mathbb S^{d-1}$, then symmetry and the spherical second and fourth moment identities give, for all fixed $a,v\in\bbR^d$,
\[
    \int H_3(u)[v]\mu(\de u)=0,
    \quad
    \int H_3(u)[v]\langle u,a\rangle\mu(\de u)=0.
\]
Thus the corrected factor is centered and orthogonal to every linear function of the reservoir direction.

To make these signatures computable, we replace the latent vectors by \emph{spectral embeddings}. For a true pair $(i,\pi(i))$, the vectors $f_A(i)$ and $f_B(\pi(i))$ approximate $\sqrt{d}u_i^A$ and $\sqrt{d}u_i^B$, up to separate orthogonal transformations. Averaging over an auxiliary set of vertices produces the \emph{estimated seed directions} $v_a^A,v_a^B$ at unit scale. Substituting these estimates into the normalized path sums gives the signatures in \eqref{eq:low-exact-signature}. Each coordinate of a signature is a scalar contraction indexed by a word, which is invariant under an orthogonal change of coordinates in each graph, so the spectral embeddings of the two graphs may use unrelated coordinate systems.
In the high-dimensional regime, we evaluate the templates on the centered adjacency entries, as in \eqref{eq:high-exact-signature}, without estimating the geometry. The analysis compares only the local moments that enter the path counts with those of correlated \ER graphs. This comparison keeps the amplification of the model with independent edges and controls the geometric error terms.

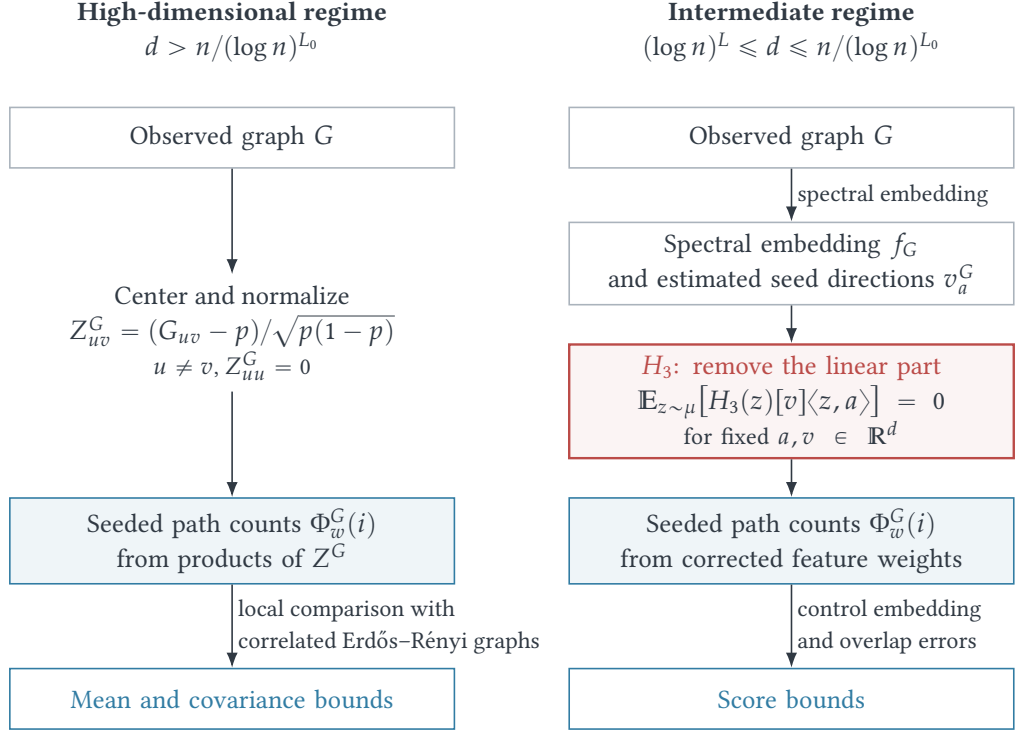
\begin{figure}[htbp]
\centering
\begin{tikzpicture}[
  font=\small, text=ink, >={Latex[length=2mm]},
  box/.style={draw=soft, line width=.6pt, fill=white, align=center,
              inner sep=4pt, minimum height=8mm, text width=5.6cm},
  count/.style={box, draw=rggA, fill=rggA!8},
  bound/.style={box, draw=rggA, text=rggA},
  corr/.style={box, draw=rggB, fill=rggB!6, line width=.9pt},
  build/.style={->, draw=ink, line width=.6pt},
  analysis/.style={->, draw=ink, dashed, line width=.6pt},
  lab/.style={font=\footnotesize, align=left, inner sep=2pt},
]
\node[align=center] (tL) at (0,1.0)
  {\textbf{High-dimensional regime}\\ $d>n/(\log n)^{L_0}$};
\node[align=center] (tR) at (7.4,1.0)
  {\textbf{Intermediate regime}\\ $(\log n)^L\le d\le n/(\log n)^{L_0}$};
\node[box] (R1) at (7.4,-0.4) {Observed graph $G$};
\node[box, below=7mm of R1] (R2) {Spectral embedding $f_G$\\ and estimated seed directions $v_a^G$};
\node[corr, below=5mm of R2] (R3)
  {{\color{rggB}$H_3$: remove the linear part}\\
   $\E_{z\sim\mu}\big[H_3(z)[v]\langle z,a\rangle\big]=0$\\
   {\footnotesize for fixed $a,v\in\bbR^d$}};
\node[count, below=5mm of R3] (R4)
  {Seeded path counts $\Phi^G_w(i)$\\ from corrected feature weights};
\node[bound, below=11mm of R4] (R5) {Score bounds};
\node[box] (L1) at (0,-0.4) {Observed graph $G$};
\node[count] (L3) at (L1 |- R4)
  {Seeded path counts $\Phi^G_w(i)$\\ from products of $Z^G$};
\node[align=center, inner sep=4pt] (L2) at ($(L1.south)!0.5!(L3.north)$)
  {Center and normalize\\
   $Z^G_{uv}=(G_{uv}-p)/\sqrt{p(1-p)}$\\
   {\footnotesize $u\ne v$, $Z^G_{uu}=0$}};
\node[bound] (L4) at (L1 |- R5) {Mean and covariance bounds};
\draw[build] (L1) -- (L2);
\draw[build] (L2) -- (L3);
\draw[build] (L3) -- node[lab, right] {local comparison with\\ correlated \ER graphs} (L4);
\draw[build] (R1) -- node[lab, right] {spectral embedding} (R2);
\draw[build] (R2) -- (R3);
\draw[build] (R3) -- (R4);
\draw[build] (R4) -- node[lab, right] {control embedding\\ and overlap errors} (R5);
\coordinate (leg) at ($(L4.south)!0.5!(R5.south)+(0,-0.6)$);
\end{tikzpicture}
\caption{Signatures in the two regimes. In the high-dimensional regime, the seeded path counts are products of entries of the centered and normalized adjacency matrix $Z^G$, and their means and covariances are bounded by a local comparison with correlated \ER graphs. In the intermediate regime, the counts are computed from the spectral embedding $f_G$ and the estimated seed directions $v_a^G$, after $H_3$ removes the linear part of each cubic factor; the scores are bounded by controlling the errors of the embedding and the overlaps of paths. 
}
\label{fig:h3-correction-overview}
\end{figure}

Another ingredient for the algorithm is a concentration argument which allows much shorter templates than a direct second moment approach. In constructions based on subgraph counts such as \cite{barak2019nearly, mao2024testing, mao2023random, niu2026high}, ruling out all false matches by Chebyshev's inequality and a union bound typically requires a squared signal-to-noise ratio polynomial in $n$, which leads to templates of size $\Theta(\log n)$. Our templates have only $V=\Theta(\log\log n)$ path positions and yield a polylogarithmic squared signal-to-noise ratio. To make this sufficient, we condition on the latent vectors of the preprocessing, seed, and reservoir vertices and keep the query vertices fresh. The resulting bilinear structure gives exponential tail bounds, via a Hanson--Wright type argument, that are strong enough to control the scores of all false pairs simultaneously.
In the high-dimensional regime, the scores of true pairs are handled by a second moment argument, since missing $o(n)$ true pairs is harmless at the stage of candidate matching. To complete the algorithm, we identify an accurate candidate among the seed guesses and correct its remaining errors, using the selection and refinement procedure motivated by \cite{ding2025polynomial2}. The verification estimates hold uniformly over all permutations, so the same observed graphs can be used both to construct the candidate list and to perform selection and refinement.

Finally, we describe our strategy to prove the low-degree lower bound in Theorem~\ref{thm:main-hardness}. Inspired by \cite{li2026computational}, we will first reduce the complex low-degree advantage to a simpler chi-square divergence between the Gram matrices of $D$ aligned latent pairs and those of two independent point clouds. An invariant spherical harmonic expansion and a moment bound show that the $\chi^2$ divergence between these laws is $O(D^2(\rho^2+d^{-1})^2)$ when $D(\rho^2+d^{-1})$ is sufficiently small. Taking square roots gives the vanishing normalized mean separation. See \S\ref{s:low-degree} for details. 

\subsection{Future perspectives}

Our work also highlights several important directions for future research, which we discuss below.

\textit{General kernels and sparse geometric graphs.} A natural extension is to replace the threshold rule by a general connection function of the latent inner product. 
It would be useful to identify conditions on the kernel and on the correlation between the corresponding latent features under which seeded path statistics still amplify the matching signal.
A particularly important goal is to allow the edge density $p$ to tend to zero, and ultimately to treat sparse graphs with $p=n^{-1+o(1)}$.  
Determining the joint roles of $d$, $p$, and $\rho$ in exact and almost exact recovery would clarify whether the absence of a positive constant correlation threshold in our dense model persists in any sparse regimes.

\textit{The low-dimensional regime.} Another important question is to understand matching from adjacency matrices when $d$ is below the polylogarithmic range of Theorem~\ref{thm:main-alg}, especially when $d$ is fixed or $d=o(\log n)$. With the Gram matrices available, the low-dimensional problem is well understood for related Gaussian models
\cite{wang2022random,gong2024umeyama}. However, the corresponding problem with only graph observations remains largely unresolved. Thresholding discards the inner products themselves, so transferring a Gram-matrix matching algorithm requires reconstructing the geometry accurately enough to distinguish nearby vertices. It is therefore natural to ask whether such sufficiently accurate reconstruction is possible, or whether direct graph statistics can achieve recovery under weaker requirements. 

\textit{Inhomogeneity and more general edge dependencies.} A broader direction is to extend the method to models with nonuniform edge probabilities, latent community structure, or dependencies beyond those generated by independent latent vertex pairs. Work on inhomogeneous random graphs, stochastic block models, and randomly growing graphs provides natural points of comparison \cite{ding2025efficiently,racz2021correlated,racz2022correlated}. 
We emphasize that it is also important to propose and study correlated graph models based on important real-world and scientific problems. 


\subsection{Notation, conventions, and organization}\label{ss:notation}

Let $\mu$ be the uniform probability measure on $\mathbb S^{d-1}=\{u\in\bbR^d:\|u\|=1\}$. For $u,w\in\mathbb S^{d-1}$, let $x(u)=\sqrt{d}u$ and $h_0(u,w)=\mathbf1\{\langle u,w\rangle\ge\tau_{p,d}\}-p$. The cap kernel $h_0$ acts on linear functions as multiplication by a scalar $\lambda$: by rotational invariance,
\[
    \int h_0(u,w)x(w)\mu(\de w)=\lambda x(u), \quad \int x(w)x(w)^\top\mu(\de w)=I_d,
\]
where
\begin{equation}\label{eq:eigen-lambda}
    \lambda =\E[h_0(u,u')\langle u,u'\rangle] =\frac{\Gamma(d/2)}{(d-1)\sqrt{\pi}\Gamma((d-1)/2)} (1-\tau_{p,d}^2)^{(d-1)/2}
\end{equation}
and $u,u'$ are independent with law $\mu$. In particular, $\lambda>0$ and $\lambda\asymp d^{-1/2}$ uniformly for $p\in[p_0,1-p_0]$. For $\vartheta\in[-1,1]$, a \emph{$\vartheta$-pair} is a pair $(u,u')$ of unit vectors in $\bbR^d$ such that $u\sim\mu$ and, given $u$, the vector $u'$ is uniform on $\{w\in\mathbb S^{d-1}:\langle u,w\rangle=\vartheta\}$. The latent pairs $(u_i^A,u_i^B)$, $i\in[n]$, of Definition~\ref{def:correlated-RGG} are independent $\rho$-pairs.

For vectors, $\|\cdot\|$ is the Euclidean norm; for matrices, $\|\cdot\|_{\mathrm{op}}$ and $\|\cdot\|_{\mathrm F}$ are the operator and Frobenius norms. We write $\mathfrak S_n$ for the set of permutations of $[n]$ and $\mathrm{id}$ for the identity permutation. For probability laws $R,S$, define $\chi^2(R\|S)=\int(dR/dS-1)^2dS$ when the density exists and is square integrable, and let $\chi^2(R\|S)=+\infty$ otherwise.

Except in Theorem~\ref{thm:main-hardness} and \S\ref{s:low-degree}, where $p\in(0,1)$ and $\rho$ may depend on $n$ and all constants are absolute, the parameters $p_0\in(0,1/2]$ and $\rho\in(0,1]$ are fixed, $p\in[p_0,1-p_0]$, and $c\in(0,1/10)$ is the fraction of reservoir vertices fixed in \S\ref{s:algorithm-construction}. Constants such as $C$, $c_1$, $C_q$, and $C_{p_0}$ are positive and depend only on $p_0,\rho,c$, on the parameters fixed in \S\ref{s:algorithm-construction}, and on fixed moment orders such as $q$; their values may change from line to line. All asymptotic statements refer to $n\to\infty$. Every $o(1)$ or $o(n)$ term and every failure probability is deterministic and uniform over $p\in[p_0,1-p_0]$, over $d$ in the stated range, and over the hidden permutation $\pi$. For positive quantities $a,b$, we write $a\asymp b$ if $c_1b\le a\le Cb$.

The rest of the paper is organized as follows.
\S\ref{s:algorithm-construction} describes the algorithm and states the guarantees for its components. \S\ref{s:preliminaries} reduces the accuracy of a candidate to the separation of scores and collects the estimates for spherical harmonics and the cap kernel used later. \S\ref{s:high-dimension-proof} and \S\ref{s:intermediate-dimension-proof} prove Proposition~\ref{prop:true-seed-candidate} in the high-dimensional regime and in the intermediate regime, respectively. \S\ref{s:color-coding} computes the signatures by color coding, proves the corresponding guarantee (Proposition~\ref{prop:true-seed-candidate-color-coding}), and combines it with Proposition~\ref{prop:rounding-finishing} to prove Theorem~\ref{thm:main-alg} (Corollary~\ref{cor:color-coded-recovery}). \S\ref{s:seeded-graph-matching} analyzes selection and refinement and proves Proposition~\ref{prop:rounding-finishing}. \S\ref{s:low-degree} proves the lower bound in the low-degree framework.

\subsection*{Acknowledgment}

Part of this work was carried out when D.H. was visiting Yale University.
S.L. would like to thank Tselil Schramm for sharing an earlier version of \cite{niu2026high} while S.L. was visiting Stanford University in March, and Yizhe Zhu for sharing the spectral embedding results of \cite{fernandez2026spectral} in July; both have been of great help to this project.
D.H. thanks Yihong Wu for helpful discussions.

\subsection*{Statement on AI use}
The results of \cite{niu2026high,fernandez2026spectral} mentioned above made the authors realize that it might be a good time to start working on graph matching in random geometric graphs: \cite{niu2026high} gives an efficient matching algorithm when the Gram matrices are observed, and with \cite{fernandez2026spectral}, the embedding argument that approximately recovers the latent geometry from the graphs is now in good shape.
Our original attempt was to apply the approach of \cite{niu2026high} to the spectral embeddings of the two graphs, and we spent some effort trying to show that this algorithm is stable, in the sense that its guarantee persists when the true Gram matrices are replaced by the estimates obtained from the embeddings.
During this process, we interacted extensively with ChatGPT (GPT-5.5 and GPT-5.6).
The first draft relied on spectral embeddings and covered only the intermediate-dimensional regime.
Along the way, we also incorporated ideas from the iterative algorithms of \cite{ding2025polynomial,ding2025polynomial2}, from which the algorithm in this paper partly originates.
Later, during interactions with GPT-6, we realized that a comparison with the correlated \ER model could extend the approach to the high-dimensional regime.
By controlling the difference between the relevant moments of path counts, we could work directly with the observed graphs.
These interactions also helped us find simpler algorithms that cover a wider range of correlations.
We subsequently turned to lower bounds; the low-degree lower bound in Theorem~\ref{thm:main-hardness} and the information-theoretic results of the companion paper \cite{huang2026matching} were developed in the course of the same process.

In addition, the authors used large language models, including GPT and Claude, to polish and edit the writing.
The authors verified all mathematical content themselves, wrote the introduction and the complete proof in their present form, and take full responsibility for the content of this paper.

\section{The algorithm and theoretical guarantee}\label{s:algorithm-construction}

Algorithm~\ref{alg:rgg-matching} in \S\ref{ss:formal-algorithm} gives the complete procedure, which has four steps.
\begin{itemize}
    \item \emph{Preprocessing} (\S\ref{ss:preprocessing}). In the intermediate regime $(\log n)^L \le d \le n/(\log n)^{L_0}$, compute spectral embeddings $f_A$ and $f_B$ that approximate the latent geometry of the two graphs. For a sufficiently large constant $K$, enumerate all ordered tuples $(s_1,\ldots,s_K)$ of distinct target vertices as guesses for the partners of the source seeds $1,\ldots,K$. This exhaustive list contains the true seed tuple.

    \item \emph{Signatures} (\S\ref{ss:high-signature-construction} and \S\ref{ss:low-signature-construction}). For each seed guess, compute signatures $\Phi^A(i)$ and $\Phi^B(j)$ indexed by words in the seed labels. In the intermediate regime they are built from the spectral embeddings; in the high-dimensional regime $d>n/(\log n)^{L_0}$, they are built from the observed graphs.

    \item \emph{Candidate matching and color coding} (\S\ref{ss:candidate-matching}). For each seed guess, retain the pairs $(i,j)$ whose score $\langle\Phi^A(i),\Phi^B(j)\rangle$ is at least a threshold, and complete a maximal matching of the retained pairs to a candidate permutation. To make this step polynomial time, the signatures are approximated by color coding. When all $K$ seed pairs are correct, the candidate has at most $6cn+o(n)$ errors with probability $1-o(1)$. The enumeration therefore produces at most $n^K$ candidates, one of which is nearly correct, although it is not known which one.

    \item \emph{Selection and refinement} (\S\ref{ss:seeded-graph-matching}). Building on \cite{ding2025polynomial, ding2025polynomial2}, select a candidate with few errors and correct its remaining errors, which recovers $\pi$ exactly.
\end{itemize}


\subsection{Preprocessing}\label{ss:preprocessing}

Let $K$ be a fixed positive integer and let $S_A=(1,\ldots,K)$. The value of $K$ is chosen for each regime in \S\ref{ss:high-signature-construction} and \S\ref{ss:low-signature-construction}. In set operations, a seed tuple denotes the set of its entries.
We first describe the spectral embeddings used in the intermediate regime $(\log n)^L\le d\le n/(\log n)^{L_0}$. Fix a constant $b\in(0,L_0)$ and let $m=\lfloor n/(\log n)^b\rfloor$; the constants $b$, $L$ and $L_0-b$ are chosen sufficiently large in the analysis. Independently of the observed graphs, sample disjoint uniform sets $T_A,D_A\subset[n]\setminus S_A$, each of size $m$, and independently sample disjoint uniform sets $T_B,D_B\subset[n]$, also of size $m$. We call $T_A,T_B$ the \emph{reference sets} and $D_A,D_B$ the \emph{auxiliary sets}. These four sets are sampled once and reused for every seed guess. Since $\pi$ is unknown, the sets for $B$ are sampled independently of those for $A$.

Fix $G\in\{A,B\}$ and order the labels in $T_G$ increasingly. Let $H_G$ be the $m\times m$ matrix indexed by $T_G$, with $(H_G)_{tt'}=G_{tt'}-p$ for $t\ne t'$ and $(H_G)_{tt}=0$. With $\lambda$ from \eqref{eq:eigen-lambda}, take the $d$ largest eigenvalues of $H_G/\lambda$, collect them in a diagonal matrix $\Lambda_G$, and let $Q_G\in\bbR^{m\times d}$ contain corresponding orthonormal eigenvectors. We declare failure unless all selected eigenvalues belong to $[m/4,4m]$ and every remaining eigenvalue is at most $m/8$. When these checks pass, define the \emph{spectral embedding} $f_G$ by
\begin{equation}\label{eq:reference-map}
    L_G=Q_G\Lambda_G^{-1/2},
    \quad
    f_G(i)=\lambda^{-1}L_G^\top
    \big(G_{i,T_G}-p\mathbf1_m\big)\in\bbR^d,
    \quad i\notin T_G,
\end{equation}
where $G_{i,T_G}$ is the column vector $(G_{it})_{t\in T_G}$ and $\mathbf1_m\in\bbR^m$ is the vector with all entries equal to one.

List all ordered tuples of $K$ distinct labels in $[n]$ as $\mathtt S_1,\ldots,\mathtt S_{\mathtt M}$, where $\mathtt M=n(n-1)\cdots(n-K+1)$. This list contains the true tuple $(\pi(1),\ldots,\pi(K))$. We apply the construction to each tuple and write $S_B=\mathtt S_{\mathtt m}=(s_1,\ldots,s_K)$ for a fixed guess. We skip $S_B$ if any of its entries belongs to $T_B\cup D_B$, and call the remaining tuples \emph{admissible}. The probability that the true tuple is skipped is at most
\begin{equation}{\label{eq:admissible-prob}}
    \P\big[ \{\pi(1),\ldots,\pi(K)\}\cap(T_B\cup D_B) \ne\emptyset \big] \le \frac{2Km}{n}=o(1).
\end{equation}
All subsequent reservoirs and query sets in graph $G\in\{A,B\}$ exclude $T_G\cup D_G\cup S_G$.
For an admissible tuple $S_B=(s_1,\ldots,s_K)$, the proposed seed pair with label $a\in[K]$ is $(a,s_a)$. Define its two \emph{estimated seed directions} by
\begin{equation}\label{eq:smoothed-seeds}
    v_a^A=\frac{1}{m\lambda\sqrt d}
          \sum_{w\in D_A}f_A(w)(A_{a,w}-p),
    \quad
    v_a^B=\frac{1}{m\lambda\sqrt d}
          \sum_{w\in D_B}f_B(w)(B_{s_a,w}-p).
\end{equation}

\subsection{Signatures in the high-dimensional regime}
\label{ss:high-signature-construction}

Throughout this subsection, we assume $d>n/(\log n)^{L_0}$. The signatures are built from the centered adjacency entries, and the spectral preprocessing is skipped: let $T_A=D_A=T_B=D_B=\varnothing$, and compute no spectral embeddings or estimated seed directions. Every tuple of $K$ distinct target vertices is then admissible.

For a fixed tuple $S_B=(s_1,\ldots,s_K)$, write $s_a^A=a$ and $s_a^B=s_a$ for the source and the proposed target seed with label $a\in[K]$. For $G\in\{A,B\}$, define $Z^G_{uv}=(G_{uv}-p)/\sqrt{p(1-p)}$ when $u\ne v$, and let $Z^G_{uu}=0$. Each entry off the diagonal has mean zero and variance one. To choose the parameters, let $(X,Y)$ be a jointly Gaussian pair with standard normal marginals and correlation $\rho^2$, and choose $z_p$ so that $\P[X\ge z_p]=p$. Define
\begin{equation}\label{eq:r-p-rho}
    r=r(p,\rho)
    =\frac{\P[X\ge z_p,Y\ge z_p]-p^2}{p(1-p)},
    \quad
    r_{\min}=\inf_{q\in[p_0,1-p_0]}r(q,\rho)>0.
\end{equation}
Fix $c\in(0,1/10)$, let $\beta=(c/2)r_{\min}^2$, and choose an integer $K$ and a constant $C_0>0$ so that
\[
    K\beta^2>1,
    \quad
    C_0\log(K\beta^2)>2L_0+4.
\]
The constant $L_0$ is fixed before $K$ and $C_0$ are chosen here (Remark~\ref{rmk:parameter-order}). Define $N,V,\mathcal W,J$ by \eqref{eq:path-parameters} with these values of $K$ and $C_0$. The constants $K,C_0$ may differ between the two regimes, and the seed enumeration uses the value of $K$ for the selected regime. For each seed guess, independently sample $R_A$ and $R_B$ as uniform $N$-subsets of $[n]\setminus S_A$ and $[n]\setminus S_B$, respectively, and define $I_G=[n]\setminus(S_G\cup R_G)$ for $G\in\{A,B\}$.

A word $w=(w_1,\ldots,w_{V-1})\in\mathcal W$ specifies the seed attached to each nonterminal reservoir vertex of a path. For $r\in R_G$, define the path coefficient
\begin{equation}\label{eq:high-path-coefficient}
    v_w^G(r)
    =\frac{1}{\sqrt{(N)_V}}
      \sum_{\substack{r_1,\ldots,r_V\in R_G\\
                      \text{all distinct}\\r_1=r}}
      \prod_{t=1}^{V-1}
      Z^G_{r_tr_{t+1}}Z^G_{s^G_{w_t}r_t}.
\end{equation}
The coefficient vector $v_w^G=(v_w^G(r))_{r\in R_G}$ is indexed by the $N$ reservoir vertices; the estimated seed directions $v_a^G\in\bbR^d$, indexed by seed labels, appear only in the intermediate regime. For each query $i\in I_G$, define
\begin{equation}\label{eq:high-exact-signature}
    \Phi_w^G(i)
    =\sum_{r\in R_G}Z^G_{ir}v_w^G(r)
    =\frac{1}{\sqrt{(N)_V}}
    \sum_{\substack{r_1,\ldots,r_V\in R_G\\\text{all distinct}}}
    Z^G_{ir_1}
    \prod_{t=1}^{V-1}Z^G_{r_tr_{t+1}}Z^G_{s^G_{w_t}r_t}.
\end{equation}
The \emph{exact signature} is $\Phi^G(i)=(\Phi_w^G(i))_{w\in\mathcal W}\in\bbR^J$, with the words listed in the same order in both graphs.

\subsection{Signatures in the intermediate regime}
\label{ss:low-signature-construction}

Throughout this subsection, $(\log n)^L\le d\le n/(\log n)^{L_0}$, and $S_B=(s_1,\ldots,s_K)$ is a fixed admissible seed guess.

Fix $c\in(0,1/10)$ and choose an integer $K$ and a constant $C_0>0$ such that
\[
    K(c/2)^2\rho^8>1,
    \quad
    C_0\log(K(c/2)^2\rho^8)>4.
\]
Let
\begin{equation}\label{eq:path-parameters}
    N=\lfloor cn\rfloor,
    \quad
    V=\lceil C_0\log\log n\rceil,
    \quad
    \mathcal W=[K]^{V-1},
    \quad
    J=|\mathcal W|=K^{V-1}.
\end{equation}
For $G\in\{A,B\}$, let $R_G$ be a uniform $N$-subset of $[n]\setminus(T_G\cup D_G\cup S_G)$, and define $I_G=[n]\setminus(T_G\cup D_G\cup S_G\cup R_G)$. We call $R_G$ the \emph{reservoir} and $I_G$ the \emph{query set}. The reservoirs are sampled independently between graphs, with fresh randomness for each seed guess, and independently of the graph data. A word $w=(w_1,\ldots,w_{V-1})\in\mathcal W$ assigns seed label $w_t$ to level $t$ of a path through $V$ reservoir vertices.

For $z\in\bbR^d$, define the symmetric trace-free cubic tensor $H_3(z)$ by
\[
    [H_3(z)]_{rst}
    =z_rz_sz_t-\frac{\|z\|^2}{d+2}
    (z_r\delta_{st}+z_s\delta_{rt}+z_t\delta_{rs}),
    \quad r,s,t\in[d],
\]
where $\delta_{rs}=\mathbf1\{r=s\}$. Contracting its third index with $v\in\bbR^d$ gives the symmetric matrix
\begin{equation}\label{eq:low-cubic-matrix}
    H_3(z)[v]
    =(z^\top v)zz^\top-\frac{\|z\|^2}{d+2}
       \big(zv^\top+vz^\top+(z^\top v)I_d\big),
\end{equation}
where $I_d$ is the $d\times d$ identity matrix. The trace of this matrix is zero for every $z,v$.

Let $(N)_V=N(N-1)\cdots(N-V+1)$ denote the number of ordered tuples of $V$ distinct vertices in $R_G$. For $w\in\mathcal W$, define a vector $b_w^G\in\bbR^d$ by
\begin{equation}\label{eq:low-path-direction}
    b_w^G=
    \frac{1}{\sqrt{(N)_Vd^{V-1}}}
    \sum_{\substack{r_1,\ldots,r_V\in R_G\\\text{all distinct}}}
    H_3(f_G(r_1))[v_{w_1}^G]\cdots
    H_3(f_G(r_{V-1}))[v_{w_{V-1}}^G]f_G(r_V).
\end{equation}
For $i\in I_G$, define
\begin{equation}\label{eq:low-exact-signature}
    \Phi_w^G(i)
    =\frac{1}{\sqrt d}f_G(i)^\top b_w^G
    =\frac{1}{\sqrt{(N)_Vd^V}}
    \sum_{\substack{r_1,\ldots,r_V\in R_G\\\text{all distinct}}}
    f_G(i)^\top
    \Big(\prod_{t=1}^{V-1}H_3(f_G(r_t))[v_{w_t}^G]\Big)
    f_G(r_V).
\end{equation}
The matrix factors are ordered by increasing $t$ from left to right. The \emph{exact signature} of $i$ is the vector $\Phi^G(i)=(\Phi_w^G(i))_{w\in\mathcal W}\in\bbR^J$, with the words listed in the same order in both graphs; \S\ref{s:color-coding} approximates it by a color-coded signature.

\subsection{Candidate matching and color coding}
\label{ss:candidate-matching}

For each seed guess, use the exact signatures in \eqref{eq:low-exact-signature} or \eqref{eq:high-exact-signature}, with $K,V,J$ chosen for the corresponding regime. For $(i,j)\in I_A\times I_B$, define
\begin{equation}{\label{eq:def-score}}
    S(i,j)=\langle\Phi^A(i),\Phi^B(j)\rangle =\sum_{w\in\mathcal W}\Phi_w^A(i)\Phi_w^B(j).
\end{equation}
Let
\begin{equation}\label{eq:def-theta-n}
    \theta_n=Jt_n/2,
    \quad
    t_n=
    \begin{cases}
        (c/2)^V\rho^{4V-2},&d\le n/(\log n)^{L_0},\\
        \beta^V,&d>n/(\log n)^{L_0}.
    \end{cases}
\end{equation}
The \emph{comparison graph} is the bipartite graph with vertex sets $I_A$ and $I_B$ and an edge $(i,j)$ whenever $S(i,j)\ge\theta_n$. Process the source labels $i$ in increasing order and, for each $i$, its neighbors $j$ in increasing order; select $(i,j)$ if neither endpoint is already matched. Complete this maximal matching to a permutation of $[n]$ by pairing all unused source and target labels in increasing order, and denote the resulting candidate by $\widehat\pi_{S_A,S_B}$. The implementation by color coding uses approximate signatures and the same threshold. \S\ref{s:color-coding} gives its construction, and Proposition~\ref{prop:color-coded-candidate} gives its accuracy.

The construction applies to every admissible seed guess. The next proposition gives an accuracy guarantee for the true seed guess $S_A=(1,\ldots,K)$ and $S_B=(\pi(1),\ldots,\pi(K))$. Probability is over the graph randomness and all reference, auxiliary, and reservoir sampling.

\begin{ppn}[Accuracy from true seeds]\label{prop:true-seed-candidate}
Fix $p_0\in(0,1/2]$, $\rho\in(0,1]$, and $c\in(0,1/10)$. There are constants $b,L,L_0$, with $0<b<L_0$, and parameters $K,C_0$ for each regime, depending only on $p_0,\rho,c$ and satisfying the requirements of the construction, for which the following holds. Uniformly over $p\in[p_0,1-p_0]$, $d\ge(\log n)^L$, and hidden permutations $\pi$, with probability $1-o(1)$, the candidate of the true seed guess is constructed and satisfies
\[
    \big|\{i\in[n]:\widehat\pi_{S_A,S_B}(i)\ne\pi(i)\}\big|
    \le 6cn+o(n).
\]
\end{ppn}

Computing the signatures in \eqref{eq:low-exact-signature} and \eqref{eq:high-exact-signature} by direct enumeration takes time $n^{O(V)}$, which is superpolynomial when $V=\Theta(\log\log n)$. We use color coding \cite{alon1995color,hopkins2017efficient,mao2024testing,mao2023random,li2025smooth} to compute approximate signatures $\widetilde\Phi^G(i)$. \S\ref{s:color-coding} gives the construction and proves the following guarantee, in which probability also includes the independent colorings.

\begin{ppn}\label{prop:true-seed-candidate-color-coding}
With a suitable polylogarithmic number of independent colorings, Proposition~\ref{prop:true-seed-candidate} continues to hold with $\widetilde\Phi^G(i)$ in place of $\Phi^G(i)$. Given the preprocessing outputs, all these approximate signatures can be computed in time $O(n^{2+o(1)})$ for each fixed seed guess.
\end{ppn}

Enumerating all admissible seed guesses thus produces at most $n^K$ candidate permutations, and by Propositions~\ref{prop:true-seed-candidate} and \ref{prop:true-seed-candidate-color-coding}, with probability $1-o(1)$, one of them has at most $6cn+o(n)$ errors. It remains to select an accurate candidate and to correct its errors.

\subsection{Selection and refinement}
\label{ss:seeded-graph-matching}

We select a candidate by comparing the overlap of the centered adjacency matrices under each proposed permutation. We then use the selected candidate to align and compare neighborhoods, which corrects the remaining errors. These steps depend only on $(A,B,p)$, are the same in both regimes, and use all vertices, including the seed, reference, auxiliary, and reservoir vertices. Selection costs $O(n^2)$ arithmetic operations per candidate, and refinement costs $O(n^3)$. \S\ref{s:seeded-graph-matching} gives the procedure and its analysis.

\begin{prop}\label{prop:rounding-finishing}
Fix $p_0\in(0,1/2]$, $\rho\in(0,1]$, and $K>0$. There exists a constant $\varepsilon=\varepsilon(p_0,\rho)>0$ for which the following holds uniformly over $p\in[p_0,1-p_0]$, $d\ge(\log n)^{14}$, and hidden permutations $\pi$. With probability $1-o(1)$, simultaneously for every nonempty list of at most $n^K$ candidate permutations, the selection and refinement procedure returns $\pi$ whenever the list contains a candidate with at most $\varepsilon n$ errors.
\end{prop}

Since the guarantee holds simultaneously for all lists, the candidate list may depend on the observed graphs and on the colorings. Choose $0<c<\min\{1/10,\varepsilon/6\}$ and take $L\ge14$. Then $6cn+o(n)\le\varepsilon n$ for all sufficiently large $n$, so Propositions~\ref{prop:true-seed-candidate-color-coding} and \ref{prop:rounding-finishing} give exact recovery with probability $1-o(1)$. \S\ref{ss:formal-algorithm} gives the complete algorithm and its running time, and Corollary~\ref{cor:color-coded-recovery} completes the proof of Theorem~\ref{thm:main-alg}.

\Needspace{10\baselineskip}
\subsection{Formal description and running time}\label{ss:formal-algorithm}

Algorithm~\ref{alg:rgg-matching} collects the steps of \S\ref{ss:preprocessing}--\S\ref{ss:seeded-graph-matching}.

\begin{breakablealgorithm}\label{alg:rgg-matching}
\caption{Matching correlated random geometric graphs}
    \begin{algorithmic}[1]
    \State \textbf{Input:} $(n,d,p,\rho,A,B)$, with $d\ge(\log n)^L$.
    \State Choose $c,b,L,L_0,K,C_0,C_{\mathrm{col}}$ for the selected regime as in Corollary~\ref{cor:color-coded-recovery}.
    \State Set $N\gets\lfloor cn\rfloor$, $V\gets\lceil C_0\log\log n\rceil$, $\mathcal W\gets[K]^{V-1}$, and $\theta_n$ as in \eqref{eq:def-theta-n}.
    \State $S_A\gets(1,\ldots,K)$; $\mathcal C\gets\varnothing$.
    \Statex \textit{Preprocessing.}
    \If{$d\le n/(\log n)^{L_0}$}
        \State Set $m\gets\lfloor n/(\log n)^b\rfloor$ and sample $T_A,D_A,T_B,D_B$ as in \S\ref{ss:preprocessing}.
        \State Apply the spectral checks in \S\ref{ss:preprocessing}; \textbf{if} any check fails, \Return $\mathrm{id}$.
        \State Construct the spectral embeddings $f_A,f_B$ using \eqref{eq:reference-map}.
    \Else
        \State Set $T_A=D_A=T_B=D_B=\varnothing$ and construct $Z^A,Z^B$ as in \S\ref{ss:high-signature-construction}.
    \EndIf
    \Statex \textit{Seed enumeration, color coding, and candidate construction.}
    \ForAll{ordered tuples $S_B=(s_1,\ldots,s_K)$ of distinct labels in $[n]\setminus(T_B\cup D_B)$}
        \State For each $G\in\{A,B\}$, sample a uniform $N$-subset $R_G\subset[n]\setminus(T_G\cup D_G\cup S_G)$.
        \State Set $I_G\gets[n]\setminus(T_G\cup D_G\cup S_G\cup R_G)$ for $G\in\{A,B\}$.
        \If{$d\le n/(\log n)^{L_0}$}
            \State Compute the estimated seed directions $v_a^A,v_a^B$ for all $a\in[K]$ using \eqref{eq:smoothed-seeds}.
            \State Compute $\{ \widetilde{\Phi}^G(i): G \in \{ A,B \}, i \in I_G \}$ using Algorithm~\ref{alg:color-coordinate}, once for each graph.
        \Else
            \State Compute $\{ \widetilde{\Phi}^G(i): G \in \{ A,B \}, i \in I_G \}$ using Algorithm~\ref{alg:color-graph}, once for each graph.
        \EndIf
        \State Compute $\widetilde S(i,j)\gets\sum_{w\in\mathcal W} \widetilde\Phi_w^A(i)\widetilde\Phi_w^B(j)$ for all $(i,j)\in I_A\times I_B$.
        \State Starting from the empty matching, scan pairs with $\widetilde S(i,j)\ge\theta_n$ in lexicographic order, retaining a pair if both endpoints are unmatched.
        \State Complete the matching to $\widetilde\pi_{S_A,S_B} \in\mathfrak S_n$ by pairing all unused source and target labels in increasing order.
        \State Append $\widetilde\pi_{S_A,S_B}$ to $\mathcal C$.
    \EndFor
    \State Remove duplicate candidates from $\mathcal C$ by lexicographic sorting.
    \State \textbf{if} $\mathcal C=\varnothing$, \Return $\mathrm{id}$.
    \State Apply Algorithm~\ref{alg:seeded-graph-matching} to $(A,B,p,\mathcal C)$ and denote its output by $\widehat\pi$.
    \State \textbf{Output:} $\widehat\pi$.
    \end{algorithmic}
\end{breakablealgorithm}

The next proposition shows that Algorithm~\ref{alg:rgg-matching} runs in polynomial time.

\begin{ppn}\label{prop:running-time}
Each candidate $\widetilde{\pi}_{S_A,S_B}$ can be computed in time $O(n^{2+o(1)})$ once the common preprocessing is available. For a single supplied seed guess, the complete procedure, including preprocessing and final refinement, takes time $O(n^{3+o(1)})$. With enumeration of all admissible seed guesses, Algorithm~\ref{alg:rgg-matching} takes time $O(n^{K+2+o(1)}+n^3)$.
\end{ppn}

\begin{proof}
We count arithmetic operations on real numbers, treat the evaluation of $\tau_{p,d}$ and $\lambda$ as single operations, and count a symmetric eigendecomposition of a matrix of size at most $n$ as $O(n^3)$ operations. Constructing $f_A,f_B$ takes at most $O(n^3)$ arithmetic operations by dense eigendecomposition and matrix multiplication, and is done only once; this step is absent in the high-dimensional regime. For each fixed $(S_A,S_B)$, computing the $K$ pairs of estimated seed directions from \eqref{eq:smoothed-seeds} takes $O(Kmd)=O(n^2)$ operations. By \S\ref{s:color-coding}, the dynamic programs with $M_{\mathrm{col}}=(\log n)^{O(1)}$ colorings compute all signatures in time $O(n^{2+o(1)})$. Computing the $n^2$ comparison scores costs $O(n^2J)=O(n^{2+o(1)})$, and the lexicographic maximal matching and its completion cost $O(n^2)$. This proves the bound $O(n^{2+o(1)})$ for each candidate after preprocessing. For a single seed guess, adding the common preprocessing and the final refinement, which costs $O(n^3)$, gives $O(n^{3+o(1)})$ operations.

There are at most $n^K$ seed guesses. Removing duplicate candidates by lexicographic sorting costs at most $O(n^{K+1}\log n)$ operations, and the selection and refinement in Algorithm~\ref{alg:seeded-graph-matching} cost $O(n^2|\mathcal C|+n^3)$. The common preprocessing is computed once and the refinement is applied only to the selected candidate, so neither cost is multiplied by the number of seed guesses. The total cost is therefore
\[
    O\big(n^{K+2+o(1)}+n^3\big),
\]
as claimed.
\end{proof}

\section{Preliminaries for the analysis}\label{s:preliminaries}\label{ss:high-preliminary-inputs}

We first reduce approximate matching to score separation, then record the spherical estimates used in \S\ref{s:high-dimension-proof} and~\S\ref{s:intermediate-dimension-proof}.

\subsection{Reduction to score separation}\label{ss:reduction}

After relabeling graph $B$, assume that $\pi=\mathrm{id}$ and write $q_{\mathrm{id}}(\tau)=|\{i:\tau(i)\ne i\}|$. Recall that the comparison graph on $I_A\times I_B$ contains $(i,j)$ when $S(i,j)\ge\theta_n$. Its edges $(i,i)$ are called true; all other edges are false. The following counting observation is similar to \cite[Lemmas~4.1 and~4.2]{ding2025polynomial2}.

\begin{lem}[Threshold matching]\label{lem:threshold-matching}
Let $\mathcal G\subseteq I_A\cap I_B$. Suppose that no false edge joins two labels of $\mathcal G$, and at most $h$ labels in $\mathcal G$ miss their true edge. Every maximal matching in the comparison graph, completed arbitrarily to a permutation $\tau$ of $[n]$, satisfies
\[
    q_{\mathrm{id}}(\tau)\le3|[n]\setminus\mathcal G|+h.
\]
\end{lem}

\begin{proof}
Put $e=|[n]\setminus\mathcal G|$. The $e$ excluded labels have at most $2e$ vertices in the bipartite graph, so matching edges incident to them meet at most $2e$ labels of $\mathcal G$. Every other label of $\mathcal G$ whose true edge is present must be matched correctly: its endpoints cannot be matched elsewhere, and leaving both unmatched would contradict maximality. Thus only the $e$ excluded labels, these at most $2e$ blocked labels, and the at most $h$ missing true edges can cause errors.
\end{proof}

Because this conclusion holds for every maximal matching and completion, relabeling $B$ is legitimate even though it changes the lexicographic scan used by the algorithm.

\begin{rmk}\label{rmk:choice-cI}
For the true seeds, relabeled as $S_A=S_B=[K]$, we use
\[
    \mathcal I=[n]\setminus(T_A\cup D_A\cup S_A\cup R_A\cup T_B\cup D_B\cup S_B\cup R_B).
\]
The reference and auxiliary sets are empty in the high-dimensional regime. In both regimes $\mathcal I\subseteq I_A\cap I_B$, and
\begin{equation}\label{eq:physical-query-exclusions}
    |[n]\setminus\mathcal I|\le
    \begin{cases}
        2N+4m+K=2cn+o(n),&d\le n/(\log n)^{L_0},\\
        2N+K=2cn+O(1),&d>n/(\log n)^{L_0}.
    \end{cases}
\end{equation}
\end{rmk}

\subsection{Spherical harmonics and the cap kernel}\label{ss:spherical-geometries}

Recall $\mu$, $x(u)=\sqrt d\,u$, $h_0$ and $\lambda$ from \S\ref{ss:notation}. Define
\[
    (\mathcal Tf)(u)=\int h_0(u,v)f(v)\mu(\de v),
    \qquad
    (\mathcal M_{u_*}f)(u)=h_0(u_*,u)f(u).
\]
Let $\Pi_\ell$ project onto the spherical harmonics of degree $\ell$, namely the restrictions of homogeneous harmonic polynomials of that degree, and let
\[
    D_\ell=\binom{d+\ell-1}{\ell}-\binom{d+\ell-3}{\ell-2},
\]
with the second term zero for $\ell<2$. For $d\ge3$, $D_\ell$ increases with $\ell$ and $D_\ell\asymp_\ell d^\ell$. The Funk--Hecke and addition formulas \cite[Theorem~2.9 and Corollary~2.7]{dai2013spherical} give
\begin{equation}\label{eq:harmonic-energy}
    \mathcal T=\sum_{\ell\ge1}\lambda_\ell\Pi_\ell,
    \quad
    \sum_{\ell\ge1}D_\ell\lambda_\ell^2=p(1-p),
    \quad
    |\lambda_\ell|\le\sqrt{\frac{p(1-p)}{D_\ell}}.
\end{equation}
Here centering removes degree zero, $\lambda_1=\lambda$, and the kernel of $\Pi_\ell$ has diagonal $D_\ell$.

\begin{lem}[Hypercontractivity and cap smoothing]\label{lem:cap-l2-lq}
For every fixed $q\ge2$, a spherical harmonic $f$ of degree $\ell$ satisfies $\|f\|_q\le(q-1)^{\ell/2}\|f\|_2$. Moreover, uniformly over $d\ge3$,
\[
    \|\mathcal T\|_{L^2(\mu)\to L^q(\mu)}\le C_qd^{-1/2},
    \qquad
    \|\mathcal T-\lambda\Pi_1\|_{L^2(\mu)\to L^q(\mu)}\le C_qd^{-1}.
\]
\end{lem}

\begin{proof}
The first assertion is \cite[Corollary~8]{montanaro2012hypercontractive}. Set $\ell_0=\lceil q\rceil$. For $j=1,2$, this assertion and \eqref{eq:harmonic-energy} imply
\[
    \Big\|\sum_{\ell=j}^{\ell_0-1}\lambda_\ell\Pi_\ell\Big\|_{2\to q}
    \le\sum_{\ell=j}^{\ell_0-1}(q-1)^{\ell/2}|\lambda_\ell|
    \le C_qd^{-j/2}.
\]
For $\mathcal R_{\ge\ell_0}=\sum_{\ell\ge\ell_0}\lambda_\ell\Pi_\ell$, monotonicity of $D_\ell$ and the addition formula give $\|\mathcal R_{\ge\ell_0}\|_{2\to2}\le C_qd^{-\ell_0/2}$ and $\|\mathcal R_{\ge\ell_0}\|_{2\to\infty}\le1$. Interpolating these two bounds gives $\|\mathcal R_{\ge\ell_0}\|_{2\to q}\le C_qd^{-\ell_0/q}\le C_qd^{-1}$, which proves both smoothing bounds.
\end{proof}

Integrating consecutive reservoir vertices of a seeded path produces the alternating kernel below.

\begin{lem}[Integrated seeded paths]\label{lem:chain-lp}
Let $k\ge1$ and $u_1,\ldots,u_k\in\mathbb S^{d-1}$, with repetitions allowed. For every fixed $q\ge2$, the integral kernel of
\[
    \mathcal A_{u_1,\ldots,u_k}
    =\mathcal T\mathcal M_{u_1}\mathcal T\mathcal M_{u_2}\cdots\mathcal T\mathcal M_{u_k}\mathcal T
\]
satisfies, uniformly in $k$ and the seed directions,
\[
    \|\mathcal A_{u_1,\ldots,u_k}\|_{L^q(\mu\otimes\mu)}\le C_q^kd^{-(k+1)/2}.
\]
\end{lem}

\begin{proof}
Write $\mathcal T_1=\lambda\Pi_1$ and $\mathcal R=\mathcal T-\mathcal T_1$. For $k=1$,
\[
    \mathcal T\mathcal M_{u_1}\mathcal T
    =\mathcal R\mathcal M_{u_1}\mathcal T
      +\mathcal T_1\mathcal M_{u_1}\mathcal R
      +\mathcal T_1\mathcal M_{u_1}\mathcal T_1.
\]
The first two kernels have $L^q(\mu\otimes\mu)$ norm at most $C_q/d$ by Lemma~\ref{lem:cap-l2-lq}: their sections input to $\mathcal R$ have $L^2$ norm at most one, since $|h_0|\le1$ and $|\lambda|\sqrt d\le1$.
For the last kernel, $\Pi_1(u,v)=x(u)^\top x(v)$ gives
\[
    (\mathcal T_1\mathcal M_{u_1}\mathcal T_1)(u,v)
    =\lambda^2x(u)^\top\Gamma_{u_1}x(v),
    \qquad
    \Gamma_{u_1}=\E[x(w)x(w)^\top h_0(u_1,w)].
\]
Rotational invariance and $\operatorname{tr}\Gamma_{u_1}=0$ imply
\[
    \Gamma_{u_1}=a u_1u_1^\top-\frac a{d-1}(I_d-u_1u_1^\top),
    \qquad |a|\le1,
\]
where the bound uses $|h_0|\le1$ and $\E[\langle x(w),u_1\rangle^2]=1$. Thus $\|\Gamma_{u_1}\|_{\mathrm F}\le2$. Hypercontractivity in each variable bounds $\|x(u)^\top\Gamma_{u_1}x(v)\|_q$ by $C_q\|\Gamma_{u_1}\|_{\mathrm F}$, and $\lambda^2\le d^{-1}$ proves the case $k=1$.

For $k\ge2$, fix $v$ and apply Lemma~\ref{lem:cap-l2-lq} to
$\mathcal A_{u_1,\ldots,u_k}(\cdot,v)=\mathcal T\mathcal M_{u_1}\mathcal A_{u_2,\ldots,u_k}(\cdot,v)$:
\[
    \|\mathcal A_{u_1,\ldots,u_k}(\cdot,v)\|_q
    \le C_qd^{-1/2}\|\mathcal A_{u_2,\ldots,u_k}(\cdot,v)\|_2
    \le C_qd^{-1/2}\|\mathcal A_{u_2,\ldots,u_k}(\cdot,v)\|_q.
\]
Taking the $L^q$ norm in $v$ completes the induction.
\end{proof}

We also use the following direct consequence of Finner's generalized H\"older inequality \cite{finner1992generalization}.

\begin{lem}[Products of kernels]\label{lem:join-chains}
Let $\mathsf M$ be a loopless multigraph of maximum degree at most eight, whose vertices carry independent latent pairs $(u_i^A,u_i^B)$. The two vectors in a pair may be dependent. For each edge $e=\{i,j\}$, let $B_e$ be a kernel evaluated at one selected vector at $i$ and one at $j$. If every selected vector has law $\mu$, then
\[
    \E\Big[\prod_{e\in E(\mathsf M)}|B_e|\Big]
    \le\prod_{e\in E(\mathsf M)}\|B_e\|_{L^{16}(\mu\otimes\mu)}.
\]
\end{lem}

Indeed, assigning weight $1/16$ to each edge satisfies the fractional matching condition in Finner's inequality on the product space of latent pairs.
\section{Analysis in the high-dimensional regime}\label{s:high-dimension-proof}

We prove Proposition~\ref{prop:true-seed-candidate} for $d>n/(\log n)^{L_0}$. Every seed tuple is admissible in this regime. Relabeling $B$, assume that $\pi=\mathrm{id}$ and that the true seeds are $S_A=S_B=[K]$. Throughout \S\ref{ss:er-comparison}--\S\ref{ss:direct-false-comparisons}, all label sets are fixed; probabilities and moments average over the independent latent $\rho$-pairs, including the seeds, reservoirs and queries.

Use the signatures \eqref{eq:high-exact-signature}, the scores $S(i,j)=\langle\Phi^A(i),\Phi^B(j)\rangle$, and $\mathcal I=[n]\setminus([K]\cup R_A\cup R_B)$ from Remark~\ref{rmk:choice-cI}. Each signature summand is a path $(i,r_1,\ldots,r_V)$ with distinct reservoir vertices and seed $w_t$ attached at $r_t$ for $t<V$. By Lemma~\ref{lem:threshold-matching} and \eqref{eq:physical-query-exclusions}, it suffices to prove, with probability $1-o(1)$,
\begin{equation}\label{eq:high-sufficient}
    |\{i\in\mathcal I:S(i,i)<\theta_n\}|=o(n),
    \qquad
    \max_{\substack{i,j\in\mathcal I\\i\ne j}}S(i,j)<\theta_n,
    \qquad \theta_n=Jt_n/2,\quad t_n=\beta^V.
\end{equation}
For true pairs, second moments suffice because $o(n)$ misses are harmless. We compute these moments by comparison with correlated \ER graphs. For false pairs, conditioning on the seeds and reservoirs leaves independent query roots; a conditional bilinear tail bound then permits a union bound over all pairs.

\subsection{Comparison with correlated Erd\H{o}s--R\'enyi graphs}\label{ss:er-comparison}

\begin{dfn}[Correlated \ER graphs]\label{dfn:correlated-ER}
Let $(X_{ij},Y_{ij})$, indexed by unordered pairs $\{i,j\}\subseteq[n]$, be independent standard Gaussian pairs of correlation $\rho^2$. With $\P[X_{ij}\ge z_p]=p$, set
\[
    A_{ij}=\mathbf1\{X_{ij}\ge z_p\},
    \qquad B_{\pi(i)\pi(j)}=\mathbf1\{Y_{ij}\ge z_p\}.
\]
Write $\P_{\ERG}$ for this law and $\P_{\RGG}$ for the law in Definition~\ref{def:correlated-RGG}.
\end{dfn}

Under $\P_{\ERG}$, the standardized edge pairs are independent across unordered pairs, centered, and have unit variances and correlation $r=r(p,\rho)$ from \eqref{eq:r-p-rho}. Subscripts distinguish moments under the two laws; unsubscripted moments refer to $\P_{\RGG}$. For a sufficiently large fixed $C$, define
\begin{equation}\label{eq:direct-moment-error}
    \xi_n=(CV)^{CV}\Big(\frac1n+\frac1{\sqrt d}\max\Big(1,\frac nd\Big)^{2V}\Big).
\end{equation}
Since $V=O(\log\log n)$ and $n/d<(\log n)^{L_0}$, we have $\xi_n\le n^{-1/2+o(1)}$.

\begin{lem}[Local comparison with correlated \ER graphs]\label{lem:local-tv}
For $Q\subseteq[n]$ with $|Q|=q\le d/8$, let $\P_{\ERG,Q}$ and $\P_{\RGG,Q}$ be the restrictions to $(A_{ij},B_{\pi(i)\pi(j)}:i,j\in Q)$. Then
\begin{equation}\label{eq:local-tv}
    \TV(\P_{\ERG,Q},\P_{\RGG,Q})\le C\sqrt{\frac{q^3}{d}}+C_{p_0}\frac{q^2}{d}.
\end{equation}
\end{lem}

\begin{proof}
Assume $Q=[q]$. The spherical Gaussian approximation \cite[Theorem~2.8]{meckes2019random} gives $|\sqrt d\tau_{p,d}-z_p|\le C_{p_0}/d$: the one-dimensional distribution functions differ by $O(d^{-1})$, and the Gaussian density is bounded below near the quantiles for $p\in[p_0,1-p_0]$. Replacing $z_p$ by $\sqrt d\tau_{p,d}$ in the Gaussian model therefore costs at most $C_{p_0}q^2/d$ in total variation. Denote the resulting law by $\widetilde\P_{\ERG,Q}$.

We adapt the vertex-exposure entropy comparison of \cite[Section~2.1 and Claim~8.1]{liu2022testing} to the two graphs jointly. Write $u_i^A=U_ie_1$ and $u_i^B=U_ia$, where the $U_i\in\bbR^{d\times2}$ are independent uniform orthonormal pairs, $e_1=(1,0)^\top$ and $a=(\rho,\sqrt{1-\rho^2})^\top$. Replacing the blocks $\sqrt dU_i^\top U_j$, $i<j$, by independent standard Gaussian $2\times2$ matrices gives edge pairs of correlation $(e_1^\top a)^2=\rho^2$.

Let $\mathcal F_t=\sigma(U_1,\ldots,U_t)$, let $\mathcal N_{t+1}=(A_{i,t+1},B_{i,t+1}:i\le t)$, and let $\nu_{t+1}$ be its law under $\widetilde\P_{\ERG,Q}$. Writing $\mathcal D$ for relative entropy, its chain rule and conditional convexity give
\begin{equation}\label{eq:local-neighborhood-entropy}
    \mathcal D(\P_{\RGG,Q}\|\widetilde\P_{\ERG,Q})
    \le\sum_{t=1}^{q-1}\E\big[\mathcal D(\mathcal L(\mathcal N_{t+1}\mid\mathcal F_t)\|\nu_{t+1})\big].
\end{equation}
Write $k=2t$ and let $\Sigma_t$ be the Gram matrix of the columns of $U_1,\ldots,U_t$. Conditional on $\mathcal F_t$, the stacked blocks $\sqrt dU_i^\top U_{t+1}$ have law $\Sigma_t^{1/2}Z$, where $Z$ is $\sqrt d$ times the top $k\times2$ block of an independent uniform orthonormal pair. If $\gamma_{k\times2}$ denotes standard Gaussian matrix law and $\Delta_k=\mathcal D(\mathcal L(Z)\|\gamma_{k\times2})$, data processing and $\operatorname{tr}\Sigma_t=k$ imply
\[
    \mathcal D(\mathcal L(\mathcal N_{t+1}\mid\mathcal F_t)\|\nu_{t+1})
    \le\Delta_k-\log\det\Sigma_t.
\]
By the truncated Haar density \cite[Theorem~2.11]{meckes2019random},
\[
    f_{d,k}(z)=c_{d,k}\det(I_2-z^\top z/d)^{(d-k-3)/2}
                  \mathbf1\{z^\top z<dI_2\},
    \qquad c_{d,k}\le(2\pi)^{-k}.
\]
Since $\E\|Z\|_{\mathrm F}^2=2k$ and $\log\det(I-B)\le-\operatorname{tr}B$, this yields
\begin{equation}\label{eq:projected-frame-entropy}
    \Delta_k\le\E\Big[\frac{\|Z\|_{\mathrm F}^2}{2}
      +\frac{d-k-3}{2}\log\det(I_2-Z^\top Z/d)\Big]
    \le\frac{k(k+3)}d.
\end{equation}

For the Gram determinant, let $P_t$ project onto the columns already exposed. The block determinant identity gives $\det\Sigma_{t+1}/\det\Sigma_t=\det(I_2-U_{t+1}^\top P_tU_{t+1})$. The beta laws for successive projection lengths \cite[Lemma~2.12]{meckes2019random}, together with $-\log(1-x)\le x/(1-x)$, give
\[
    -\E\Big[\log\frac{\det\Sigma_{t+1}}{\det\Sigma_t}\,\Big|\mathcal F_t\Big]
    \le\frac{2t}{d-2t-2}+\frac{2t}{d-2t-3}
    \le\frac{Ct}d.
\]
Thus $-\E\log\det\Sigma_t\le Ct^2/d$. Substitution into \eqref{eq:local-neighborhood-entropy} gives relative entropy at most $Cq^3/d$. Pinsker's inequality and the threshold comparison prove \eqref{eq:local-tv}.
\end{proof}

For $w,z\in\mathcal W$, put $T_w=\Phi_w^A(i)\Phi_w^B(i)$ and $g_{wz}^G=\langle v_w^G,v_z^G\rangle$, where the coefficient vectors are defined in \eqref{eq:high-path-coefficient}.

\begin{lem}[Moment comparison]\label{lem:moment-compare-err}
Uniformly over the label sets, $w,z\in\mathcal W$ and $G\in\{A,B\}$, the following hold.
\begin{enumerate}[label=(\roman*)]
\item For $i\in\mathcal I$,
\[
    |\E_{\RGG}T_w-\E_{\ERG}T_w|\le\xi_n,
    \qquad |\E_{\RGG}[T_wT_z]-\E_{\ERG}[T_wT_z]|\le\xi_n.
\]
\item For the coefficient inner products,
\[
    |\E_{\RGG}g_{wz}^G-\E_{\ERG}g_{wz}^G|\le\xi_n,
    \qquad |\E_{\RGG}[(g_{wz}^G)^2]-\E_{\ERG}[(g_{wz}^G)^2]|\le\xi_n.
\]
\item For every query $i\in I_G$,
\[
    \E_{\ERG}[(\Phi_w^G(i))^2]=1,
    \qquad |\E_{\RGG}[(\Phi_w^G(i))^2]-1|\le\xi_n.
\]
\item For $i\in\mathcal I$ and all sufficiently large $n$,
\[
\begin{split}
    |\E_{\RGG}S(i,i)-\E_{\ERG}S(i,i)|&\le J\xi_n,\\
    |\operatorname{Var}_{\RGG}(S(i,i))-\operatorname{Var}_{\ERG}(S(i,i))|&\le4J^2\xi_n.
\end{split}
\]
\end{enumerate}
\end{lem}

\begin{proof}
For an injective reservoir tuple $\boldsymbol a=(a_1,\ldots,a_V)$, write
\[
    Q_w^G(\boldsymbol a)=\prod_{t<V}Z^G_{a_ta_{t+1}}Z^G_{s^G_{w_t}a_t},
    \qquad P_w^G(\boldsymbol a)=Z^G_{ia_1}Q_w^G(\boldsymbol a).
\]
Then $\Phi_w^G(i)=(N)_V^{-1/2}\sum_{\boldsymbol a}P_w^G(\boldsymbol a)$ and
\begin{equation}\label{eq:direct-coefficient-expansion}
    g_{wz}^G=\frac1{(N)_V}
    \sum_{\substack{\boldsymbol a,\boldsymbol b\in R_G^V\text{ injective}\\a_1=b_1}}
       Q_w^G(\boldsymbol a)Q_z^G(\boldsymbol b).
\end{equation}
Every moment in (i)--(iii) expands into $s\in\{2,4\}$ paths, normalized by $(N)_V^{-s/2}$. In coefficient moments, their first reservoir positions are identified in pairs; signature paths share the root $i$. Partition the $sV$ reservoir positions according to equality of their labels. There are at most $(sV)^{sV}$ such patterns, and a pattern with $u$ distinct reservoir vertices has at most $n^u$ labelings. Its product $\Psi$ has at most $s(2V-1)$ bounded factors, so $|\Psi|\le C^V$.

Suppose first that $k\ge1$ reservoir vertices occur only once. Under $\P_{\ERG}$ a centered edge at any such vertex forces $\E_{\ERG}\Psi=0$. Under $\P_{\RGG}$, a singleton terminal vertex also integrates to zero. Otherwise the singletons form $a\ge1$ maximal runs of lengths $k_1,\ldots,k_a$, with $\sum_bk_b=k$. Every run has two distinct boundaries: paths are injective, no run reaches a terminal vertex, and a coefficient path's first vertex is identified with another path's first vertex.

Condition on all seed pairs. Integrating a run of length $k_b$ gives a kernel $K_b$ from Lemma~\ref{lem:chain-lp}, multiplied by $[p(1-p)]^{-(2k_b+1)/2}$. Uniformly in the seed directions, including repetitions,
\[
    \|K_b\|_{L^{16}(\mu\otimes\mu)}\le C^{k_b}d^{-(k_b+1)/2}.
\]
The multigraph with one edge per run and vertices at its boundaries is loopless, with maximum degree at most $2s\le8$. Conditional on the seeds, the remaining latent pairs are independent. Bound the factors outside the runs by $C^V$ and apply Lemma~\ref{lem:join-chains} to obtain $|\E_{\RGG}\Psi|\le C^Vd^{-(k+a)/2}$. Since $u\le sV/2+k/2$ and $(N)_V\ge(cn/2)^V$, the normalized sum over singleton patterns is at most
\[
    (sV)^{sV}\max_{\substack{1\le k\le sV\\a\ge1}}
    \frac{C^Vn^{sV/2+k/2}}{(N)_V^{s/2}}d^{-(k+a)/2}
    \le(CV)^{CV}d^{-1/2}\max(1,n/d)^{2V}.
\]

A pattern without singletons has $u\le sV/2\le2V$. Including the seeds and query, it uses at most $K+2V+1$ vertices, so Lemma~\ref{lem:local-tv} bounds its expectation difference by $C^V[\sqrt{(K+2V+1)^3/d}+(K+2V+1)^2/d]$. Summing over patterns and labelings costs at most $(CV)^{CV}d^{-1/2}$. This proves the comparisons in (i)--(iii).

For the exact \ER identity in (iii), a nonzero product of two rooted paths requires every edge to occur in both paths. Their root edges therefore agree, and injectivity propagates agreement along the path. There are exactly $(N)_V$ coincident tuples, each contributing one. This argument uses only the reservoir and seeds of graph $G$, so it applies to every $i\in I_G$, including labels in the other graph's reservoir.

Finally, (iii) and Cauchy--Schwarz give $|\E_{\ERG}T_w|\le1$. Summing (i) gives differences at most $J\xi_n$ for the score means and $J^2\xi_n$ for their second moments. The difference of squared means is at most $J^2\xi_n(2+\xi_n)\le3J^2\xi_n$, proving (iv).
\end{proof}

\subsection{True pairs}\label{ss:direct-signature-analysis}

Let $H=|R_A\cap R_B|$ and $\mu_n=r^{2V-1}(H)_V/(N)_V$.

\begin{lem}[Scores of true pairs]\label{lem:direct-path-moments}
For all reservoir label sets and $i\in\mathcal I$,
\[
    |\E S(i,i)-J\mu_n|\le J\xi_n,
    \qquad \operatorname{Var}(S(i,i))\le2J+CJ^2\xi_n.
\]
\end{lem}

\begin{proof}
We compute under $\P_{\ERG}$ and transfer by Lemma~\ref{lem:moment-compare-err}(iv). In a nonzero term of $\E_{\ERG}T_w$, the two rooted paths coincide, as in the proof of that lemma, and their common reservoir tuple lies in $R_A\cap R_B$. Each of their $2V-1$ repeated edges contributes $r$, so
\[
    \E_{\ERG}T_w=r^{2V-1}\frac{(H)_V}{(N)_V}=\mu_n.
\]

Denote the four paths in $\E_{\ERG}[T_wT_z]$ by $A_w,B_w,A_z,B_z$. A nonzero pattern has at most $2V$ reservoir labels. With exactly $2V$, every label occurs twice: the root edge determines a path's partner, and edge repetition propagates this pairing to the terminal vertex. Thus the four paths form two coincident pairs with disjoint reservoir supports. Pairing different words also requires equality of their seed edges, hence $w=z$. The possible pairings are
\[
    (A_w,B_w)(A_z,B_z),\qquad
    (A_w,A_z)(B_w,B_z),\qquad
    (A_w,B_z)(B_w,A_z).
\]
Writing their normalized contributions as $L_1,L_2,L_3$, and all remaining contributions as $R_{wz}$, gives
\[
    \E_{\ERG}[T_wT_z]=L_1+L_2+L_3+R_{wz},
\]
where
\[
    L_1=r^{4V-2}\frac{(H)_{2V}}{(N)_V^2}\le\mu_n^2,
    \quad 0\le L_2\le\mathbf1\{w=z\},
    \quad L_3=\mathbf1\{w=z\}L_1\le\mathbf1\{w=z\}.
\]
Indeed, $L_1$ counts disjoint common paths, $L_2$ pairs paths within each graph, and $L_3$ pairs them across graphs. The remainder uses at most $2V-1$ reservoir labels, so the pattern count from Lemma~\ref{lem:moment-compare-err} gives $|R_{wz}|\le(CV)^{CV}/n$. Consequently,
\[
    \operatorname{Cov}_{\ERG}(T_w,T_z)
    \le2\mathbf1\{w=z\}+\frac{(CV)^{CV}}n.
\]
Summing over $w,z$ and applying Lemma~\ref{lem:moment-compare-err}(iv) proves the result.
\end{proof}

\subsection{False pairs}\label{ss:direct-false-comparisons}

Condition on
\[
    \mathcal F=\sigma((u_k^A,u_k^B):k\in[K]\cup R_A\cup R_B).
\]
The coefficient matrices $\mathcal V_G=(v_w^G(r))_{r\in R_G,w\in\mathcal W}$ are $\mathcal F$-measurable. For $k\in\mathcal I$, put $y_k^G=(Z^G_{kr})_{r\in R_G}$. Then
\begin{equation}\label{eq:high-bilinear-score}
    S(i,j)=(y_i^A)^\top\mathcal V_A\mathcal V_B^\top y_j^B.
\end{equation}
Conditional on $\mathcal F$, the coordinates of $y_k^G$ are centered functions of one fresh uniform point and may be dependent. For distinct $i,j$, however, $y_i^A$ and $y_j^B$ are conditionally independent. Write $\kappa_*=C(1+n/d)$.

\begin{lem}[Linear forms in a query vector]\label{lem:query-mgf}
There are $C_{p_0}\ge1$ and an event $\mathcal E'\in\mathcal F$, with $\P[\mathcal E']=1-o(1)$, on which, for both graphs, every $k\in\mathcal I$ and every $\alpha\in\bbR^{R_G}$,
\begin{equation}\label{eq:query-mgf}
    \E\big[e^{\langle\alpha,y_k^G\rangle}\mid\mathcal F\big]
    \le C_{p_0}\exp(C_{p_0}\kappa_*\|\alpha\|^2/2).
\end{equation}
\end{lem}

\begin{proof}
Let $U_G$ have columns $(u_r^G)_{r\in R_G}$ and define
\[
    \mathcal E'=\bigcap_{G\in\{A,B\}}
       \{\|U_G^\top U_G\|_{\mathrm{op}}\le\kappa_*\}.
\]
Gaussian singular-value and norm concentration \cite[Theorems~4.4.5 and~3.1.1]{vershynin2018high}, applied to $u_r^G=g_r/\|g_r\|$, give $\P[(\mathcal E')^c]\le Cne^{-c_3d}=o(1)$ for an absolute constant $c_3>0$ and sufficiently large $C$ in $\kappa_*$. We also use the Gaussian norm bound
\begin{equation}\label{eq:gaussian-norm-tail}
    \P[|\|g\|-\sqrt d|\ge s]\le2e^{-c_2s^2},\qquad s\ge0.
\end{equation}

Fix $\mathcal F$ in $\mathcal E'$, a graph $G$, and a query $k$, and suppress $G$. Represent $u_k=g/\|g\|$ independently of $\mathcal F$. With $X_r=\langle g,u_r\rangle$, $z_d=\sqrt d\tau_{p,d}$, $\delta=z_d(\|g\|/\sqrt d-1)$ and $\bar\alpha=\alpha/\sqrt{p(1-p)}$,
\[
    \langle\alpha,y_k\rangle
    =\sum_r\bar\alpha_r\big(\mathbf1\{X_r\ge z_d+\delta\}-p\big).
\]
The Gaussian vector $X$ has unit marginal variances and covariance $U^\top U\preceq\kappa_*I$. The Gaussian H\"older inequality \cite[Theorem~1(i)]{chen2015improved}, with all exponents equal to $\kappa_*$, and Hoeffding's lemma therefore give, for every deterministic threshold vector $t$ and every $\lambda\in\bbR$,
\[
    \E\Big[\exp\Big(\lambda\sum_r\bar\alpha_r
       (\mathbf1\{X_r\ge t_r\}-\bar\Phi(t_r))\Big)\,\Big|\mathcal F\Big]
    \le\exp(C\lambda^2\kappa_*\|\bar\alpha\|^2).
\]
Here $\bar\Phi$ is the standard Gaussian tail. This inequality applies to singular covariance matrices as well.

To handle the random threshold without conditioning on it, let $\Delta_j=[j/\sqrt d,(j+1)/\sqrt d)$ for $j\in\mathbb Z$, and set
\[
    \Xi_j=\sum_r\bar\alpha_r(\mathbf1\{X_r\ge t_{j,r}\}-p),
    \qquad t_{j,r}=z_d+
       \begin{cases}j/\sqrt d,&\bar\alpha_r\ge0,\\
       (j+1)/\sqrt d,&\bar\alpha_r<0.
       \end{cases}
\]
On $\{\delta\in\Delta_j\}$, monotonicity gives $\langle\alpha,y_k\rangle\le\Xi_j$. Cauchy--Schwarz then yields
\begin{equation}\label{eq:query-mgf-grid}
    \E[e^{\langle\alpha,y_k\rangle}\mid\mathcal F]
    \le e^{C\kappa_*\|\bar\alpha\|^2}
       \sum_{j\in\mathbb Z}\P[\delta\in\Delta_j]^{1/2}
                             e^{\E[\Xi_j\mid\mathcal F]}.
\end{equation}
By the threshold estimate in Lemma~\ref{lem:local-tv} and the Lipschitz bound for $\bar\Phi$,
\[
    |\E[\Xi_j\mid\mathcal F]|\le C_{p_0}(|j|+2)\omega,
    \qquad \omega=\sqrt{N/d}\|\bar\alpha\|.
\]
Also $|z_d|\le C_{p_0}$, so \eqref{eq:gaussian-norm-tail} gives $\P[\delta\in\Delta_j]^{1/2}\le\sqrt2e^{-c_{p_0}(|j|-1)_+^2}$. Completing the square in the sum in \eqref{eq:query-mgf-grid} bounds it by $C_{p_0}e^{C_{p_0}\omega^2}$. Since $\omega^2\le(n/d)\|\bar\alpha\|^2$ and $p\in[p_0,1-p_0]$, this proves \eqref{eq:query-mgf}, uniformly over $\alpha$ and the fresh query label.
\end{proof}

\begin{lem}[Bilinear forms in independent vectors]\label{lem:bilinear-mgf-tail}
Let $x,y$ be independent random vectors and $M$ a deterministic matrix of compatible dimensions. Suppose that $K_0\ge1$, $\sigma>0$, and
\[
    \E e^{\langle\alpha,x\rangle}\le K_0e^{\sigma^2\|\alpha\|^2/2},
    \qquad
    \E e^{\langle\alpha',y\rangle}\le K_0e^{\sigma^2\|\alpha'\|^2/2}
\]
for all $\alpha,\alpha'$. Coordinates within either vector may be dependent. For $t>0$,
\[
    \P[|x^\top My|\ge t]\le2K_0^2\exp\Big(-\min\Big\{
       \frac{t^2}{4\sigma^4\|M\|_{\mathrm F}^2},
       \frac{t}{2\sqrt2\sigma^2\|M\|_{\mathrm{op}}}\Big\}\Big).
\]
\end{lem}

\begin{proof}
Condition first on $y$, then use the quadratic-form Gaussian comparison in \cite[Remark~2.3]{hsu2012tail}. Its comparison of exponential moments uses the assumed linear-form bound once, so a prefactor $K_0$ contributes one additional factor $K_0$. Thus, for $0\le\theta\le(\sqrt2\sigma^2\|M\|_{\mathrm{op}})^{-1}$,
\[
    \E e^{\theta x^\top My}
    \le K_0\E e^{\theta^2\sigma^2\|My\|^2/2}
    \le K_0^2\det(I-\theta^2\sigma^4MM^\top)^{-1/2}
    \le K_0^2e^{\theta^2\sigma^4\|M\|_{\mathrm F}^2}.
\]
The last inequality uses $-\log(1-v)\le2v$ for $0\le v\le1/2$. Apply Chernoff's bound with
\[
    \theta=\min\Big\{\frac{t}{2\sigma^4\|M\|_{\mathrm F}^2},
       \frac{1}{\sqrt2\sigma^2\|M\|_{\mathrm{op}}}\Big\}.
\]
This gives the upper tail; replacing $M$ by $-M$ gives the lower tail. The case $M=0$ is immediate.
\end{proof}

\begin{lem}[False pairs]\label{lem:direct-score-separation}
There are $c_1>0$ and an event $\mathcal E\in\mathcal F$, with $\P[\mathcal E]=1-o(1)$, on which, for every distinct $i,j\in\mathcal I$ and $t>0$,
\begin{equation}\label{eq:direct-false-tail}
    \P[|S(i,j)|\ge t\mid\mathcal F]
    \le2C_{p_0}^2\exp\Big(-c_1\min\Big\{
       \frac{t^2}{\kappa_*^2J},\frac t{\kappa_*}\Big\}\Big),
\end{equation}
where $C_{p_0}$ is from Lemma~\ref{lem:query-mgf}.
\end{lem}

\begin{proof}
The Gram matrix $\mathcal V_G^\top\mathcal V_G$ has entries $g_{wz}^G$. In the \ER expansion \eqref{eq:direct-coefficient-expansion}, edge repetition and the identified first vertices force the two paths to coincide, and their seed edges force $w=z$. Hence $\E_{\ERG}g_{wz}^G=\mathbf1\{w=z\}$. In its second moment, a nonzero pattern with $2V$ reservoir vertices has multiplicity two everywhere. The identified first positions fix the partners, and repeated edges propagate the pairing, so the two coefficient factors have disjoint supports. Their contribution is at most the product of the means. Every remaining nonzero pattern uses at most $2V-1$ labels; the preceding pattern count therefore gives $\operatorname{Var}_{\ERG}(g_{wz}^G)\le(CV)^{CV}/n$. Lemma~\ref{lem:moment-compare-err}(ii) now yields
\[
    |\E g_{wz}^G-\mathbf1\{w=z\}|\le\xi_n,
    \qquad \operatorname{Var}(g_{wz}^G)\le4\xi_n.
\]
Consequently,
\[
    \E\|\mathcal V_G^\top\mathcal V_G-I_J\|_{\mathrm F}^2
    \le J^2(4\xi_n+\xi_n^2)=o(1).
\]
By Markov's inequality, the event
\[
    \mathcal E_0=\bigcap_{G\in\{A,B\}}
       \{\|\mathcal V_G^\top\mathcal V_G-I_J\|_{\mathrm F}\le1\}
\]
has probability $1-o(1)$. On $\mathcal E_0$, the matrix $M=\mathcal V_A\mathcal V_B^\top$ satisfies
\[
    \|M\|_{\mathrm{op}}\le2,
    \qquad
    \|M\|_{\mathrm F}^2
       =\operatorname{tr}((\mathcal V_A^\top\mathcal V_A)
                           (\mathcal V_B^\top\mathcal V_B))\le4J.
\]
Set $\mathcal E=\mathcal E_0\cap\mathcal E'$, with $\mathcal E'$ from Lemma~\ref{lem:query-mgf}. Conditional on $\mathcal F$, distinct query roots are independent; applying Lemma~\ref{lem:bilinear-mgf-tail} to \eqref{eq:high-bilinear-score}, with $K_0=C_{p_0}$ and $\sigma^2=C_{p_0}\kappa_*$, proves \eqref{eq:direct-false-tail}.
\end{proof}

\subsection{Proof of Proposition~\ref{prop:true-seed-candidate} in the high-dimensional regime}\label{ss:high-proof}

\begin{proof}[Proof of Proposition~\ref{prop:true-seed-candidate} in the high-dimensional regime]
We prove the stronger separation $S(i,i)\ge3Jt_n/4$ for all but $o(n)$ queries and $|S(i,j)|<Jt_n/4$ for all false pairs.

Conditional on $R_A$, the reservoir overlap $H$ is hypergeometric, with mean $N^2/(n-K)$ and variance at most $N$. Thus $\mathcal L_n=\{H\ge3c^2n/4\}$ has probability $1-O(n^{-1})$. On this event, for large $n$,
\[
    \frac{(H)_V}{(N)_V}\ge\Big(\frac{H-V+1}{N}\Big)^V\ge(c/2)^V,
    \qquad \mu_n\ge(c/2)^Vr_{\min}^{2V}=t_n.
\]
The parameter choice in \S\ref{ss:high-signature-construction} gives $Jt_n^2\ge(\log n)^{2L_0+4}$, while $t_n\ge(\log n)^{-C}$ and $\xi_n\le n^{-1/2+o(1)}$. Lemma~\ref{lem:direct-path-moments} therefore gives $\E[S(i,i)\mid R_A,R_B]\ge7Jt_n/8$ on $\mathcal L_n$ and, by Chebyshev's inequality,
\[
    \P[S(i,i)<3Jt_n/4\mid R_A,R_B]
    \le\frac C{Jt_n^2}+\frac{C\xi_n}{t_n^2}=:q_n=o(1).
\]
If $h_n$ counts these misses, then $\E[h_n\mathbf1_{\mathcal L_n}]\le nq_n$, so $\P[h_n>n\sqrt{q_n}]\le\P[\mathcal L_n^c]+\sqrt{q_n}=o(1)$.

For false pairs, $\kappa_*=O((\log n)^{L_0})$ and $0<t_n\le1$ imply
\[
    \min\Big\{\frac{(Jt_n/4)^2}{\kappa_*^2J},
                  \frac{Jt_n/4}{\kappa_*}\Big\}\ge c_5(\log n)^4.
\]
Lemma~\ref{lem:direct-score-separation} and a conditional union bound give
\[
    \P\Big[\max_{\substack{i,j\in\mathcal I\\i\ne j}}|S(i,j)|\ge Jt_n/4\Big]
    \le\P[\mathcal E^c]+2C_{p_0}^2n^2e^{-c_1c_5(\log n)^4}=o(1).
\]
Hence \eqref{eq:high-sufficient} holds. Lemma~\ref{lem:threshold-matching} gives at most $3(2cn+O(1))+o(n)=6cn+o(n)$ errors. Undoing the relabeling proves the proposition.
\end{proof}

\section{Analysis in the intermediate regime} \label{s:intermediate-dimension-proof}


We prove Proposition~\ref{prop:true-seed-candidate} for $(\log n)^L\le d\le n/(\log n)^{L_0}$, with the signatures and constants $c,K,C_0$ from \S\ref{ss:low-signature-construction}. In \S\ref{ss:idealized-scores}, we state four main lemmas and prove the proposition assuming them. Their proofs are given, in order, in \S\ref{ss:query-error-analysis}--\S\ref{ss:reservoir-coordinate-transfer}. Additional spherical and tensor estimates are collected in \S\ref{ss:intermediate-spherical-estimates}; their proofs may be skipped on a first reading.

Relabel graph $B$ so that $\pi=\mathrm{id}$ and the true seed guess is $S_A=S_B=[K]$. Throughout this section, the constants in the estimates may depend on $p_0,\rho,c,K,C_0$, but are uniform over $p\in[p_0,1-p_0]$ and the intermediate range of $d$.

\subsection{Outline of the proof}\label{ss:idealized-scores}

The proof rests on Lemmas~\ref{prop:reference-event}--\ref{lem:direction-transfer} below. We first approximate the latent vectors by the spectral embeddings and control the estimated seed directions. We then analyze path directions built from the latent coordinates and compare them with the observed path directions. A concentration estimate for spherical bilinear forms completes the argument.

We use the common query set $\mathcal I$ of Remark~\ref{rmk:choice-cI}. Let
\[
    \mathcal L=\sigma(T_G,D_G,S_G,R_G:G\in\{A,B\})
\]
be the $\sigma$-field generated by the sampled label sets, and let $\mathcal A\in\mathcal L$ be the event that the true seed tuple is admissible in the sense of \eqref{eq:admissible-prob}. \S\ref{ss:preprocessing} gives $\P[\mathcal A^c]\le2Km/n=o(1)$. Extend the embedding in \eqref{eq:reference-map} to the sphere by
\[
    f_G(u)=\lambda^{-1}L_G^\top \big(h_0(u,u_t^G)\big)_{t\in T_G},
\]
with $L_G=0$ when a preprocessing check fails. Let $X_G$ have rows $x(u_t^G)^\top$ in increasing label order, and put $C_{1,G}=X_G^\top L_G$. When $C_{1,G}$ is invertible, write its left polar decomposition as $C_{1,G}=P_GO_G$; otherwise set $O_G=I_d$. Define $\widetilde f_G=O_Gf_G$ and $\widetilde v_a^G=O_Gv_a^G$. These rotations depend only on the reference points. We also write $H_2(z)=zz^\top-\|z\|^2I_d/d$. The first lemma quantifies the error in approximating $x(u)$ by $\widetilde f_G(u)$.

\begin{lem}[Proved in \S\ref{ss:query-error-analysis}] \label{prop:reference-event}
    There are constants $C,\kappa<\infty$ and $\ell_n=C(\log(en))^\kappa$ such that the following holds whenever $d\le m\le n$ and
    \begin{equation}\label{eq:reference-range}
        \ell_n\big(d^{-1/2}+(d/m)^{1/2}\big)\le\frac1{16}.
    \end{equation}
    For each $G\in\{A,B\}$, there is an event measurable with respect to its reference points, of probability at least $1-C(\log(en))^{-1}$ conditional on $\mathcal L$, on which the spectral checks pass, $C_{1,G}$ is invertible, and
    \begin{align}
        \E_u[\widetilde f_G(u)]=0,\qquad \sup_u\|\widetilde f_G(u)-x(u)\|^2 &\le\ell_nd^2/m, \label{eq:reference-event-uniform}\\
        \Big\|\E_u\big[(\widetilde f_G(u)-x(u)) (\widetilde f_G(u)-x(u))^\top\big]\Big\|_{\mathrm{op}} &\le\ell_nd/m, \label{eq:reference-event-covariance}\\
        \big\|\E_u[H_2(\widetilde f_G(u))]\big\|_{\mathrm{op}} &\le\ell_nd/m. \label{eq:reference-event-quadratic}
    \end{align}
    Here $\E_u$ integrates a fresh $u\sim\mu$ with the reference points held fixed. The conclusions hold for both graphs simultaneously with conditional probability at least $1-2C(\log(en))^{-1}$.
\end{lem}

Define the estimated seed Gram matrices by
\begin{equation}\label{eq:matrix-Gamma}
    \Gamma_G(a,b)=\langle\widetilde v_a^G,\widetilde v_b^G\rangle,\qquad \Gamma_{AB}(a,b)=\langle\widetilde v_a^A,\widetilde v_b^B\rangle, \quad a,b \in[K].
\end{equation}
For $v\in\mathbb R^d$, define the local error 
\[
     D_v(u)=H_3(\widetilde f_G(u))[v]-H_3(x(u))[v].
\]
We next record a lemma controlling both the seed Grams and the local errors that are propagated along a path. For $G\in\{A,B\}$, define
\[
    \mathcal F_{G,T}=\mathcal L\vee\sigma(u_t^G:t\in T_G),\qquad \mathcal F_{G,S}=\mathcal F_{G,T}\vee\sigma(u_a^G:a\in S_G), \qquad \mathcal F_G=\mathcal F_{G,S}\vee\sigma(u_w^G:w\in D_G).
\]
To retain joint independence after exposing both graphs' preprocessing,  set
\[
    \mathcal F_0=\mathcal L\vee \sigma\Big((u_i^A,u_i^B): i\in\cup_{G\in\{A,B\}}(T_G\cup D_G\cup S_G)\Big).
\]
Conditional on $\mathcal F_G$, all reservoir vectors $(u_r^G)_{r\in R_G}$ are independent with law $\mu$. Set
\[
    \delta^2=\ell_nd/m,\qquad \zeta^2=\ell_n^4\Big(\frac{N}{md}+\frac{Nd}{m^2}\Big),
\]
and let
\[
    \delta_*^2=\ell_n^6\delta^2,\qquad \zeta_*^2=\ell_n^6\zeta^2,\qquad \varepsilon_n=\ell_n^2\big(d^{-1/2}+\sqrt{d/m}\big).
\]
The constants $b,L,L_0$ are chosen with the logarithmic margins verified in \S\ref{ss:reservoir-coordinate-transfer}; in particular, \eqref{eq:reference-range} holds, all these error parameters tend to zero, and $V\varepsilon_n=o(1)$.

\begin{lem}[Proved in \S\ref{ss:seeded-local-analysis}]\label{lem:seeded-local}
    Assume the conditions of Lemma~\ref{prop:reference-event}, $\delta,\zeta\le1$, $d\le m$, and $\varepsilon_n=o(1)$. For each $G\in\{A,B\}$ there are events $\mathcal E_{\mathrm{ref},G}\in\mathcal F_{G,T}$ and $\mathcal E_G\in\mathcal F_G$ with $\mathcal E_G\subseteq\mathcal E_{\mathrm{ref},G}$ and $\P[\mathcal E_G^c\mid\mathcal L]=o(1)$ uniformly on $\mathcal A$. The conclusions of Lemma~\ref{prop:reference-event} hold on $\mathcal E_{\mathrm{ref},G}$. On $\mathcal E_G$, $\|\widetilde v_a^G\|\le2$ for every $a\in[K]$, the bounds
    \begin{equation}\label{eq:seeded-two-port-bounds}
        \big\|\E_u[D_v(u)^\top D_v(u)]\big\|_{\mathrm{op}} \le\delta_*^2d, \qquad \big\|\E_u[D_v(u)]\big\|_{\mathrm{op}} \le\zeta_*\sqrt{d/N}
    \end{equation}
    hold for $v=\widetilde v_a^G$ and every $a\in[K]$, and
    \[
        \E_u[\widetilde f_G-x]=0, \qquad \E_u[(\widetilde f_G-x)(\widetilde f_G-x)^\top]\preceq\delta_*^2I_d.
    \]
    Finally, there is an event $\mathcal E_{\mathrm{loc}}\in\mathcal F_0$ with $\mathcal E_{\mathrm{loc}}\subseteq\mathcal E_A\cap\mathcal E_B$ and $\P[\mathcal E_{\mathrm{loc}}^c\mid\mathcal L]=o(1)$ uniformly on $\mathcal A$, on which
    \begin{equation}\label{eq:seed-gram-accuracy}
        \|\Gamma_A-I_K\|_{\mathrm{op}} +\|\Gamma_B-I_K\|_{\mathrm{op}} +\|\Gamma_{AB}-\rho I_K\|_{\mathrm{op}} \le\varepsilon_n.
    \end{equation}
\end{lem}

Further define
\[
    \mathcal F=\mathcal L\vee \sigma\big((u_i^A,u_i^B):i\in[n] \setminus\mathcal I\big).
\]
Thus $\mathcal F_G\subseteq\mathcal F_0\subseteq\mathcal F$. Define
\begin{equation}\label{eq:R-A,B-circ}
    R_A^\circ=R_A\setminus(T_B\cup D_B),\qquad R_B^\circ=R_B\setminus(T_A\cup D_A).
\end{equation}
On $\mathcal A$, their sizes lie between $N-2m$ and $N$, their intersection equals $R_A\cap R_B$, and the latent pairs at $R_A^\circ\cup R_B^\circ$ are independent $\rho$-pairs conditional on $\mathcal F_0$. For $\varphi\in\{x,\widetilde f_G\}$, a word $w\in\mathcal W$, and an $\mathcal L$-measurable set $\mathcal R\subseteq R_G$, define
\[
    b_w(\varphi;\mathcal R)= \frac1{\sqrt{(N)_Vd^{V-1}}} \sum_{(r_1,\ldots,r_V)\in\mathcal R_{\ne}^V} \Big(\prod_{t=1}^{V-1} H_3(\varphi(u_{r_t}^G))[\widetilde v_{w_t}^G]\Big) \varphi(u_{r_V}^G),
\]
where $\mathcal R_{\ne}^V$ denotes the injective tuples and the matrix factors are ordered by increasing $t$. Put
\[
    \mathcal B_G=(b_w(\widetilde f_G;R_G))_{w\in\mathcal W}, \qquad \mathcal B_G^{0,\circ}=(b_w(x;R_G^\circ))_{w\in\mathcal W}.
\]
Both matrices use the original normalization $(N)_V$ and the same estimated seed directions. The identity $H_3(Oz)[Ov]=OH_3(z)[v]O^\top$ gives $\mathcal B_G=(O_Gb_w^G)_{w\in\mathcal W}$, so the rotation leaves the signatures unchanged. Recall \eqref{eq:def-score}. We then have
\[
     S(i,j)=d^{-1} \widetilde{f}_A(u_i^A)^\top \mathcal B_A\mathcal B_B^\top \widetilde{f}_B(u_{j}^B).
\]
As Lemma~\ref{prop:reference-event} implies that $\widetilde{f}_G(u) \approx x(u)$, this motivates us to consider the ``cleaned'' score
\[
    S^{\cl}(i,j)=\langle z_i^A,z_j^B\rangle =d^{-1}x(u_i^A)^\top\mathcal B_A\mathcal B_B^\top x(u_j^B), \qquad z_i^G=\mathcal B_G^\top x(u_i^G)/\sqrt d.
\]
Note that $\mathcal B_A,\mathcal B_B$ are $\cF$-measurable, and conditional on $\mathcal F$, the pairs $(u_i^A,u_i^B)$, $i\in\mathcal I$, are independent $\rho$-pairs. Thus, by a Hanson--Wright type inequality for bilinear forms, it suffices to bound the norm of the coefficient matrices $\mathcal B_A\mathcal B_B^\top$. Our strategy is to first bound the norm of the simpler matrices $\mathcal B_A^{0,\circ} (\mathcal B^{0,\circ}_B)^\top$, and then control the approximation error of $\cB_G^{0,\circ}$ to $\cB_G$.

\begin{lem}[Proved in \S\ref{ss:coordinate-path-moments}]\label{lem:oracle-gram-concentration}
    For the parameter choices in \S\ref{ss:reservoir-coordinate-transfer}, with probability $1-o(1)$,
    \[
        \max_{G\in\{A,B\}} \Big\|\frac{(\mathcal B_G^{0,\circ})^\top \mathcal B_G^{0,\circ}}d-I_J\Big\|_{\mathrm{op}} =o(1), \qquad \rho\operatorname{tr}\Big( \frac{(\mathcal B_A^{0,\circ})^\top \mathcal B_B^{0,\circ}}d \Big)\ge(1-o(1))Jt_n.
    \]
\end{lem}

\begin{lem}[Proved in \S\ref{ss:reservoir-coordinate-transfer}] \label{lem:direction-transfer}
    For the same parameter choices, with probability $1-o(1)$,
    \begin{equation}\label{eq:full-to-clean-direction-error}
        \max_{G\in\{A,B\}} \frac{\|\mathcal B_G-\mathcal B_G^{0,\circ}\|_{\mathrm F}}{\sqrt d} \le t_n(\log n)^{-2}.
    \end{equation}
\end{lem}

To pass from the Gram estimates to score separation, we use the following
concentration bound for spherical bilinear forms.

\begin{lem}\label{lem:spherical-query-scores}
    Let $d\ge3$ and $\vartheta\in[-1,1]$, let $(u,u')$ be a $\vartheta$-pair, and let $x=\sqrt du$ and $y=\sqrt du'$. There are absolute constants $C,c_1>0$ such that, for every deterministic real $d\times d$ matrix $Q\ne0$ and every $t>0$,
    \begin{equation}\label{eq:spherical-query-tail}
        \P\big[|x^\top Qy-\vartheta\operatorname{tr}Q|>t\big]\le C\exp\Big(-c_1\min\Big\{\frac{t^2}{\|Q\|_{\mathrm F}^2},\frac{t}{\|Q\|_{\mathrm{op}}}\Big\}\Big).
    \end{equation}
    If instead $x$ and $y$ are independent and uniform on the sphere of radius $\sqrt d$, then \eqref{eq:spherical-query-tail} holds with $\vartheta\operatorname{tr}Q$ replaced by $0$.
\end{lem}
\begin{proof}
    Write $y=\sqrt d(\vartheta u+\sqrt{1-\vartheta^2}w)$, where $(u,w)$ is a uniform orthonormal pair. The quadratic form in $z=\sqrt d(u,w)$ with coefficient matrix
\[
 M=\frac12\begin{pmatrix}
 \vartheta(Q+Q^\top)&\sqrt{1-\vartheta^2}Q\\
 \sqrt{1-\vartheta^2}Q^\top&0
 \end{pmatrix}
\]
equals $x^\top Qy$, and
$\operatorname{tr}M=\vartheta\operatorname{tr}Q$,
$\|M\|_{\mathrm F}\le2\|Q\|_{\mathrm F}$,
$\|M\|_{\mathrm{op}}\le2\|Q\|_{\mathrm{op}}$.
Apply \cite[Theorem~7.1(1)]{gotze2023higher} to the two-column Stiefel
matrix with coefficient $dM$, using $(d-2)/d\ge1/3$.
For independent $x,y$, condition on $r=x^\top y/d$ and apply this bound
with $\vartheta=r$. The spherical linear-form inequality
\cite[Eq.~(7.1)]{gotze2023higher}, together with
$|\operatorname{tr}Q|\le\sqrt d\|Q\|_{\mathrm F}$, bounds
$\P[|r\operatorname{tr}Q|>t]$ by
$2\exp(-c_1t^2/\|Q\|_{\mathrm F}^2)$. A union bound at threshold $t/2$
proves the independent case.
\end{proof}

We now use these estimates to prove Proposition~\ref{prop:true-seed-candidate}.

\begin{proof}[Proof of Proposition~\ref{prop:true-seed-candidate}
in the intermediate regime]
Choose the parameters as in
\S\ref{ss:reservoir-coordinate-transfer}. 
For a constant $C\ge2$, let $\mathcal E_{\norm}\in\mathcal F$ be the
event that $\mathcal A$ holds, the preprocessing checks pass,
$C_{1,A},C_{1,B}$ are invertible, and
\[
 \max_{G\in\{A,B\}}
 \Big\|\frac{\mathcal B_G^\top\mathcal B_G}d\Big\|_{\mathrm{op}}
 \le C,\qquad
 \rho\operatorname{tr}\Big(\frac{\mathcal B_A^\top\mathcal B_B}d\Big)
 \ge\tfrac78Jt_n.
\]
Write
\[
 e_i^G=\frac1{\sqrt d}\mathcal B_G^\top
              \big(\widetilde f_G(u_i^G)-x(u_i^G)\big),
 \qquad
 \mathcal G=\Big\{i\in\mathcal I:
       \max_G\|e_i^G\|\le\varepsilon t_n\sqrt J\Big\},
\]
where $\varepsilon>0$ will be fixed in the proof. Thus
$\Phi^G(i)=z_i^G+e_i^G$.
On the event in
Lemma~\ref{lem:oracle-gram-concentration}, the oracle matrices have
operator norm $O(\sqrt d)$ and Frobenius norm $O(\sqrt{dJ})$.
Expanding the Gram products and applying
Lemma~\ref{lem:direction-transfer} changes the within-graph operator
norms by $o(1)$ and the normalized cross trace by at most
\[
 C\sqrt Jt_n(\log n)^{-2}+t_n^2(\log n)^{-4}=o(Jt_n).
\]
Consequently, with probability $1-o(1)$,
\[
 \max_{G\in\{A,B\}}
 \Big\|\frac{\mathcal B_G^\top\mathcal B_G}d-I_J\Big\|_{\mathrm{op}}
 =o(1),\qquad
 \rho\operatorname{tr}\Big(\frac{\mathcal B_A^\top\mathcal B_B}d\Big)
 \ge(1-o(1))Jt_n.
\]
Together with Lemma~\ref{prop:reference-event} and
$\P[\mathcal A^c]=o(1)$, this proves
$\P[\mathcal E_{\norm}^c]=o(1)$ with $C=2$.

Fix $C_2\ge\sqrt{2C}$ and choose $\varepsilon>0$ so that
$2C_2\varepsilon+\varepsilon^2<1/8$. Let
$\mathcal E_{\clean}$ be the event that, for all $i,j\in\mathcal I$ and
$G\in\{A,B\}$,
\[
 S^{\cl}(i,i)\ge\tfrac34Jt_n,\qquad
 |S^{\cl}(i,j)|\le\tfrac14Jt_n\quad(i\ne j),\qquad
 \|z_i^G\|\le C_2\sqrt J.
\]
Conditional on $\mathcal F$ on $\mathcal E_{\norm}$, set
$Q=\mathcal B_A\mathcal B_B^\top/d$. Then
\[
 \|Q\|_{\mathrm{op}}\le C,\qquad
 \|Q\|_{\mathrm F}^2
 =\operatorname{tr}\Big(
       \frac{\mathcal B_A^\top\mathcal B_A}d
       \frac{\mathcal B_B^\top\mathcal B_B}d\Big)
 \le C^2J.
\]
The conditional means of the idealized scores are
$\rho\operatorname{tr}Q\ge7Jt_n/8$ for true pairs and zero for false
pairs. Lemma~\ref{lem:spherical-query-scores} bounds the probability of
each score failure by $C_1e^{-c_1Jt_n^2}$. Applied with $\vartheta=1$ to
$Q^{(G)}=\mathcal B_G\mathcal B_G^\top/d$, whose operator norm, squared
Frobenius norm and trace are at most $C,C^2J,CJ$, respectively, the same
lemma bounds the probability of each norm failure by $C_1e^{-c_1J}$.
Since \S\ref{ss:low-signature-construction} gives
\[
 Jt_n^2=K^{-1}\rho^{-4}
       \big(K(c/2)^2\rho^8\big)^V\gg(\log n)^4,
 \qquad J\ge Jt_n^2,
\]
a conditional union bound yields
$\P[\mathcal E_{\clean}^c\mid\mathcal F]
\le C_1n^2e^{-c_1Jt_n^2}+2C_1ne^{-c_1J}=o(1)$ on
$\mathcal E_{\norm}$. Integrating gives
\[
 \P[\mathcal E_{\clean}^c\cap\mathcal E_{\norm}]=o(1).
\]
Let $\mathcal E_{\mathrm{ref}}\in\mathcal F$ be the intersection of the
reference events in Lemma~\ref{prop:reference-event}.
On $\mathcal E_{\mathrm{ref}}\cap\mathcal E_{\norm}$, the covariance
estimate \eqref{eq:reference-event-covariance} gives
\[
 \E[\|e_i^G\|^2\mid\mathcal F]
 \le\delta^2\operatorname{tr}
       (\mathcal B_G^\top\mathcal B_G/d)\le C\delta^2J.
\]
Thus
$\E[|\mathcal I\setminus\mathcal G|/n\mid\mathcal F]
\le2C\delta^2/(\varepsilon^2t_n^2)$ on that event.
Take $r_n=(\log n)^{-1}$ and let
$\mathcal E_{\err}=\{|\mathcal I\setminus\mathcal G|\le nr_n\}$.
Markov's inequality, applied conditionally, and
\eqref{eq:query-error-scale} imply
\[
 \P[\mathcal E_{\err}^c\cap\mathcal E_{\norm}]
 \le\P[\mathcal E_{\mathrm{ref}}^c]
       +\frac{2C}{\varepsilon^2(\log n)^3}=o(1).
\]

The event
$\mathcal E_{\norm}\cap\mathcal E_{\clean}\cap\mathcal E_{\err}$
therefore has probability $1-o(1)$. On this event the true candidate is
constructed and, for $i,j\in\mathcal G$,
\begin{align*}
 |S(i,j)-S^{\cl}(i,j)|
 \le\|e_i^A\|\|z_j^B\|+\|z_i^A\|\|e_j^B\|
               +\|e_i^A\|\|e_j^B\| \le(2C_2\varepsilon+\varepsilon^2)Jt_n
 <\tfrac18Jt_n.
\end{align*}
Hence true scores on $\mathcal G$ exceed $5Jt_n/8$ and false scores
within $\mathcal G$ are below $3Jt_n/8$. At threshold
$\theta_n=Jt_n/2$, Lemma~\ref{lem:threshold-matching} applies with $h=0$.
By \eqref{eq:physical-query-exclusions}, the completed candidate has at
most
\[
 3|[n]\setminus\mathcal G|
 \le3(2N+4m+K+nr_n)\le6cn+o(n)
\]
errors. Undoing the relabeling of $B$ proves the assertion.
\end{proof}

\begin{rmk}\label{rmk:coordinate-score-separation}
    In \S\ref{s:color-coding}, we will use the following corollary in the proof of Proposition~\ref{prop:color-coded-candidate}. For the parameter choices above, with probability $1-o(1)$ there is a set $\mathcal G\subseteq\mathcal I$ with $|\mathcal I\setminus\mathcal G|=o(n)$, such that
    \[
        S(i,i)>\tfrac58Jt_n\quad(i\in\mathcal G),\qquad S(i,j)<\tfrac38Jt_n\quad(i\ne j,\ i,j\in\mathcal G).
    \]
\end{rmk}

\subsection{Spectral embedding error bound}
\label{ss:query-error-analysis}
\label{ss:intermediate-preliminary-inputs}
\label{ss:reference-calibration}
\label{ss:local-moment-comparison}

We prove Lemma~\ref{prop:reference-event} using the spherical reference
estimates collected in \S\ref{ss:intermediate-spherical-estimates}. Fix $G\in\{A,B\}$ and
condition on $\mathcal L$. Write $t_1<\cdots<t_m$ for the labels in $T_G$
and $x_j=x(u_{t_j}^G)$, so that the rows of $X_G$ are $x_j^\top$.
Since the label sets are sampled independently of the latent vectors,
the reference points remain independent with law $\mu$.
Separate the linear harmonic from the cap kernel by defining
\[
 k(u,w)=h_0(u,w)-\lambda\langle x(u),x(w)\rangle,\qquad
 e_G(u)=\lambda^{-1}L_G^\top\big(k(u,u_t^G)\big)_{t\in T_G}.
\]
Since $\mathcal Tx=\lambda x$, $\E_u[xx^\top]=I_d$, and $h_0$
is centered, the residual is orthogonal to constants and linear
functions. Thus
\begin{equation}\label{eq:f-G-decomposition}
 f_G(u)=C_{1,G}^\top x(u)+e_G(u),\qquad
 \E_u[e_G(u)]=0,\qquad
 \E_u[x(u)e_G(u)^\top]=0.
\end{equation}
In particular, when $C_{1,G}=P_GO_G$ is invertible, its rotated
embedding satisfies $\widetilde f_G=P_Gx+O_Ge_G$.

For $u\in\mathbb S^{d-1}$, write
\[
 h_0(u,T_G)=\big(h_0(u,u_t^G)\big)_{t\in T_G},\qquad
 k(u,T_G)=\big(k(u,u_t^G)\big)_{t\in T_G}.
\]
When the spectral checks pass, $L_G=Q_G\Lambda_G^{-1/2}$ and
$f_G(u)=\lambda^{-1}L_G^\top h_0(u,T_G)$. Define
\[
 W_G=X_G^\top X_G,\qquad
 \overline X_G=X_GW_G^{-1/2},\qquad
 \mathcal K_G=H_G-\lambda X_GX_G^\top,\qquad
 E_G=\lambda^{-1}\mathcal K_G,
\]
so that $H_G/\lambda=X_GX_G^\top+E_G$. The off-diagonal entries of
$\mathcal K_G$ are $k(u_{t_i}^G,u_{t_j}^G)$, and its diagonal entries
are $-d\lambda$. Also let
\[
 \sigma_k^2=\int k(u,v)^2\mu(\de v)=p(1-p)-d\lambda^2\le C,\qquad
 g(u,v)=\int k(u,w)k(v,w)\mu(\de w),
\]
and
\[
 g(T_G,T_G)=\big(g(u_{t_i}^G,u_{t_j}^G)\big)_{i,j\le m}.
\]
Throughout this subsection, $\E_u$ integrates a fresh point
$u\sim\mu$, with the references held fixed, and $\sup_u$ is over
$\mathbb S^{d-1}$.

\begin{proof}[Proof of Lemma~\ref{prop:reference-event}]
Write $\operatorname{polylog}(n)$ for a logarithmic factor whose fixed
exponent may change between bounds. All events below depend only on the
reference points. To control their failure probability, apply Markov's
inequality to the estimates in Lemmas~\ref{lem:reference-positive-gram}
and~\ref{lem:reference-operator-estimates}, and to
\eqref{eq:projected-reference-Gram-moment}. Intersect these events with the covariance and
empirical-process events below; the total failure probability is
$O((\log(en))^{-1})$. A sufficiently large fixed $\kappa$ absorbs all
resulting logarithmic losses. Condition~\eqref{eq:reference-range}
ensures $m\ge Cd$.

The sample-covariance theorem
\cite[Theorem~4.6.1]{vershynin2018high} gives
$mI_d/2\preceq W_G\preceq2mI_d$. Lemma~\ref{lem:reference-operator-estimates}
and $\lambda^{-1}\le C\sqrt d$ give
\[
 \frac{\|E_G\|_{\mathrm{op}}}{m}
 \le\ell_n\big(d^{-1/2}+\sqrt{d/m}\big)
\]
and
\[
 \|E_G\overline X_G\|_{\mathrm{op}}\le\operatorname{polylog}(n)\sqrt{md},
 \qquad
 \|X_G^\top E_GX_G\|_{\mathrm{op}}\le\operatorname{polylog}(n)md.
\]
Since $\|E_G\|_{\mathrm{op}}\le m/16$, Weyl's inequality puts the signal
eigenvalues of $H_G/\lambda$ in $[3m/8,17m/8]$ and the remaining
ones in $[-m/8,m/8]$. Thus the spectral checks pass,
$\|L_G\|_{\mathrm{op}}\le C/\sqrt m$ and
$\|C_{1,G}\|_{\mathrm{op}}\le3$.

Set $\Delta_G=(I_m-\overline X_G\overline X_G^\top)L_G$,
$Y_G=(I_m-\overline X_G\overline X_G^\top)Q_G$,
$J_G=\overline X_G^\top Q_G$, and
$E_{\perp,G}=(I_m-\overline X_G\overline X_G^\top)
E_G(I_m-\overline X_G\overline X_G^\top)$.
Projection gives the decomposition
\begin{equation}\label{eq:reference-event-decomposition}
 L_G=X_GW_G^{-1}C_{1,G}+\Delta_G,\qquad X_G^\top\Delta_G=0,
\end{equation}
and
\[
 Y_G\Lambda_G-E_{\perp,G}Y_G
 =(I_m-\overline X_G\overline X_G^\top)E_G\overline X_GJ_G.
\]
The Sylvester separation bound
\cite[Section~VII.2]{bhatia1997matrix}, with spectral gap at least $m/4$,
and $\Delta_G=Y_G\Lambda_G^{-1/2}$ give
\begin{equation}\label{eq:reference-perpendicular-weight}
 \|\Delta_G\|_{\mathrm{op}}
 \le C\frac{\|E_G\overline X_G\|_{\mathrm{op}}}{m\sqrt m}
 \le\operatorname{polylog}(n)\frac{\sqrt d}{m}.
\end{equation}
Next,
$I_d=L_G^\top(H_G/\lambda)L_G=C_{1,G}^\top C_{1,G}+L_G^\top E_GL_G$.
Expand the last term using \eqref{eq:reference-event-decomposition}:
\[
 \|C_{1,G}^\top C_{1,G}-I_d\|_{\mathrm{op}}
 \le\frac C{m^2}\big(\|X_G^\top E_GX_G\|_{\mathrm{op}}
                 +\|E_G\overline X_G\|_{\mathrm{op}}^2\big)
 \le\operatorname{polylog}(n)d/m.
\]
This also bounds $C_{1,G}C_{1,G}^\top-I_d$. Enlarging the fixed
logarithmic factor $\ell_n$ if necessary, we have therefore proved
\[
 \|C_{1,G}\|_{\mathrm{op}}\le2,\qquad
 \|C_{1,G}C_{1,G}^\top-I_d\|_{\mathrm{op}}\le\ell_nd/m,\qquad
 \|\Delta_G\|_{\mathrm{op}}\le\ell_n\sqrt d/m.
\]
Condition~\eqref{eq:reference-range} makes the middle bound smaller than
one, so $C_{1,G}$ is invertible. These estimates will also be used in the
local comparisons below.

For the residual covariance, set $L_{0,G}=X_GW_G^{-1}$ and
$e_{0,G}(u)=\lambda^{-1}L_{0,G}^\top k(u,T_G)$. Then
\begin{equation}\label{eq:reference-exact-decomposition}
 e_G(u)=C_{1,G}^\top e_{0,G}(u)
             +\lambda^{-1}\Delta_G^\top k(u,T_G).
\end{equation}
Since $\E_u[k(u,T_G)k(u,T_G)^\top]=g(T_G,T_G)$,
\[
 \Sigma_{0,G}:=\E_u[e_{0,G}e_{0,G}^\top]
 =\lambda^{-2}W_G^{-1}X_G^\top g(T_G,T_G)X_GW_G^{-1}.
\]
Equations~\eqref{eq:projected-reference-Gram}--\eqref{eq:projected-reference-Gram-moment}
for the full sample give
\[
 \|X_G^\top g(T_G,T_G)X_G\|_{\mathrm{op}}
 \le\operatorname{polylog}(n)m,\qquad
 \|\Sigma_{0,G}\|_{\mathrm{op}}
 \le\operatorname{polylog}(n)d/m.
\]
Using \eqref{eq:reference-exact-decomposition},
\eqref{eq:reference-perpendicular-weight} and
Lemma~\ref{lem:reference-positive-gram}, we obtain
\begin{align*}
 \Sigma_{e,G}:=\E_u[e_Ge_G^\top]
 &\preceq2C_{1,G}^\top\Sigma_{0,G}C_{1,G}
       +2\lambda^{-2}\Delta_G^\top g(T_G,T_G)\Delta_G,\\
 \|\Sigma_{e,G}\|_{\mathrm{op}}
 &\le\operatorname{polylog}(n)(d/m+d^2/m^2+1/m)
 \le\operatorname{polylog}(n)d/m.
\end{align*}
The polar factor satisfies
$\|P_G-I_d\|_{\mathrm{op}}\le\operatorname{polylog}(n)d/m$.
The orthogonality in \eqref{eq:f-G-decomposition} now gives
\begin{align*}
 \E_u[(\widetilde f_G-x)(\widetilde f_G-x)^\top]
 =(P_G-I_d)^2+O_G\Sigma_{e,G}O_G^\top, \qquad
 \E_u[\widetilde f_G\widetilde f_G^\top]
 =P_G^2+O_G\Sigma_{e,G}O_G^\top.
\end{align*}
These identities imply \eqref{eq:reference-event-covariance}
and~\eqref{eq:reference-event-quadratic}. The mean is zero because every
section of $h_0$ is centered.

To control the supremum, define
$\widehat x_G(u)=(m\lambda)^{-1}\sum_jh_0(u,u_{t_j}^G)x_j$.
Its reference expectation is $x(u)$, and
\begin{equation}\label{eq:backprojection-identity}
 (W_G/m)^{-1}\widehat x_G(u)=x(u)+e_{0,G}(u).
\end{equation}
Apply Lemma~\ref{lem:uniform-cap-backprojection} with $t=\log n$.
Together with \eqref{eq:backprojection-identity} and the covariance
estimate \cite[Theorem~4.6.1]{vershynin2018high}, it gives
\[
 \sup_u\|e_{0,G}(u)\|
 \le2\sup_u\|\widehat x_G(u)-x(u)\|
       +2\sqrt d\|W_G/m-I_d\|_{\mathrm{op}}
 \le\operatorname{polylog}(n)d/\sqrt m.
\]
Finally $X_G^\top\Delta_G=0$ implies
$\Delta_G^\top k(u,T_G)=\Delta_G^\top h_0(u,T_G)$, so
\[
 \widetilde f_G-x=(P_G-I_d)x+P_Ge_{0,G}
                    +\lambda^{-1}O_G\Delta_G^\top h_0(u,T_G).
\]
Use $\sup_u\|h_0(u,T_G)\|\le\sqrt m$ and
\eqref{eq:reference-perpendicular-weight} to bound its supremum by
\[
 \operatorname{polylog}(n)(d^{3/2}/m+d/\sqrt m) \le\operatorname{polylog}(n)d/\sqrt m.
\]
Squaring proves \eqref{eq:reference-event-uniform}. A union bound gives
the simultaneous assertion for both graphs.
\end{proof}

\subsection{Local estimates and seed Gram matrix bounds}
\label{ss:seeded-local-analysis}

We prove Lemma~\ref{lem:seeded-local}. We first control the local tensor
differences before contraction with a seed, and then use the construction
of the estimated seed directions to obtain the stronger matrix bounds
required by the path expansion. The conditioning fields
$\mathcal F_{G,T}\subseteq\mathcal F_{G,S}\subseteq\mathcal F_G$
and $\mathcal F_0$ are those defined in the outline.

For a tensor $T$ of order $q$ on $\bbR^d$ and a slot $a\in[q]$, let
$\mathcal M_a(T)$ be the linear map that contracts the $a$th slot of $T$
with an input vector in $\bbR^d$; its output, a tensor of order $q-1$, is
measured in Frobenius norm. Define the \emph{contraction norm}
\[
 \|T\|_{1\to\mathrm F}=\max_{a\in[q]}\|\mathcal M_a(T)\|_{\mathrm{op}}.
\]
For a matrix, this is the operator norm. When $T=T(u)$ depends on a fresh
point $u\sim\mu$, define the \emph{mean-square contraction norm}
\[
 \|T\|_{\mathcal P}^2
 =\max_{a\in[q]}\big\|\E_u[\mathcal M_a(T)^*\mathcal M_a(T)]\big\|_{\mathrm{op}}
 =\max_{a\in[q]}\sup_{\|v\|=1}\E_u\big[\|\mathcal M_a(T)v\|_{\mathrm F}^2\big].
\]
Here and below, $\E_u$ denotes integration over a fresh point $u\sim\mu$,
independent of all sampled data, with all maps and seed directions held
fixed. For $q=1$, $\|T\|_{\mathcal P}^2=\|\E_u[TT^\top]\|_{\mathrm{op}}$.
Let $\operatorname{Sym}$ be the average over all permutations of the
slots, and let $\cP_{\mathsf{sym}}$ be the orthogonal projection onto
symmetric trace-free tensors of the relevant order; we write
$\cP_{\mathsf{sym}}^{(q)}$ when the order $q$ is indicated. For
$q\in\{1,2,3\}$ and $z\in\bbR^d$, let
$H_q(z)=\cP_{\mathsf{sym}}(z^{\otimes q})$. Thus $H_1(z)=z$; the tensors
$H_2$ and $H_3$ agree with the definitions already used above and in \eqref{eq:low-cubic-matrix}.

The generic bounds for these norms are proved in
Lemmas~\ref{lem:stf-one-slot} and~\ref{lem:positive-local-difference}.
For $G\in\{A,B\}$ and $q\in\{1,2,3\}$, define
\[
 D_q^G(u)=H_q(\widetilde f_G(u))-H_q(x(u)),\qquad
 A_q^G=\E_u[D_q^G(u)].
\]
Recall $\delta^2=\ell_nd/m$ and
$\zeta^2=\ell_n^4(N/(md)+Nd/m^2)$ from the outline.

\begin{ppn}\label{prop:local-moments}
Assume the conditions of Lemma~\ref{prop:reference-event}. There are
events $\mathcal E_{\mathrm{ref},G}\in\mathcal F_{G,T}$, $G\in\{A,B\}$,
with $\P[\mathcal E_{\mathrm{ref},G}^c\mid\mathcal L]=o(1)$ uniformly on
$\mathcal A$, on which the conclusions of
Lemma~\ref{prop:reference-event} hold for $G$. If moreover
$\delta,\zeta\le1$, then on $\mathcal E_{\mathrm{ref},G}$ we have
$A_1^G=0$ and, for $q\in\{1,2,3\}$,
\begin{align}
 \|D_q^G\|_{\mathcal P}^2&\le C\delta^2d^{q-1},
       \label{eq:local-positive-interface}\\
 \|A_q^G\|_{1\to\mathrm F}^2&\le C\frac{\zeta^2}{N}d^{q-1},
       \label{eq:local-singleton-interface}\\
 \|H_q(x)\|_{\mathcal P}^2+\|H_q(\widetilde f_G)\|_{\mathcal P}^2
       &\le Cd^{q-1}. \notag
\end{align}
\end{ppn}


\begin{proof}
Fix $G$, conditional on $\mathcal L$, and work on the reference event
constructed in the proof of Lemma~\ref{prop:reference-event}, including
the calibration estimates established there. Set $\widetilde L_G=L_GO_G^\top$ and
$\widetilde\Delta_G=\Delta_GO_G^\top$. Equations
\eqref{eq:reference-event-decomposition} and~\eqref{eq:f-G-decomposition}
give
\begin{align}
 \widetilde L_G&=X_GW_G^{-1}P_G+\widetilde\Delta_G,
 \quad X_G^\top\widetilde\Delta_G=0,
 \label{eq:local-factor-decomposition}\\
 \widetilde f_G&=P_Gx+\widetilde e_G,
 \quad\widetilde e_G(u)=\lambda^{-1}\widetilde L_G^\top k(u,T_G).
 \label{eq:local-embedding-decomposition}
\end{align}
On the reference event, $\eta_G=\widetilde f_G-x$ satisfies
\eqref{eq:local-error-premises}, so
Lemma~\ref{lem:positive-local-difference} gives
\eqref{eq:local-positive-interface} and the final bound.
Centering gives $A_1^G=0$, while
\eqref{eq:reference-event-quadratic} gives
$\|A_2^G\|_{\mathrm{op}}^2\le\ell_n^2d^2/m^2$.
For the cubic mean, expand
\[
 A_3^G=\cP_{\mathsf{sym}}\operatorname{Sym}\E_u
 [3\eta_G\otimes x\otimes x+3\eta_G\otimes\eta_G\otimes x
       +\eta_G^{\otimes3}].
\]
Oddness removes the $(P_G-I_d)x$ part of the first term.
The Funk--Hecke formula \cite[Theorem~2.9]{dai2013spherical} gives
\begin{equation}\label{eq:local-first-variation}
 V_3^G:=3\cP_{\mathsf{sym}}\operatorname{Sym}
       \E_u[\widetilde e_G\otimes x\otimes x]
 =3\frac{\lambda_2}{\lambda}\cP_{\mathsf{sym}}\operatorname{Sym}
       \sum_j\widetilde L_{G,j}\otimes x_j\otimes x_j,
\end{equation}
where $\widetilde L_{G,j}$ is the transpose of row $j$.
Indeed, for a symmetric trace-free cubic tensor $S$,
$S[y](x,x)$ is a degree-two harmonic, so integration against
$k(u,u_{t_j}^G)$ multiplies it by $\lambda_2$.
Recall $\lambda_2/\lambda=O(d^{-1/2})$ from
\eqref{eq:exact-cap-multipliers}.

Inserting \eqref{eq:local-factor-decomposition} into
\eqref{eq:local-first-variation} gives two contributions. For the first,
Lemma~\ref{lem:empirical-cubic-tensor} gives, with conditional failure
probability $n^{-\gamma}$ for any fixed $\gamma>0$,
\[
 \Big\|\sum_jx_j^{\otimes3}\Big\|_{1\to\mathrm F}
 \le C_\gamma(d\sqrt{m\log n}+d^{3/2}\log n).
\]
Since $\|P_GW_G^{-1}\|_{\mathrm{op}}\le4/m$, its contribution to
$\|V_3^G\|_{1\to\mathrm F}$ is at most
$C\ell_n(\sqrt{d/m}+d/m)$.

For the second contribution, let $D\odot Y$ denote rowwise tensor
products. The Schur product bound
\cite[Section~5.2]{horn1991topics} gives
$\|D\odot Y\|_{\mathrm{op}}^2
\le\|D\|_{\mathrm{op}}^2\max_j\|Y(j,:)\|^2$.
The relevant flattenings are
$\widetilde\Delta_G^\top(X_G\odot X_G)$ and
$X_G^\top(\widetilde\Delta_G\odot X_G)$, up to transposition.
Their norms are at most
$C\sqrt{md}\|\widetilde\Delta_G\|_{\mathrm{op}}
\le C\ell_nd/\sqrt m$.
Multiplying by $\lambda_2/\lambda$ and applying
Lemma~\ref{lem:stf-one-slot} to both contributions yields
\[
 \|V_3^G\|_{1\to\mathrm F}^2\le C\ell_n^2d/m.
\]

For the terms containing at least two error factors, Cauchy--Schwarz
gives, for any Euclidean-valued $Y=Y(u)$,
\[
 \|\E_u[Y\otimes\eta_G]\|_{\mathrm F}^2
 \le\|\E_u[\eta_G\eta_G^\top]\|_{\mathrm{op}}\E_u\|Y\|^2.
\]
After contracting a unit $v$, take
$Y=\langle v,x\rangle\eta_G$,
$\langle v,\eta_G\rangle x$, or
$\langle v,\eta_G\rangle\eta_G$, according to the slot. From
\eqref{eq:local-error-premises},
\[
 \|\E_u[x\otimes\eta_G\otimes\eta_G]\|_{1\to\mathrm F}^2\le\delta^4d,
 \qquad
 \|\E_u[\eta_G^{\otimes3}]\|_{1\to\mathrm F}^2\le\delta^6d\le\delta^4d.
\]
Consequently,
\[
 \|A_3^G\|_{1\to\mathrm F}^2\le C\ell_n^2(d/m+d^3/m^2).
\]
Recall that $\zeta^2/N=\ell_n^4(1/(md)+d/m^2)$.
Combining this identity with the cubic and quadratic bounds proves
\eqref{eq:local-singleton-interface}.
Intersect the original reference event with the cubic Bernstein event
to obtain $\mathcal E_{\mathrm{ref},G}\in\mathcal F_{G,T}$.
All conditional failure bounds are uniform over $\mathcal L$ on
$\mathcal A$, and a union bound covers both graphs.
\end{proof}

\begin{proof}[Proof of Lemma~\ref{lem:seeded-local}]
Fix $G$ and condition on $\mathcal F_{G,T}$, on
$\mathcal E_{\mathrm{ref},G}$. Seeds and auxiliary points are independent
with law $\mu$. We first treat one seed $u_*=u_a^G$.

By \eqref{eq:smoothed-seeds},
\[
 \overline v_G(u_*):=\E[\widetilde v_a^G\mid\mathcal F_{G,S}]
 =\frac1{\lambda\sqrt d}\int\widetilde f_G(u)h_0(u_*,u)\mu(\de u).
\]
Using \eqref{eq:local-embedding-decomposition}, the linear part is
$P_Gu_*$ and the residual is a convolution of $k$. Thus
\[
 \widetilde v_a^G=P_Gu_*+r_G(u_*)+\epsilon_G(u_*),\qquad
 r_G(u_*)=\frac{\widetilde L_G^\top(g(u_*,u_t^G))_{t\in T_G}}
                      {\lambda^2\sqrt d},
\]
where $\epsilon_G=\widetilde v_a^G-\overline v_G$.
Since $\|\widetilde L_G\|_{\mathrm{op}}\le2/\sqrt m$,
\eqref{eq:residual-convolution-moments} gives
\[
 \E_{u_*}\|r_G(u_*)\|^2
 \le\frac C{m\lambda^4d}\sum_{t\in T_G}\E_{u_*}g(u_*,u_t^G)^2
 \le C/d.
\]
Conditional on $\mathcal F_{G,S}$, $\epsilon_G$ averages $m$ independent
centered vectors of second moment at most $Cd$. Hence
$\E[\|\epsilon_G\|^2\mid\mathcal F_{G,S}]\le Cd/m$.
Markov's inequality gives, with conditional failure probability
$O((\log n)^{-2})$ given $\mathcal F_{G,T}$,
\begin{equation}\label{eq:smoothed-seed-sampling-bound}
 \|r_G(u_*)\|\le C\log n/\sqrt d,
 \qquad \|\epsilon_G(u_*)\|\le C\log n\sqrt{d/m}.
\end{equation}
Together with $\|P_G-I_d\|_{\mathrm{op}}\le\ell_nd/m$, this implies
\[
 \|\widetilde v_a^G-u_a^G\|
 \le C\ell_n(d^{-1/2}+\sqrt{d/m})=o(1),
\]
so $\|\widetilde v_a^G\|\le2$ eventually.
Let $U_G=(\langle u_a^G,u_b^G\rangle)_{a,b}$ and
$U_{AB}=(\langle u_a^A,u_b^B\rangle)_{a,b}$ be the true seed Gram
matrices. Their diagonal entries are $1,1,\rho$, and every off-diagonal entry has conditional second
moment $1/d$ given $\mathcal L$. Since $K$ is fixed, Markov's inequality
bounds their total Frobenius distance from $I_K,I_K,\rho I_K$ by
$\ell_n/\sqrt d$ except with probability $O(\ell_n^{-2})$.
Combining this $\mathcal F_0$-measurable event with the seed approximation
proves \eqref{eq:seed-gram-accuracy}.

We next need both matrix variances of the cubic difference. Suppress $G$
and set $\eta=\widetilde f-x$, $D=D_3^G$. Define
$T_j:\bbR^d\to L^2(\mu;\bbR^d)$ by
$(T_jy)(u)=D(u)[e_j,y]$.
Equation~\eqref{eq:local-positive-interface} gives
$\|\sum_jT_j^*T_j\|_{\mathrm{op}}\le C\delta^2d^2$. For the other variance, let
$\mathscr S(A)=\sum_jT_j(Ae_j)$, so that
$\mathscr S\mathscr S^*=\sum_jT_jT_j^*$. The second-flattening
estimate \eqref{eq:seeded-second-flattening} of
Lemma~\ref{lem:positive-local-difference} therefore gives
\[
 \Big\|\sum_jT_jT_j^*\Big\|_{\mathrm{op}}
 =\sup_{\|A\|_{\mathrm F}=1}\E_u\|D[A]\|^2
 \le C\delta^2d^2,
\]
where $D[A]$ contracts two slots of $D$ against $A$.

The seed contraction argument applies to any
$\mathcal F_{G,T}$-measurable operators $T_1,\ldots,T_d$ from $\bbR^d$
to a Hilbert space with both variance operators bounded by $V_0^2$.
Write $T(y)=\sum_jy_jT_j$. We claim, with conditional probability
$1-o(1)$ uniformly on $\mathcal E_{\mathrm{ref},G}$, that
\begin{equation}\label{eq:seeded-contraction}
 \|T(\widetilde v_a^G)\|_{\mathrm{op}}^2
 \le C\ell_n^4V_0^2/d.
\end{equation}
Indeed, $\|T(y)\|_{\mathrm{op}}\le V_0\|y\|$, and covariance domination
$0\preceq M\preceq cI_d$ bounds both operator sums
$\sum_{j,k}M_{jk}T_j^*T_k$ and $\sum_{j,k}M_{jk}T_jT_k^*$ by
$cV_0^2$ in norm. All ranges lie in the span of the $d^2$ vectors
$T_je_k$. The Gaussian series inequality
\cite[Theorem~4.1.1]{tropp2015introduction} in this span, followed by the
Gaussian radial decomposition and $\|P_G\|_{\mathrm{op}}\le2$, gives
\[
 \E_{u_*}\|T(P_Gu_*)\|_{\mathrm{op}}^2
 =d^{-1}\E_\xi\|T(P_G\xi)\|_{\mathrm{op}}^2
 \le CV_0^2\log(2d)/d,
 \qquad \xi\sim\mathcal N(0,I_d).
\]
Markov's inequality bounds this component by $\ell_n^2V_0/\sqrt d$ with failure
probability $o(1)$, and \eqref{eq:smoothed-seed-sampling-bound} bounds
$\|T(r_G)\|_{\mathrm{op}}$ by $CV_0\log n/\sqrt d$.
For the sampling error, condition on $\mathcal F_{G,S}$ and set
\[
 Y(u)=\frac{\widetilde f_G(u)h_0(u_*,u)}{\lambda\sqrt d}.
\]
Then $\epsilon_G=m^{-1}\sum_{w\in D_G}(Y(u_w^G)-\E_uY)$, with
$\sup_u\|Y(u)\|\le C\sqrt d$ and $\E_u[YY^\top]\preceq CI_d$.
The centered summands after applying $T/m$ have norm at most
$CV_0\sqrt d/m$ and summed variances at most $CV_0^2/m$.
Rectangular matrix Bernstein
\cite[Theorem~6.1.1]{tropp2015introduction} therefore gives, with
conditional failure probability $n^{-1}$,
\[
 \|T(\epsilon_G)\|_{\mathrm{op}}
 \le CV_0(\sqrt{\log n/m}+\sqrt d\log n/m)
 \le\ell_nV_0/\sqrt d.
\]
The last step uses $d\le m$. Adding the three components proves
\eqref{eq:seeded-contraction}.

Apply it first to $T_jy=D[e_j,y]$, for which
$V_0^2=C\delta^2d^2$ by
\eqref{eq:local-positive-interface} and~\eqref{eq:seeded-second-flattening}.
Since $T(v)y=D_v(\cdot)y$, this gives the first bound of
\eqref{eq:seeded-two-port-bounds}. Next take
$T_j=A_3^G[e_j]$. Symmetry makes the norms of both variance operators equal to
$\|A_3^G\|_{1\to\mathrm F}^2\le C\zeta^2d^2/N$, so the same argument
gives the second bound. The factors $C\ell_n^4$ are absorbed by
$\delta_*^2=\ell_n^6\delta^2$ and $\zeta_*^2=\ell_n^6\zeta^2$.
The linear mean and covariance follow from the reference event.

For each $G$, intersect $\mathcal E_{\mathrm{ref},G}$ with the sampling
and contraction events for all $K$ seeds and the two operator families
to obtain $\mathcal E_G\in\mathcal F_G$. Intersect the two graph events
with the true-seed Gram event to obtain
$\mathcal E_{\mathrm{loc}}\in\mathcal F_0$.
The stated conditional probabilities follow by a finite union bound.
\end{proof}

\subsection{Cleaned matrix norm bound}\label{ss:coordinate-path-moments}

Conditional on $\mathcal F_0$, the estimated seed directions are fixed
and the trimmed reservoir pairs are independent. Expanding an oracle
Gram entry gives a sum over pairs of paths. Terms whose labels agree
at every level give the leading contribution. The remaining terms
either vanish by harmonic orthogonality or have fewer free coordinate
or label sums.

Figure~\ref{fig:cubic-path-cancellation} illustrates the orthogonality.
For four distinct labels, suppress the graph superscript, write
$x_j=x(u_j)=\sqrt d u_j$, and set
\[
 \widehat W_1=d^{-1}x_i^\top H_3(x_{r_1})[u_{s_a}]x_{r_2},\qquad
 \widehat W_2=d^{-1}x_i^\top H_3(x_{r_2})[u_{s_a}]x_{r_1}.
\]
Conditional on $\mathcal H=\sigma(u_i,u_{s_a},u_{r_2})$, these have
harmonic degrees three and one in $u_{r_1}$, respectively. By Lemma~\ref{lem:spherical-local-covariances},
$\E[\widehat W_1\widehat W_2\mid\mathcal H]=0$; the raw cubic matrix
would retain a degree-one component.

\begin{figure}[htbp]
\centering
\begin{tikzpicture}[
  font=\small, text=ink,
  vertex/.style={circle, draw=ink, fill=white, minimum size=7mm,
                 inner sep=0pt, line width=.8pt},
  highlighted/.style={vertex, draw=rggA, fill=rggA!12,
                      text=rggA, line width=1.2pt},
  seed/.style={rectangle, draw=rggB, fill=rggB!9, text=rggB,
               minimum size=7mm, inner sep=0pt, line width=.9pt},
  path edge/.style={draw=ink, line width=1pt},
  seed edge/.style={draw=rggB, dashed, line width=.9pt}
]
\node[vertex] (i1) at (0,0) {$i$};
\node[highlighted] (r11) at (1.7,0) {$r_1$};
\node[vertex] (r21) at (3.4,0) {$r_2$};
\node[seed] (s1) at (1.7,1.1) {$s_a$};
\draw[path edge] (i1) -- (r11) -- (r21);
\draw[seed edge] (s1) -- (r11);
\node[font=\footnotesize] at (1.7,-.6) {degree $3$};
\node[vertex] (i2) at (5.1,0) {$i$};
\node[vertex] (r22) at (6.8,0) {$r_2$};
\node[highlighted] (r12) at (8.5,0) {$r_1$};
\node[seed] (s2) at (6.8,1.1) {$s_a$};
\draw[path edge] (i2) -- (r22) -- (r12);
\draw[seed edge] (s2) -- (r22);
\node[font=\footnotesize] at (8.5,-.6) {degree $1$};
\end{tikzpicture}
\caption{The highlighted label is internal in one path and terminal in
the other. Its cubic harmonic and linear factors are orthogonal.
Dashed edges attach the seed.}
\label{fig:cubic-path-cancellation}
\end{figure}
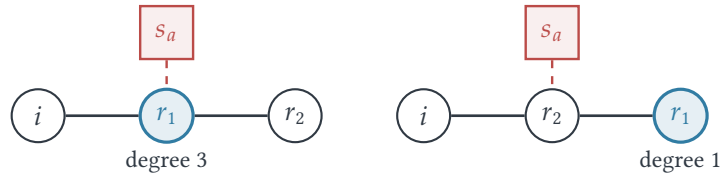

For a $K\times K$ matrix $M$, we index the tensor power $M^{\otimes(V-1)}$ by words $w,z\in\mathcal W=[K]^{V-1}$, so that
\[
 (M^{\otimes(V-1)})_{wz}=\prod_{t=1}^{V-1}M(w_t,z_t).
\]
With this indexing, the leading terms of the oracle Gram matrices are tensor powers of the seed Gram matrices $\Gamma_G,\Gamma_{AB}$ in \eqref{eq:matrix-Gamma}. Let $\gamma_{3,d}$ be as in \eqref{eq:gamma-three}, and let
$H=|R_A^\circ\cap R_B^\circ|$.

\begin{lem}\label{lem:clean-coordinate-paths}
Suppose that $\max_{G,a}\|\widetilde v_a^G\|\le2$. Conditional on $\mathcal F_0$,
\begin{align}
 \E\Big[\frac{(\mathcal B_G^{0,\circ})^\top\mathcal B_G^{0,\circ}}{d}\Big|\mathcal F_0\Big]
 &=\frac{(|R_G^\circ|)_V}{(N)_V}
     \gamma_{3,d}(1)^{V-1}\Gamma_G^{\otimes(V-1)}+\Xi_G,
       \label{eq:clean-within-gram}\\
 \E\Big[\frac{(\mathcal B_A^{0,\circ})^\top\mathcal B_B^{0,\circ}}{d}\Big|\mathcal F_0\Big]
 &=\frac{(H)_V}{(N)_V}\rho\gamma_{3,d}(\rho)^{V-1}
     \Gamma_{AB}^{\otimes(V-1)}+\Xi_{AB},
       \label{eq:clean-cross-gram}
\end{align}
and there is a constant $C$ such that, whenever $N,d\ge CV^C$, every entry of $\Xi_G,\Xi_{AB}$ has absolute value at most $C^VV^C/d$, and every entry of the three random matrices $(\mathcal B_A^{0,\circ})^\top\mathcal B_A^{0,\circ}/d$, $(\mathcal B_B^{0,\circ})^\top\mathcal B_B^{0,\circ}/d$ and $(\mathcal B_A^{0,\circ})^\top\mathcal B_B^{0,\circ}/d$ has conditional variance, given $\mathcal F_0$, at most $C^VV^C(N^{-1}+d^{-1})$.

Moreover, let $G\in\{A,B\}$ and let $\mathcal R\subseteq R_G$ be $\mathcal L$-measurable. If $\max_a\|\widetilde v_a^G\|\le2$, then \eqref{eq:clean-within-gram} and the same bounds on the remainder and on the conditional variances of the Gram entries hold conditionally on $\mathcal F_G$, with $\mathcal B_G^{0,\circ}$ replaced by $(b_w(x;\mathcal R))_{w\in\mathcal W}$ and $|R_G^\circ|$ replaced by $|\mathcal R|$.
\end{lem}

\begin{proof}
Condition on $\mathcal F_0$ and fix a realization with $\max_{G,a}\|\widetilde v_a^G\|\le2$. All expectations and variances in this proof are conditional on $\mathcal F_0$. By the choice \eqref{eq:R-A,B-circ} of the trimmed reservoirs, the pairs $(u_r^A,u_r^B)$ with $r\in R_A^\circ\cup R_B^\circ$ are independent $\rho$-pairs.

\textit{Step 1: local covariances.} Let $\vartheta\in\{1,\rho\}$,
let $(u,u')$ be a $\vartheta$-pair, and set $(x,y)=(x(u),x(u'))$.
Lemma~\ref{lem:spherical-local-covariances} gives the spherical Wick
formula \eqref{eq:clean-spherical-pairings} and the local covariances
\eqref{eq:clean-local-covariances}. In the coordinate expansion of the
cubic covariance, there are six crossing pairings and nine trace
pairings; each trace term carries a factor $O(d^{-1})$. We use this
expansion below.

\textit{Step 2: diagram expansion and the deficit identity.} Let $k\in\{1,2\}$. A product of $k$ normalized Gram entries expands into a sum over $2k$ injective paths through the reservoirs, with denominator $(N)_V^kd^{kV}$. An \emph{index} is a coordinate summed along a path edge or in the outer inner product of a Gram entry; the $k$ indices of the second kind are called \emph{external}. There are $E=k(2V-1)$ indices. A \emph{vertex} of the expansion is a reservoir label together with all its occurrences in the $2k$ paths, in either graph, and its multiplicity is the number of these occurrences. Integrate independently over the latent pairs at distinct vertices. A vertex of multiplicity one contributes zero by centering, and a vertex with one terminal and one cubic occurrence contributes zero by \eqref{eq:clean-local-covariances}.

Each local moment has degree at most twelve. By
Lemma~\ref{lem:spherical-local-covariances}, it is a combination of a
bounded number of coordinate pairings with bounded coefficients. A
choice of coordinate pairings at all vertices is a \emph{diagram}. In a
diagram, the indices form closed cycles, each contributing a factor
$d$, and chains joining two seed slots, each contributing an inner
product of estimated seed directions of absolute value at most four.

For a diagram, let $u$ be its number of distinct reservoir labels, $c_\ell$ its number of cycles of length $\ell$, and $L_{\mathrm{seed}}$ the total number of indices on chains ending at seeds. Define $q=\sum_{\ell\ge1}c_\ell$, $r=kV-u$, and $s=kV-q$. Every vertex has multiplicity at least two and there are $2kV$ occurrences, so $r\ge0$. A cycle of length one is an index whose two endpoints are paired with each other; since the paths are injective, only an external index can form such a cycle, so $c_1\le k$. Counting indices gives $E=L_{\mathrm{seed}}+\sum_\ell\ell c_\ell$, and $E+k=2kV$, so
\begin{equation}\label{eq:clean-wire-deficit}
 2s=(k-c_1)+L_{\mathrm{seed}}+\sum_{\ell\ge3}(\ell-2)c_\ell\ge0.
\end{equation}
In particular, since $\ell\le3(\ell-2)$ for $\ell\ge3$, at most $6s$ indices lie on cycles of length at least three or on chains ending at seeds. If $m_v$ is the multiplicity of the vertex $v$, then
\[
 \sum_v(m_v-2)=2r,
 \quad
 \sum_{v:m_v\ge3}m_v\le6r.
\]

\textit{Step 3: counting diagrams.} Let $\mathcal X$ be the set consisting of the occurrences at vertices of multiplicity at least three, the endpoints of indices on cycles of length at least three and on chains of positive length ending at seeds, and the endpoints of the external indices, together with all other occurrences of the vertices involved. The multiplicity bound and \eqref{eq:clean-wire-deficit} give
\[
 |\mathcal X|\le C(r+s+1),
 \quad
 \#\{\text{maximal path intervals outside }\mathcal X\}\le|\mathcal X|+2k.
\]
Outside $\mathcal X$, every vertex has multiplicity two and every index lies on a cycle of length two, formed by two parallel indices between the same two vertices. A trace pairing at a cubic occurrence either joins one of its two matrix slots to its seed slot or joins its two matrix slots. The first creates a chain of positive length ending at a seed. The second cannot close a cycle of length two, because the two neighbors of the occurrence in its injective path are distinct vertices. Hence no trace pairing occurs outside $\mathcal X$.

Along each maximal path interval outside $\mathcal X$, the partner of an occurrence, that is, the other occurrence of its vertex, moves along another path with a fixed orientation. Encode a diagram by $\mathcal X$, the partition of $\mathcal X$ by vertices, the starting partner and the orientation of each interval, and the local pairings. Let $\mathfrak D_{k,r,s}$ be the set of diagrams with parameters $k,r,s$. These choices give
\begin{equation}\label{eq:clean-diagram-count}
 |\mathfrak D_{k,r,s}|
 \le C^V\sum_{m\le C(r+s+1)}\binom{2kV}{m}(2kV)^m(4kV)^{m+2k}
 \le C^VV^{C(r+s+1)}.
\end{equation}
A diagram has at most $N^u$ labelings, since the label of each of its vertices lies in $R_A^\circ$ or in $R_B^\circ$, and both sets have at most $N$ elements. Let $t$ be the number of trace terms at cubic vertices of multiplicity two. The normalized absolute contribution of one diagram, summed over its labelings, is at most
\begin{equation}\label{eq:clean-diagram-weight}
 \frac{C^V N^u d^{q-t}}{(N)_V^k d^{kV}}
 \le C^V\exp(CkV^2/N)N^{-r}d^{-s-t}
 \le C^V N^{-r}d^{-s}.
\end{equation}
The last inequality uses $N\ge CV^C$ with exponent $C\ge2$, and the factor $C^V$ includes the coefficients at vertices of multiplicity at least three.

\textit{Step 4: leading term.} Let $k=1$ and consider a nonzero diagram.
Every vertex has at least two occurrences by centering and at most two
by path injectivity. Thus every vertex has multiplicity two and $r=0$. Equality $s=0$ in \eqref{eq:clean-wire-deficit} forces $c_1=1$, $L_{\mathrm{seed}}=0$ and no cycle of length at least three: the external index forms a cycle of length one, and the other $2V-2$ indices form $V-1$ cycles of length two. Then also $t=0$, since a trace pairing would either create a chain of positive length ending at a seed or prevent a cycle of length two, as in Step 3. Starting from the external index and moving along the paths, the vertex labels of the two paths agree at every level, and at each level the two seed slots are paired with each other. Thus, for $G,G'\in\{A,B\}$, with $\vartheta=1$ if $G=G'$ and $\vartheta=\rho$ otherwise, the diagrams with $s=0$ contribute
\[
 \frac{(|R_G^\circ\cap R_{G'}^\circ|)_V}{(N)_V}
 \vartheta\gamma_{3,d}(\vartheta)^{V-1}
 \prod_{a=1}^{V-1}
 \langle\widetilde v_{w_a}^G,\widetilde v_{z_a}^{G'}\rangle
\]
to the conditional mean of $\langle b_w(x;R_G^\circ),b_z(x;R_{G'}^\circ)\rangle/d$. Every other diagram has $s\ge1$, since $s$ is an integer. Summing \eqref{eq:clean-diagram-weight} over the diagrams counted in \eqref{eq:clean-diagram-count} gives
\[
 |(\Xi_G)_{wz}|+|(\Xi_{AB})_{wz}|
 \le C^V V^C\sum_{s\ge1}(V^C/d)^s
 \le C^V V^C/d,
\]
after enlarging the constants so that $V^C/d\le1/2$. By the entrywise formula for tensor powers, the leading terms are the matrices in \eqref{eq:clean-within-gram} and \eqref{eq:clean-cross-gram}.

\textit{Step 5: variance.} Write a normalized Gram entry as $Q=\sum_\alpha X_\alpha$, where $\alpha$ runs over the pairs of injective tuples forming its two paths and $X_\alpha$ includes the normalization. Let $\mathcal O$ be the set of pairs $(\alpha,\beta)$ whose sets of reservoir labels intersect. Since $X_\alpha$ and $X_\beta$ are independent when these sets are disjoint,
\[
 \operatorname{Var}(Q)=\sum_{(\alpha,\beta)\in\mathcal O}
 \big(\E[X_\alpha X_\beta]-\E[X_\alpha]\E[X_\beta]\big).
\]
Apply Steps 2 and 3 with $k=2$ to $\E[X_\alpha X_\beta]$. A diagram for a pair in $\mathcal O$ cannot have $r=s=0$: the two external cycles of length one, propagated as in Step 4, would make the two paths of each Gram entry identical, and then $u=2V$ would make the label sets of $\alpha$ and $\beta$ disjoint. Hence
\begin{align*}
 \sum_{(\alpha,\beta)\in\mathcal O}|\E[X_\alpha X_\beta]|
 &\le C^V V^C\sum_{\substack{r,s\ge0\\r+s\ge1}}
       (V^C/N)^r(V^C/d)^s\\
 &=C^V V^C\Big[\frac1{(1-V^C/N)(1-V^C/d)}-1\Big]\\
 &\le C^V V^C(N^{-1}+d^{-1}).
\end{align*}
For the product of means, each diagram of $\E[X_\alpha]$ with nonzero weight has $V$ vertices by Step 4. An overlap identifies one of at most $V^2$ pairs of vertices, which leaves at most $V^2N^{2V-1}$ labelings. Applying the bounds for $k=1$ to both factors gives
\[
 \sum_{(\alpha,\beta)\in\mathcal O}|\E[X_\alpha]\E[X_\beta]|
 \le\frac{C^V V^C}{N}
 \sum_{s_1,s_2\ge0}(V^C/d)^{s_1+s_2}
 \le C^V V^C/N.
\]
The two bounds prove the variance estimate.

For the last assertion, condition on $\mathcal F_G$ instead, and let $\mathcal R\subseteq R_G$ be $\mathcal L$-measurable. The latent vectors $u_r^G$, $r\in\mathcal R$, are independent with law $\mu$, the estimated seed directions of graph $G$ are fixed, and the counts of labelings in \eqref{eq:clean-diagram-weight} use only $|\mathcal R|\le N$. The same computation with $\vartheta=1$ gives the formula and the variance bound for the Gram matrix within graph $G$.
\end{proof}

\begin{proof}[Proof of Lemma~\ref{lem:oracle-gram-concentration}]
Let $\mathcal E_{\mathrm{loc}}$ be as in
Lemma~\ref{lem:seeded-local}, and take $A_n=C^VV^C$ to bound the
constants in Lemma~\ref{lem:clean-coordinate-paths}. Since
$\P[\mathcal A^c]=o(1)$, the event
$\mathcal A\cap\mathcal E_{\mathrm{loc}}$ has probability $1-o(1)$.

First control the reservoir labels. Conditional on
$\mathcal L_-=\sigma(T_A,D_A,S_A,T_B,D_B,S_B,R_A)$, on $\mathcal A$,
$H=|R_A\cap R_B|$ is hypergeometric: $R_B$ is a uniform $N$-subset of an
eligible set of size $n-2m-K$, containing between $N-2m$ and $N$
elements of $R_A$. Its mean is $c^2n+O(m+1)$ and variance $O(n)$.
Chebyshev's inequality gives $\mathcal E_{\mathrm{lab}}\in\mathcal L$ with
$\P[\mathcal A\setminus\mathcal E_{\mathrm{lab}}]=o(1)$ on which
$|H-c^2n|\le C(m+1)+\sqrt n\log n$. Since trimming preserves the
intersection and $N-2m\le|R_G^\circ|\le N$, eventually
\begin{equation}\label{eq:retained-pool-factors}
 \frac{(|R_G^\circ|)_V}{(N)_V}=1+O(Vm/n),\qquad
 \frac{(H)_V}{(N)_V}\ge(c/2)^V.
\end{equation}

Condition on $\mathcal F_0$, on
$\mathcal A\cap\mathcal E_{\mathrm{loc}}\cap\mathcal E_{\mathrm{lab}}$,
and let
\[
 M_G=\frac{(|R_G^\circ|)_V}{(N)_V}
      \gamma_{3,d}(1)^{V-1}\Gamma_G^{\otimes(V-1)},\qquad
 M_{AB}=\frac{(H)_V}{(N)_V}\rho\gamma_{3,d}(\rho)^{V-1}
                         \Gamma_{AB}^{\otimes(V-1)}.
\]
For each oracle Gram matrix, Lemma~\ref{lem:clean-coordinate-paths}
bounds the conditional second moment of its centered version in
Frobenius norm by $J^2A_n(N^{-1}+d^{-1})$. Its conditional mean differs
from the corresponding leading term by at most $JA_n/d$ in
Frobenius norm. The parameter bounds
\eqref{eq:intermediate-parameter-scales}, verified in the proof of
Lemma~\ref{lem:direction-transfer}, give
\[
 JA_n/d=o(t_n(\log n)^{-2}),\qquad
 J^2A_n(\log n)^4t_n^{-2}(N^{-1}+d^{-1})=o(1).
\]
Applying Markov's inequality conditionally and then integrating gives,
with probability $1-o(1)$,
\begin{align*}
 \Big\|(\mathcal B_G^{0,\circ})^\top\mathcal B_G^{0,\circ}/d-M_G
      \Big\|_{\mathrm{op}}&\le t_n(\log n)^{-2},\quad G=A,B,\\
 \Big|\operatorname{tr}\big((\mathcal B_A^{0,\circ})^\top
             \mathcal B_B^{0,\circ}/d-M_{AB}\big)\Big|
      &\le Jt_n(\log n)^{-2}.
\end{align*}
For the trace, use
$\operatorname{Var}(\sum_wZ_w\mid\mathcal F_0)
\le J\sum_w\operatorname{Var}(Z_w\mid\mathcal F_0)$.

Equation~\eqref{eq:seed-gram-accuracy} gives
$\|\Gamma_G^{\otimes(V-1)}-I_J\|_{\mathrm{op}}
\le(1+\varepsilon_n)^{V-1}-1=o(1)$.
Together with \eqref{eq:gamma-three} and~\eqref{eq:retained-pool-factors},
this shows
$\|(\mathcal B_G^{0,\circ})^\top\mathcal B_G^{0,\circ}/d-I_J
\|_{\mathrm{op}}=o(1)$.
For the cross term,
\begin{align*}
 \rho\operatorname{tr}M_{AB}
 &=\frac{(H)_V}{(N)_V}\rho^2\gamma_{3,d}(\rho)^{V-1}
                     (\operatorname{tr}\Gamma_{AB})^{V-1}\\
 &=(1+o(1))J\frac{(H)_V}{(N)_V}\rho^{4V-2}
 \ge(1-o(1))Jt_n.
\end{align*}
Here $\rho>0$ is fixed,
$\operatorname{tr}\Gamma_{AB}=K\rho+O(K\varepsilon_n)$, and
$V\varepsilon_n+V/d=o(1)$. The trace error is $o(Jt_n)$, so the claimed
cross-trace bound follows as well.
\end{proof}

\subsection{Approximation error of cleaned matrices}
\label{ss:reservoir-coordinate-transfer}\label{ss:reservoir-label-restriction}
\label{ss:observable-direction-grams}

Let $\mathcal B_G^\circ=(b_w(\widetilde f_G;R_G^\circ))_{w\in\mathcal W}$
be the path matrix after trimming the reservoir. We use the decomposition
\[
 \mathcal B_G-\mathcal B_G^{0,\circ}
 =(\mathcal B_G-\mathcal B_G^\circ)
       +(\mathcal B_G^\circ-\mathcal B_G^{0,\circ}).
\]
The coordinate replacement is analyzed under $\mathcal F_0$, which
leaves the trimmed reservoir independent. Trimming is analyzed under
$\mathcal F_G$, which leaves the full reservoir independent.
Throughout, $\delta_*,\zeta_*$ and $\mathcal E_G$ are as in
Lemma~\ref{lem:seeded-local}.

In a second-moment expansion, a vertex appearing twice is controlled
by a local second moment. At a vertex appearing only once, oracle
factors have mean zero, whereas the mean of a difference factor
supplies $N^{-1/2}$. This compensates for the additional factor
$N^{1/2}$ in the label count.

\begin{lem}\label{lem:seeded-path-transfer}
Suppose that $N,d\ge CV^C$ and that $V^C(\delta_*+\zeta_*)$ is sufficiently small. On $\mathcal E_G$, for every $\mathcal L$-measurable $\mathcal R\subseteq R_G$ and every $w\in\mathcal W$,
\[
 \E\Big[\frac{\|b_w(\widetilde f_G;\mathcal R)-b_w(x;\mathcal R)\|^2}{d}\Big|\mathcal F_G\Big]
 \le C^VV^C(\delta_*+\zeta_*)^2.
\]
The same bound holds with $\mathcal R=R_G^\circ$ and conditioning on $\mathcal F_0$.
\end{lem}

\begin{proof}
\textit{Step 1: expansion into difference patterns.} Fix $G$ and condition on $\mathcal F_G$, on $\mathcal E_G$; all expectations in the proof are conditional on $\mathcal F_G$. Let $\mathcal R\subseteq R_G$ be $\mathcal L$-measurable. The latent vectors $u_r^G$, $r\in\mathcal R$, are independent with law $\mu$. Abbreviate $f=\widetilde f_G$ and suppress $\mathcal R$ in the path sums. For $r\in\mathcal R$, write $x_r=x(u_r^G)$ and $f_r=f(u_r^G)$. For $t<V$, let
\[
 A_{t,0}(r)=H_3(x_r)[\widetilde v_{w_t}^G],\quad
 A_{t,1}(r)=D_{\widetilde v_{w_t}^G}(u_r^G),
\]
and at the terminal level let $a_{V,0}(r)=x_r$ and $a_{V,1}(r)=f_r-x_r$. We call $A_{t,0},a_{V,0}$ oracle factors and $A_{t,1},a_{V,1}$ difference factors. For $\sigma\in\{0,1\}^V$, write $|\sigma|=\sum_t\sigma_t$ and
\[
 P_\sigma(\mathbf r)=A_{1,\sigma_1}(r_1)\cdots A_{V-1,\sigma_{V-1}}(r_{V-1})a_{V,\sigma_V}(r_V).
\]
With $\mathcal R_{\ne}^V$ the set of injective tuples in $\mathcal R$, expanding the products gives
\begin{align*}
 b_w(f)-b_w(x)
 &=\frac{1}{\sqrt{(N)_Vd^{V-1}}}
   \sum_{\mathbf r\in\mathcal R_{\ne}^V}\sum_{\sigma\ne0}P_\sigma(\mathbf r),\\
 \frac1d\E[\|b_w(f)-b_w(x)\|^2]
 &=\frac{1}{(N)_Vd^V}
   \sum_{\mathbf r,\mathbf r'\in\mathcal R_{\ne}^V}
   \sum_{\sigma,\tau\ne0}
   \E[\langle P_\sigma(\mathbf r),P_\tau(\mathbf r')\rangle].
\end{align*}
Call $\mathbf r$ and $\mathbf r'$ the two paths. As in the proof of Lemma~\ref{lem:clean-coordinate-paths}, a vertex is a label together with its occurrences. Each vertex occurs at most once in each path, hence once or twice in total, and its two occurrences, if there are two, involve the same latent vector. Integrate first over the vertices that occur once. An oracle factor at such a vertex has mean zero, and a difference factor is replaced by its mean, which is zero at the terminal level. For a surviving term, let $k$ be the number of vertices of multiplicity one, $b$ the number of difference factors at vertices of multiplicity two, and $u$ the number of distinct labels. Then
\[
 k+b=|\sigma|+|\tau|\ge2,
 \quad 2V=k+2(u-k),
 \quad u=V+k/2.
\]
Mark every occurrence that carries a difference factor, together with the other occurrence of its vertex, when there is one. There are at most $k+2b$ marked occurrences. The maximal intervals of marked occurrences along each path are the \emph{marked components}.

\textit{Step 2: component energy.} For a marked component $\mathcal C$, let $v_{\mathcal C}$ be its number of occurrences, $I_{\mathcal C}=v_{\mathcal C}-1$ its number of internal indices, and $B_{\mathcal C}$ its number of open indices, which join it to the rest of the diagram. Its boundary tensor $K_{\mathcal C}$ is a product of matrices if $B_{\mathcal C}=2$, and a product of matrices applied to a terminal vector if $B_{\mathcal C}=1$. Let $k_{\mathcal C}$ and $b_{\mathcal C}$ be the numbers of difference factors of $\mathcal C$ at vertices of multiplicity one and two, respectively. For each matrix $A$ in a component, Lemmas~\ref{lem:seeded-local} and~\ref{lem:spherical-local-covariances} give
\[
 \E[A^\top A]\preceq d\alpha^2I_d,
 \quad
 \alpha^2=\begin{cases}
 C,&A\text{ is an oracle factor},\\
 \delta_*^2,&A\text{ is a difference factor at a vertex of multiplicity two},\\
 \zeta_*^2/N,&A\text{ is the mean of a difference factor}.
 \end{cases}
\]
If $Q$ is a product of factors at later vertices of the same path, then $Q$ is independent of $A$, because the path is injective, and
\[
 \E[\|AQ\|_{\mathrm F}^2]
 =\E[\operatorname{tr}(Q^\top\E[A^\top A]Q)]
 \le d\alpha^2\E[\|Q\|_{\mathrm F}^2].
\]
When $B_{\mathcal C}=2$, start with $Q=I_d$ and iterate over the $v_{\mathcal C}$ matrices. When $B_{\mathcal C}=1$, start with the terminal bound $\E[\|a\|^2]\le d$ for an oracle factor or $\E[\|a\|^2]\le\delta_*^2d$ for a difference factor, and iterate over the remaining $v_{\mathcal C}-1$ matrices. Since $I_{\mathcal C}+B_{\mathcal C}$ equals $v_{\mathcal C}+1$ or $v_{\mathcal C}$, respectively,
\begin{equation}\label{eq:seeded-component-energy}
 \E[\|K_{\mathcal C}\|_{\mathrm F}^2]
 \le C^{v_{\mathcal C}}d^{I_{\mathcal C}+B_{\mathcal C}}
       (\zeta_*^2/N)^{k_{\mathcal C}}\delta_*^{2b_{\mathcal C}}.
\end{equation}

\textit{Step 3: index counting.} At each unmarked vertex both occurrences carry oracle factors. Expand their covariance into coordinate pairings as in the proof of Lemma~\ref{lem:clean-coordinate-paths}; an unmarked vertex with one cubic and one terminal occurrence contributes zero. Let $I=\sum_{\mathcal C}I_{\mathcal C}$ and $B=\sum_{\mathcal C}B_{\mathcal C}$, let $h$ be the number of index chains that leave a marked component and return to it, and let $\ell$ be the number of closed index cycles. A trace term at a cubic vertex of multiplicity two pairs two slots within each of its two occurrences and has coefficient $O(d^{-1})$. Charge $d^{-1/2}$ to each of these two internal pairings, and let $J_{\mathrm{tr}}$ be the number of such pairings.

There are $E=2V-1$ indices, of which $I$ are internal to marked components. An index chain between two distinct components uses at least two indices, a chain that returns to its component at least three, and a chain from a component to a seed at least one. Indeed, a crossing pairing at an unmarked vertex moves to the other path, so a chain cannot return to its component through a single unmarked vertex. Trace pairings do not shorten any of these chains, and only the external index can give one exception. Hence $E+1\ge I+B+h+2\ell$, and a fortiori
\[
 E+1+J_{\mathrm{tr}}\ge I+B+h+2\ell,
 \quad
 s:=\frac{E+1+J_{\mathrm{tr}}-I-B-h-2\ell}{2}\ge0.
\]
For one fixed choice of coordinate pairings at the unmarked vertices, let $Z$ be the resulting scalar contraction. Each closed cycle contributes $d$, and contracting the two open indices of one component with each other costs at most $\sqrt d$ relative to its Frobenius norm. The trace terms contribute $d^{-J_{\mathrm{tr}}/2}$. Contracting the remaining shared indices by Cauchy--Schwarz gives
\[
 |Z|\le C^Vd^{\ell+(h-J_{\mathrm{tr}})/2}\prod_{\mathcal C}\|K_{\mathcal C}\|_{\mathrm F}.
\]
Distinct components of one path involve disjoint sets of vertices and are therefore independent. Hence the Cauchy--Schwarz inequality between the two paths $p=1,2$ and \eqref{eq:seeded-component-energy} give
\begin{align*}
 \E\Big[\prod_{\mathcal C}\|K_{\mathcal C}\|_{\mathrm F}\Big]
 &\le\prod_{p=1}^2
   \Big(\E\Big[\prod_{\mathcal C\subset p}\|K_{\mathcal C}\|_{\mathrm F}^2\Big]\Big)^{1/2}
 =\prod_{\mathcal C}(\E[\|K_{\mathcal C}\|_{\mathrm F}^2])^{1/2},\\
 |\E[Z]|
 &\le C^V\zeta_*^k\delta_*^bN^{-k/2}
   d^{(I+B+h+2\ell-J_{\mathrm{tr}})/2}
 =C^V\zeta_*^k\delta_*^bN^{-k/2}d^{V-s},
\end{align*}
where the last equality uses $E+1=2V$.

\textit{Step 4: counting and summation.} We count the diagrams by their exceptional occurrences. Let $q_{\mathrm{open}}$ be the total length of all open index chains, including those ending at seeds, and let $r_1,\ldots,r_\ell$ be the lengths of the closed cycles. Since at most one cycle has length one and $h\le B/2$,
\[
 E=I+q_{\mathrm{open}}+\sum_{j=1}^{\ell}r_j,\quad
 J_{\mathrm{tr}}+q_{\mathrm{open}}+\sum_{j:r_j\ge3}(r_j-2)
 \le 2s+B+h\le2s+3B/2.
\]
Since $r\le3(r-2)$ for $r\ge3$, the total number of indices on cycles of length at least three, on open chains, and at trace pairings is $O(s+B)$. Remove their endpoints, all marked occurrences, and the endpoints of the external index, together with all other occurrences of the vertices involved. Since $B\le2(k+2b)$, at most $C(k+b+s+1)$ occurrences are removed. Every remaining index lies on a cycle of length two, so a starting partner and an orientation determine the partners along each remaining maximal path interval.

An encoding of a diagram consists of the removed set, its partition by vertices, and these starting partners and orientations; the local covariance choices and the positions of the difference factors cost a further factor $C^V$. Let $\mathfrak D_{k,b,s}$ be the set of diagrams, that is, of identifications among coordinate indices and among vertex labels, with parameters $k,b,s$. For a fixed, sufficiently large constant $C'$,
\[
 \#\mathfrak D_{k,b,s}
 \le C^V\sum_{m\le C(k+b+s+1)}
       \binom{2V}{m}(2V)^m(4V)^{m+2}
 \le C^VV^{C'(k+b+s+1)}.
\]
Each diagram has at most $N^{V+k/2}$ labelings, since $u=V+k/2$ and $|\mathcal R|\le N$. Together with the bound on $|\E[Z]|$ and the normalization, each diagram contributes at most
\[
 \frac{N^{V+k/2}}{(N)_Vd^V}
 C^V\zeta_*^k\delta_*^bN^{-k/2}d^{V-s}
 \le C^Ve^{CV^2/N}\zeta_*^k\delta_*^bd^{-s}.
\]
Finally, $s\in\frac12\bbZ_{\ge0}$. Let $a=V^{C'}(\delta_*+\zeta_*)$ and $q=V^{C'/2}/\sqrt d$. Enlarging the fixed exponent in the hypotheses makes $a,q\le1/2$ and absorbs $e^{CV^2/N}$. Summing over diagrams and using $\sum_{k+b=t}x^ky^b\le(x+y)^t$ in the second inequality gives
\begin{align*}
 \frac1d\E[\|b_w(f)-b_w(x)\|^2]
 &\le C^VV^{C'}
   \sum_{\substack{k,b\ge0\\k+b\ge2}}
   (V^{C'}\zeta_*)^k(V^{C'}\delta_*)^b
   \sum_{j\ge0}q^j\\
 &\le C^VV^{C'}\frac{a^2}{(1-a)(1-q)}
 \le C^VV^C(\delta_*+\zeta_*)^2.
\end{align*}

For the second assertion, condition on $\mathcal F_0$, on $\mathcal E_G$, and take $\mathcal R=R_G^\circ$. The rotated spectral embedding $\widetilde f_G$ and the estimated seed directions of graph $G$ are again fixed, the estimates over a fresh point from Lemma~\ref{lem:seeded-local} are unchanged, and the latent vectors $u_r^G$, $r\in R_G^\circ$, are independent with law $\mu$. Every step above therefore holds under this conditioning.
\end{proof}

\begin{cor}\label{lem:approx-cB-0-circ}
Under the conditions of Lemma~\ref{lem:seeded-path-transfer}, on $\mathcal E_G$,
\[
 \E\Big[\frac{\|\mathcal B_G^\circ-\mathcal B_G^{0,\circ}\|_{\mathrm F}^2}{d}\Big|\mathcal F_0\Big]
 \le C^VV^C J(\delta_*+\zeta_*)^2.
\]
\end{cor}

\begin{proof}
The squared Frobenius norm is the sum of the squared Euclidean norms of the $J$ column differences. Sum the second assertion of Lemma~\ref{lem:seeded-path-transfer}, conditional on $\mathcal F_0$, over $w\in\mathcal W$.
\end{proof}

Trimming requires a restriction estimate for kernels that need not be
centered. The following elementary inequality applies after symmetrizing
each path sum over its unordered support.

\begin{lem}\label{lem:numerical-restriction}
Let $X_1,\ldots,X_N$ be independent with a common distribution, and let
$g$ be a square-integrable symmetric kernel of $V$ variables with values
in a real Hilbert space. For deterministic weights $0\le w(S)\le1$, set
\[
 U=\sum_{|S|=V}g(X_S),\qquad
 U_w=\sum_{|S|=V}w(S)g(X_S),
\]
where the sums run over the $V$-subsets of $[N]$, and let
$p_w=\binom{N}{V}^{-1}\sum_{|S|=V}w(S)$. Then
\begin{equation}\label{eq:weighted-kernel-restriction}
 \E[\|U_w\|^2]\le p_w\E[\|U\|^2].
\end{equation}
In particular, for a fixed $D\subseteq[N]$ and $U_{-D}=\sum_{|S|=V,S\cap D=\varnothing}g(X_S)$,
\begin{equation}\label{eq:numerical-restriction}
 \E[\|U-U_{-D}\|^2]
 \le\frac{V|D|}{N}\E[\|U\|^2].
\end{equation}
The same inequalities hold conditionally on a $\sigma$-field $\mathcal H$ if $g,w,D$ are $\mathcal H$-measurable and, given $\mathcal H$, the $X_i$ are independent with a common conditional law.
\end{lem}

\begin{proof}
For two supports $S,T$, conditional independence given
$\mathcal H_{S\cap T}=\sigma(X_i:i\in S\cap T)$ and symmetry give
\[
 \E\langle g(X_S),g(X_T)\rangle
 =\E\|\E[g(X_S)\mid\mathcal H_{S\cap T}]\|^2\ge0.
\]
By exchangeability, its sum over $T$ is independent of $S$ and equals
$\binom{N}{V}^{-1}\E\|U\|^2$. Expand $\E\|U_w\|^2$ and use
$w(S)w(T)\le w(S)$ to prove \eqref{eq:weighted-kernel-restriction}.
For \eqref{eq:numerical-restriction}, take
$w(S)=\mathbf1\{S\cap D\ne\varnothing\}$ and use
$p_w\le V|D|/N$.
For the conditional version, perform the same calculation given
$\mathcal H$, adjoining it to $\mathcal H_{S\cap T}$.
\end{proof}

\begin{lem}\label{lem:approx-cB-circ}
Under the conditions of Lemma~\ref{lem:seeded-path-transfer}, on $\mathcal E_G$,
\begin{equation}\label{eq:full-reservoir-second-moment}
 \E\Big[\frac{\|\mathcal B_G\|_{\mathrm F}^2}{d}\Big|\mathcal F_G\Big]
 \le C^VV^C J,
 \quad
 \E\Big[\frac{\|\mathcal B_G-\mathcal B_G^\circ\|_{\mathrm F}^2}{d}\Big|\mathcal F_G\Big]
 \le C^VV^C J\frac{Vm}{n}.
\end{equation}
Consequently,
\begin{equation}\label{eq:opposite-pool-restriction}
 \E\Big[\mathbf1_{\mathcal A\cap\mathcal E_G}
       \frac{\|\mathcal B_G-\mathcal B_G^\circ\|_{\mathrm F}^2}{d}\Big]
 \le C^VV^C J\frac{Vm}{n}.
\end{equation}
\end{lem}

\begin{proof}
Condition on $\mathcal F_G$, on $\mathcal E_G$.
Lemma~\ref{lem:clean-coordinate-paths} applies to $R_G$; each diagonal
leading term is at most $4^{V-1}$ because the seed norms are at most two.
Sum over words and use Lemma~\ref{lem:seeded-path-transfer} to obtain the
first bound in \eqref{eq:full-reservoir-second-moment}.
Group all columns of $\mathcal B_G/\sqrt d$ by unordered supports:
summing their orderings gives a fixed symmetric matrix-valued kernel,
and the $N$ reservoir points are conditionally independent with law
$\mu$. The deleted set $D=R_G\setminus R_G^\circ$ is
$\mathcal L$-measurable and has size at most $2m$.
Apply Lemma~\ref{lem:numerical-restriction} in Frobenius norm, conditional
on $\mathcal F_G$, and use $N\asymp n$ to obtain the second bound.
Multiplication by $\mathbf1_{\mathcal A\cap\mathcal E_G}$ and integration
give \eqref{eq:opposite-pool-restriction}.
\end{proof}

\begin{proof}[Proof of Lemma~\ref{lem:direction-transfer}]
Let $\mathcal E_A,\mathcal E_B,\mathcal E_{\mathrm{loc}}$ be as in
Lemma~\ref{lem:seeded-local}, and take $A_n=C^VV^C$ to bound the constants
in the oracle and transfer lemmas. We first choose common parameters
for the four lemmas whose proofs were deferred from
\S\ref{ss:idealized-scores}.

Choose $a_1$ so that $V,J,t_n^{-1},A_n\le(\log n)^{a_1}$, and
$\kappa_1$ so that $(\log n)^{\kappa_1}$ bounds all fixed logarithmic
losses, including $\ell_n^{10}$ and that in $\varepsilon_n^2$. Fix
\[
 b>10a_1+100,\quad
 L>\max\{14,b+\kappa_1+10a_1+100\},\quad
 L_0>2b+\kappa_1+10a_1+100.
\]
The four underlying scales satisfy
\begin{align*}
 d/m \le C(\log n)^{b-L_0}, \quad
 d^{-1}\le C(\log n)^{-L},\quad 
 N/(md)\le C(\log n)^{b-L}, \quad
 Nd/m^2\le C(\log n)^{2b-L_0}.
\end{align*}
Consequently, uniformly in the intermediate regime,
\begin{align}
    &\delta^2+\delta_*^2+\zeta_*^2+\varepsilon_n^2 \le C(\log n)^{-10a_1-50},\qquad Vm/n\le C(\log n)^{a_1-b},\qquad JA_n/d =o(t_n(\log n)^{-2}), \nonumber \\
    & \frac{J^2A_n(\log n)^4}{t_n^2}(N^{-1}+d^{-1})=o(1),  \frac{CA_nJ(\log n)^4}{t_n^2} \big(Vm/n+(\delta_*+\zeta_*)^2\big)=o(1). \label{eq:intermediate-parameter-scales}
\end{align}
These choices ensure \eqref{eq:reference-range} and every hypothesis of
the local and path-transfer lemmas, including
$V^C(\delta_*+\zeta_*)=o(1)$ and $N,d\ge CV^C$. They also give the
query-error scale used in \S\ref{ss:idealized-scores}:
\begin{equation}\label{eq:query-error-scale}
 \delta^2/t_n^2\le(\log n)^{-4}.
\end{equation}

Lemma~\ref{lem:approx-cB-circ} and
Corollary~\ref{lem:approx-cB-0-circ}, applied under their respective
conditioning fields, bound the two terms in
$\mathcal B_G-\mathcal B_G^{0,\circ}$.
Markov's inequality bounds the probability that
\eqref{eq:full-to-clean-direction-error} fails on
$\mathcal A\cap\mathcal E_{\mathrm{loc}}$ by
\[
 \frac{CA_nJ(\log n)^4}{t_n^2}
       \big(Vm/n+(\delta_*+\zeta_*)^2\big)=o(1).
\]
Since $\P[\mathcal A^c]=o(1)$ and
$\P[\mathcal E_{\mathrm{loc}}^c\mid\mathcal L]=o(1)$ uniformly on
$\mathcal A$, this proves the lemma for both graphs.
\end{proof}

\begin{rmk}[Choice of parameters]\label{rmk:parameter-order}
The constants are chosen in the following order. First fix $0<c<\min\{1/10,\varepsilon/6\}$, with $\varepsilon$ from Proposition~\ref{prop:rounding-finishing}. Then fix the constants $K,C_0$ of the intermediate regime as in \S\ref{ss:low-signature-construction}, and then choose $b,L,L_0$ as in the proof of Lemma~\ref{lem:direction-transfer}. After $L_0$ has been fixed, choose the constants $K,C_0$ of the high-dimensional regime so that $K\beta^2>1$ and $C_0\log(K\beta^2)>2L_0+4$. The constants of the high-dimensional regime do not enter the estimates of the intermediate regime, so the choices are consistent in both regimes. Finally, choose $C_{\mathrm{col}}\ge2\kappa_{\mathrm{col}}+2$ as in Lemma~\ref{lem:color-accuracy-good-queries}.
\end{rmk}

\subsection{Preliminary spherical and tensor estimates}
\label{ss:intermediate-spherical-estimates}

We collect the spherical identities, concentration bounds, and tensor
estimates used in the preceding proofs. The notation for the cap kernel
and its residual is as in \S\ref{ss:query-error-analysis}, and the
tensor norms are as in \S\ref{ss:seeded-local-analysis}. These
estimates concern fresh spherical points or independent reference
samples and do not require any of the embedding or path estimates.

\begin{lem}[Cap multipliers and the residual kernel]\label{lem:cap-residual-estimates}
Let $P_j^{(d)}$ be the zonal polynomial of degree $j$ normalized by
$P_j^{(d)}(1)=1$, and let $w_d(t)=c_d(1-t^2)^{(d-3)/2}$ be the density of
one coordinate of $u\sim\mu$. For $j\ge1$,
\begin{equation}\label{eq:exact-cap-multipliers}
 \lambda_j=\int_{\tau_{p,d}}^1P_j^{(d)}(t)w_d(t)\de t
 =\lambda P_{j-1}^{(d+2)}(\tau_{p,d}),\qquad
 \lambda=\frac{c_d}{d-1}(1-\tau_{p,d}^2)^{(d-1)/2}.
\end{equation}
In particular, $\lambda\asymp d^{-1/2}$ and
$\lambda_2/\lambda=\tau_{p,d}=O(d^{-1/2})$.
For $\mathcal T_k=\mathcal T-\lambda\Pi_1$,
\begin{equation}\label{eq:residual-population-operator}
 \|\mathcal T_k\|_{\mathrm{op}}\le C/d,\qquad
 \|g\|_{L^2(\mu\otimes\mu)}^2
 =\operatorname{tr}(\mathcal T_k^4)\le C/d^2.
\end{equation}
For every fixed $q\ge2$,
\begin{equation}\label{eq:residual-convolution-moments}
 \sup_v\|g(\cdot,v)\|_{L^q(\mu)}\le C_q/d,
 \qquad \|g\|_{L^q(\mu\otimes\mu)}\le C_q/d.
\end{equation}
\end{lem}

\begin{proof}
The Funk--Hecke formula \cite[Theorem~2.9]{dai2013spherical} and the
weighted Gegenbauer differentiation identity
\cite[Eq.~18.9.20]{dlmf} give \eqref{eq:exact-cap-multipliers}.
Since $p\in[p_0,1-p_0]$, we have $|\tau_{p,d}|\le C/\sqrt d$ and
$c_d\asymp\sqrt d$, proving the assertions about $\lambda$ and
$\lambda_2/\lambda$.

Parseval's identity \eqref{eq:harmonic-energy} and
$D_j\ge D_2=(d-1)(d+2)/2$ for $j\ge2$ give
$\|\mathcal T_k\|_{\mathrm{op}}\le C/d$ and
$\operatorname{tr}(\mathcal T_k^4)
\le\|\mathcal T_k\|_{\mathrm{op}}^2
\operatorname{tr}(\mathcal T_k^2)\le C/d^2$.
Since $g$ is the kernel of $\mathcal T_k^2$, this proves
\eqref{eq:residual-population-operator}.
Finally, $g(\cdot,v)=\mathcal T_k k(\cdot,v)$ and
$\|k(\cdot,v)\|_2\le C$. Applying Lemma~\ref{lem:cap-l2-lq} gives
\eqref{eq:residual-convolution-moments}.
\end{proof}

\begin{lem}[Spherical moments and local covariances]\label{lem:spherical-local-covariances}
Let $d\ge3$, let $(u,u')$ be a $\vartheta$-pair, where
$\vartheta\in[-1,1]$, and set $x=\sqrt d\,u$ and $y=\sqrt d\,u'$.
Define
\begin{equation}\label{eq:gamma-three}
 \gamma_{3,d}(\vartheta)=\frac{d^2}{(d+2)(d+4)}
 \Big(\vartheta^3-\frac{3\vartheta(1-\vartheta^2)}{d-1}\Big)
 =\vartheta^3+O(d^{-1}).
\end{equation}
For every integer $a\ge1$,
\begin{equation}\label{eq:clean-spherical-pairings}
 \E\Big[\prod_{j=1}^{2a}x_{i_j}\Big]
 =\frac{d^a}{d(d+2)\cdots(d+2a-2)}
   \sum_{\pi\in\mathcal P_2(2a)}
       \prod_{\{j,k\}\in\pi}\delta_{i_ji_k},
\end{equation}
where $\mathcal P_2(2a)$ is the set of pairings of $[2a]$; odd moments
vanish. Every even mixed moment of at most twelve coordinates of $(x,y)$
also has an expansion over coordinate pairings with a bounded number of
terms and coefficients bounded uniformly in $d$ and $\vartheta$; odd
mixed moments vanish. Moreover,
\begin{equation}\label{eq:clean-local-covariances}
 \E[xy^\top]=\vartheta I_d,\qquad
 \E[H_3(x)\otimes y]=0,\qquad
 \E[H_3(x)\otimes H_3(y)]
 =6\gamma_{3,d}(\vartheta)\cP_{\mathsf{sym}}^{(3)}.
\end{equation}
In coordinates,
\[
 6\cP_{\mathsf{sym}}^{(3)}
 =\sum_{\pi\in\mathcal C_3}\Delta_\pi
   -\frac{2}{d+2}\sum_{\pi\in\mathcal T_3}\Delta_\pi,
\]
where $\Delta_\pi$ is the product of Kronecker deltas over the pairs of
$\pi$, $\mathcal C_3$ consists of the six crossing pairings between two
triples of indices, and $\mathcal T_3$ consists of the nine pairings with
one internal pair in each triple. Finally, for every deterministic
$v\in\bbR^d$,
\[
 \E[H_3(x)[v]^2]\preceq Cd\|v\|^2I_d.
\]
\end{lem}

\begin{proof}
Equation~\eqref{eq:clean-spherical-pairings} is the spherical Wick
formula \cite[Section~3]{vignat2008wick}. Orthogonality of spherical
harmonics and the Funk--Hecke formula
\cite[Section~1 and Theorem~2.9]{dai2013spherical} give
\eqref{eq:clean-local-covariances}: the latitude average multiplies
degree-three harmonics by
$P_3^{(d)}(\vartheta)=\vartheta^3
-3\vartheta(1-\vartheta^2)/(d-1)$, and the remaining factor in
$\gamma_{3,d}$ is the sixth-moment coefficient in
\eqref{eq:clean-spherical-pairings}. The latitude formula follows from
Funk--Hecke by approximating the point mass at $\vartheta$; at
$\vartheta=1$ it is the identity, and at $\vartheta=-1$ it follows by
parity. The coordinate expansion of $6\cP_{\mathsf{sym}}^{(3)}$ follows
by subtracting the three traces from the symmetric cubic tensor. Each
trace term carries a factor $O(d^{-1})$.

For the mixed-moment assertion, condition on $x$ and write
$y=\vartheta x+\sqrt{d(1-\vartheta^2)}w$, where $w$ is uniform on the unit
sphere of $x^\perp$. Apply the spherical Wick formula in this subspace,
whose identity operator is $I_d-xx^\top/d$, and then apply
\eqref{eq:clean-spherical-pairings} to $x$. The resulting terms are
coordinate pairings. Their number and coefficients are bounded
uniformly because $d\ge3$ and the total degree is at most twelve.
Odd mixed moments vanish by the symmetry $(x,y)\mapsto(-x,-y)$.

For the last assertion, \eqref{eq:low-cubic-matrix} and $\|x\|^2=d$ give,
for every $z\in\bbR^d$,
\[
 H_3(x)[v]z=\langle x,v\rangle\langle x,z\rangle x
 -\frac d{d+2}\big(\langle x,v\rangle z
                    +\langle v,z\rangle x+\langle x,z\rangle v\big).
\]
Apply \eqref{eq:clean-spherical-pairings} with $a=1,2$ to obtain
$\E\|H_3(x)[v]z\|^2\le Cd\|v\|^2\|z\|^2$.
\end{proof}

\begin{lem}[Empirical cubic tensors]\label{lem:empirical-cubic-tensor}
Let $d\le m\le n$, with $d\ge3$, and let $x_1,\ldots,x_m$ be
independent and uniform on the sphere of radius $\sqrt d$.
For every fixed $\gamma>0$, with probability at least $1-n^{-\gamma}$,
\[
 \Big\|\sum_{j=1}^m x_j^{\otimes3}\Big\|_{1\to\mathrm F}
 \le C_\gamma\big(d\sqrt{m\log n}+d^{3/2}\log n\big).
\]
\end{lem}

\begin{proof}
Each flattening of $\sum_jx_j^{\otimes3}$ is a sum of independent
centered matrices. The summand norm is $d^{3/2}$, and both variance
parameters are at most $Cmd^2$, by
\eqref{eq:clean-spherical-pairings} with $a=1,2$.
The rectangular matrix Bernstein inequality
\cite[Theorem~6.1.1]{tropp2015introduction}, followed by a union bound
over the three flattenings, gives the result.
\end{proof}

\begin{lem}[Uniform cap backprojection]\label{lem:uniform-cap-backprojection}
Let $m\ge d\ge3$, let $u_1,\ldots,u_m$ be independent with law $\mu$,
and set $x_j=x(u_j)$. Define
\[
 \widehat x(u)=\frac1{m\lambda}\sum_{j=1}^m h_0(u,u_j)x_j,
 \qquad u\in\mathbb S^{d-1}.
\]
Then $\E[\widehat x(u)]=x(u)$ for every $u$, and, for every $t\ge1$,
\begin{equation}\label{eq:uniform-reference-backprojection}
 \P\Big[\sup_u\|\widehat x(u)-x(u)\|
        >Cd/\sqrt m+Cd\sqrt{t/m}\Big]\le e^{-t}.
\end{equation}
\end{lem}

\begin{proof}
The expectation identity follows from $\mathcal Tx=\lambda x$.
Let $X$ have rows $x_j^\top$. The cap class has VC dimension at most
$d+1$. Conditional on the sample, its weighted VC covering bound and an
output-direction net give covering numbers at most
$(C/\varepsilon)^{Cd}$, at Euclidean radius
$\varepsilon\|X\|_{\mathrm{op}}$, for the vectors
$(h_0(u,u_j)\langle x_j,v\rangle)_j$ indexed by unit $u,v$.
For each fixed $v$ use the weights
$\langle x_j,v\rangle^2/\|Xv\|^2$; if $Xv=0$ the vector vanishes.
Dudley's inequality and symmetrization
\cite[Theorems~8.3.18 and~8.1.3]{vershynin2018high} yield
\[
 \E\sup_u\|\widehat x(u)-x(u)\|
 \le\frac{C\sqrt d}{m\lambda}\E\|X\|_{\mathrm{op}}
 \le Cd/\sqrt m.
\]
The last inequality uses the isotropic sub-Gaussian covariance bound
\cite[Theorem~4.6.1]{vershynin2018high}, $m\ge d$, and
$\lambda\asymp d^{-1/2}$.
Changing one sample point changes the supremum by at most
$2\sqrt d/(m\lambda)\le Cd/m$. Bounded differences
\cite[Theorem~2.9.1]{vershynin2018high} therefore proves
\eqref{eq:uniform-reference-backprojection}.
\end{proof}

\begin{lem}\label{lem:stf-one-slot}
For $q\in\{1,2,3\}$ and every tensor $T$ of order $q$,
\[
 \|\cP_{\mathsf{sym}}\operatorname{Sym}T\|_{1\to\mathrm F}
 \le C\|T\|_{1\to\mathrm F},\quad
 \|\cP_{\mathsf{sym}}\operatorname{Sym}T\|_{\mathcal P}
 \le C\|T\|_{\mathcal P},
\]
where the second inequality concerns tensors $T=T(u)$ depending on a fresh
point, and $C$ is an absolute constant.
\end{lem}

\begin{proof}
Slot permutations preserve both norms, so averaging reduces to symmetric
$T$. The claim is immediate for $q=1$. For $q=2$, the correction is
$(\operatorname{tr}T)I_d/d$; use
$|\operatorname{tr}T|\le d\|T\|_{\mathrm{op}}$ and
$\E_u|\operatorname{tr}T|^2\le d^2\|T\|_{\mathcal P}^2$.
For $q=3$, put $s_i=\sum_jT_{ijj}$. The correction contracted against a
unit $v$ is
$ (\langle s,v\rangle I_d+sv^\top+vs^\top)/(d+2)$.
Its Frobenius norm is at most
$(\sqrt d|\langle s,v\rangle|+2\|s\|)/(d+2)$.
Since $\langle s,v\rangle=\operatorname{tr}(\mathcal M_1(T)v)$,
\[
 \|s\|\le\sqrt d\|T\|_{1\to\mathrm F},\quad
 \sup_{\|v\|=1}\E_u\langle s,v\rangle^2\le d\|T\|_{\mathcal P}^2,
 \quad \E_u\|s\|^2\le d^2\|T\|_{\mathcal P}^2.
\]
These bounds prove both assertions.
\end{proof}

\begin{lem}\label{lem:positive-local-difference}
Let $\eta:\mathbb S^{d-1}\to\bbR^d$ be measurable, let $f=x+\eta$, and
evaluate $x=x(u)$, $\eta=\eta(u)$ and $f=f(u)$ at the same fresh point
$u\sim\mu$. Suppose that $0<\delta\le1$ and
\begin{equation}\label{eq:local-error-premises}
 \sup_u\|\eta(u)\|^2\le\delta^2d,\quad
 \E_u[\eta\eta^\top]\preceq\delta^2I_d.
\end{equation}
Then, for $q\in\{1,2,3\}$ and $D_q=H_q(f)-H_q(x)$,
\[
 \|D_q\|_{\mathcal P}^2\le C\delta^2d^{q-1},\quad
 \|H_q(x)\|_{\mathcal P}^2+\|H_q(f)\|_{\mathcal P}^2\le Cd^{q-1}.
\]
Moreover, if $D_3[A]$ denotes contraction of two slots of $D_3$ against
a matrix $A$, then
\begin{equation}\label{eq:seeded-second-flattening}
 \sup_{\|A\|_{\mathrm F}=1}\E_u\|D_3[A]\|^2\le C\delta^2d^2.
\end{equation}
\end{lem}

\begin{proof}
Expand $f^{\otimes q}-x^{\otimes q}$. A product with $r\ge1$ error slots
has, after contraction with a unit vector, expected squared Frobenius
norm at most $\delta^{2r}d^{q-1}$: use the covariance bound for the
contracted slot and the uniform bounds for the other slots.
The triangle inequality, $\delta\le1$ and Lemma~\ref{lem:stf-one-slot}
give the first claim. The same calculation with $r=0$ bounds
$H_q(x)$, and $H_q(f)=H_q(x)+D_q$ gives the remaining bound.

For \eqref{eq:seeded-second-flattening}, symmetry allows us to restrict
to symmetric $A$ with $\|A\|_{\mathrm F}\le1$. Put
$B=2A+(\operatorname{tr}A)I_d$. Since
$H_3(z)[A]=z(z^\top Az)-\|z\|^2Bz/(d+2)$,
\[
 D_3[A]=\eta(x^\top Ax)+(x+\eta)(2x^\top A\eta+\eta^\top A\eta)
 -\frac{dB\eta+(2\langle x,\eta\rangle+\|\eta\|^2)B(x+\eta)}{d+2}.
\]
Use \eqref{eq:local-error-premises},
$\E_u(x^\top Ax)^2\le Cd$ from
\eqref{eq:clean-spherical-pairings},
$\E_u\|A\eta\|^2\le\delta^2$, and
$\|B\|_{\mathrm F}^2\le Cd^2$. The expected squared norm is at most
\[
 C\big(\delta^2d\E_u(x^\top Ax)^2+d^2\E_u\|A\eta\|^2
       +\delta^2\E_u\|Bx\|^2+\E_u\|B\eta\|^2\big)
 \le C\delta^2d^2. \qedhere
\]
\end{proof}

\begin{lem}\label{lem:reference-positive-gram}
    We have $\E[\|g(T_G,T_G)\|_{\mathrm{op}}]\le C(m/d^2+\log(2m))$, where the expectation is over the reference points.
\end{lem}

\begin{proof}
By \eqref{eq:residual-population-operator},
$\|\mathcal T_k^2\|_{\mathrm{op}}\le C/d^2$.
The independent vectors $z_j=k(\cdot,u_{t_j}^G)\in L^2(\mu)$ have
second-moment operator $\mathcal T_k^2$ and squared norm
$\sigma_k^2\le C$. Their Gram matrix $g(T_G,T_G)$ has the same nonzero
eigenvalues as $\sum_jz_j\otimes z_j$.
Thus \cite[Lemma~3.2]{fernandez2026spectral}, with tail exponent $2$,
gives the claimed expectation bound, using the deterministic bound
$\|g(T_G,T_G)\|_{\mathrm{op}}\le m\sigma_k^2$ on the exceptional event.
\end{proof}

\begin{lem}\label{lem:reference-operator-estimates}
If $m\ge Cd$, then
\begin{align}
    \E\big[\|\mathcal K_G\|_{\mathrm{op}}^2\big]&\le C\big(m^2/d^2+m\log(2m)+d\big),
    \label{eq:reference-residual-bound}\\
    \E\big[\|\mathcal K_GX_G\|_{\mathrm{op}}^2\big]&\le Cm^2\log(2d),
    \label{eq:reference-residual-X-bound}\\
    \E\big[\|X_G^\top\mathcal K_GX_G\|_{\mathrm{op}}^2\big]&\le Cm^2d\log^2(2d).
    \label{eq:reference-X-residual-X-bound}
\end{align}
\end{lem}

\begin{proof}
The isotropic sub-Gaussian covariance bound
\cite[Theorem~4.6.1]{vershynin2018high} gives
$\|X_G\|_{\mathrm{op}}\le C\sqrt m$ except with probability
$e^{-c_1m}$. On this event, $|h_0|\le1$ and $\lambda\sqrt d\le C$ imply
\begin{equation}\label{eq:uniform-reference-kernel-vector}
 \sup_u\|k(u,T_G)\|
 \le\sup_u\big(\|h_0(u,T_G)\|+\lambda\|X_Gx(u)\|\big)\le C\sqrt m.
\end{equation}
Choose an independent uniform balanced partition
$[m]=\Omega_1\cup\Omega_2$, with $m_a=|\Omega_a|$. Writing
$P_{\Omega_a}$ for coordinate projections, we have
\begin{equation}\label{eq:balanced-reference-split}
 (\mathcal K_G)_{\mathrm{off}}
 =\frac{m(m-1)}{m_1m_2}\E_{\mathrm{split}}
       [P_{\Omega_1}\mathcal K_GP_{\Omega_2}].
\end{equation}
The prefactor is bounded, so Jensen's inequality reduces all three
estimates to rectangular blocks. Condition on the references in
$\Omega_2$ and on $\|(X_G)_{\Omega_2}\|_{\mathrm{op}}\le C\sqrt m$.
The complement contributes negligibly: its probability is exponential
and all squared norms under consideration are polynomially bounded in
$m$. Subscripts below restrict reference indices, and
$\E_{\Omega_1}$ integrates the references in $\Omega_1$.

The rows of $(\mathcal K_G)_{\Omega_1\Omega_2}$ are conditionally
independent, have norm at most $C\sqrt m$, and have second-moment matrix
$g(T_G,T_G)_{\Omega_2\Omega_2}$. The expectation form of matrix Chernoff
\cite[Chapter~5]{tropp2015introduction} gives
\[
 \E_{\Omega_1}\| (\mathcal K_G)_{\Omega_1\Omega_2}\|_{\mathrm{op}}^2
 \le C\big(m\|g(T_G,T_G)_{\Omega_2\Omega_2}\|_{\mathrm{op}}
              +m\log(2m)\big).
\]
Lemma~\ref{lem:reference-positive-gram}, \eqref{eq:balanced-reference-split},
and the diagonal $-d\lambda I_m$ prove \eqref{eq:reference-residual-bound}.

For \eqref{eq:reference-residual-X-bound}, the rows of
$(\mathcal K_G)_{\Omega_1\Omega_2}(X_G)_{\Omega_2}$ are
$v(u_{t_i}^G)^\top$, where
\[
 v(u)=\sum_{j\in\Omega_2}k(u,u_{t_j}^G)x_j,
 \qquad \mathcal C_{\Omega_2}=\E_u[v(u)v(u)^\top].
\]
Their second moment decomposes as
\begin{equation}\label{eq:projected-reference-Gram}
 \mathcal C_{\Omega_2}
 =\sigma_k^2(X_G)_{\Omega_2}^\top(X_G)_{\Omega_2}
   +\sum_{j\ne k\in\Omega_2}x_jx_k^\top g(u_{t_j}^G,u_{t_k}^G).
\end{equation}
The off-diagonal summands have mean zero when either reference is
integrated out, because $g$ annihilates linear functions. Only matching
unordered index pairs survive the squared Frobenius expansion, so
\begin{equation}\label{eq:projected-reference-Gram-moment}
 \E\Big\|\sum_{j\ne k\in\Omega_2}
   x_jx_k^\top g(u_{t_j}^G,u_{t_k}^G)\Big\|_{\mathrm F}^2
 \le Cm^2d^2\|g\|_{L^2(\mu\otimes\mu)}^2\le Cm^2.
\end{equation}
Thus $\E\|\mathcal C_{\Omega_2}\|_{\mathrm{op}}\le Cm$.
Since \eqref{eq:uniform-reference-kernel-vector} gives
$\sup_u\|v(u)\|\le Cm$, matrix Chernoff yields
\[
 \E_{\Omega_1}\| (\mathcal K_G)_{\Omega_1\Omega_2}(X_G)_{\Omega_2}
                     \|_{\mathrm{op}}^2
 \le C\big(m\|\mathcal C_{\Omega_2}\|_{\mathrm{op}}+m^2\log(2d)\big).
\]
Integration and \eqref{eq:balanced-reference-split} give the desired
off-diagonal bound. The diagonal contribution satisfies
\[
 \E\|d\lambda X_G\|_{\mathrm{op}}^2\le Cmd.
\]

Finally,
$(X_G)_{\Omega_1}^\top(\mathcal K_G)_{\Omega_1\Omega_2}(X_G)_{\Omega_2}
=\sum_{i\in\Omega_1}x_iv(u_{t_i}^G)^\top$
is a sum of conditionally independent centered matrices, each with norm
at most $Cm\sqrt d$. Its one-summand variance matrices are
$d\mathcal C_{\Omega_2}$ and
$\mathcal V_{\Omega_2}=\E_u[x(u)x(u)^\top\|v(u)\|^2]$.
The diagonal terms in the expansion of $\mathcal V_{\Omega_2}$ have
total norm at most $Cmd$, by the bounded fourth moments of
$\langle x(u),a\rangle$, $\|a\|=1$, and of $k(u,w)$.
For the off-diagonal terms, set
$M_{jk}=\E_u[x(u)x(u)^\top k(u,u_{t_j}^G)k(u,u_{t_k}^G)]$.
Equation~\eqref{eq:residual-convolution-moments} gives
$\|g\|_{L^4(\mu\otimes\mu)}\le C/d$. Consequently,
\[
 \E\|M_{jk}\|_{\mathrm F}^2
 =\iint\langle x(u),x(v)\rangle^2g(u,v)^2\mu(\de u)\mu(\de v)
 \le\big(\E\langle x(u),x(v)\rangle^4\big)^{1/2}\|g\|_4^2
 \le C/d.
\]
The summands of
$(\mathcal V_{\Omega_2})_{\mathrm{off}}
=\sum_{j\ne k}\langle x_j,x_k\rangle M_{jk}$
are again canonical in either reference. Using
$|\langle x_j,x_k\rangle|\le d$ as in
\eqref{eq:projected-reference-Gram-moment} gives
$\E\| (\mathcal V_{\Omega_2})_{\mathrm{off}}\|_{\mathrm F}^2\le Cm^2d$
and hence $\E\|\mathcal V_{\Omega_2}\|_{\mathrm{op}}\le Cmd$.
Both summed variances therefore have expected norm $O(m^2d)$.
The integrated rectangular matrix Bernstein bound
\cite[Theorem~6.1.1]{tropp2015introduction} gives the squared-norm
estimate $Cm^2d\log^2(2d)$. Average over the partition and add
$-d\lambda X_G^\top X_G$, whose squared norm has expectation
$O(m^2d)$, to obtain \eqref{eq:reference-X-residual-X-bound}.
\end{proof}

\section{Computing the signatures by color coding}\label{s:color-coding}

Direct enumeration of \eqref{eq:low-exact-signature} and \eqref{eq:high-exact-signature} costs $n^{O(V)}$ operations. Color coding \cite{alon1995color,mao2023random} replaces these sums by unbiased estimates computable by dynamic programming. A polylogarithmic number of colorings gives accurate signatures outside $o(n)$ query labels, which suffices for recovery by Lemma~\ref{lem:threshold-matching}.

\subsection{Colorful signatures}\label{ss:colorful-signatures}

Fix $G\in\{A,B\}$ and an admissible seed guess. For a constant $C_{\mathrm{col}}$ specified in Lemma~\ref{lem:color-accuracy-good-queries}, set
\begin{equation}\label{eq:coloring-budget}
 p_{\mathrm{col}}=\frac{V!}{V^V},\qquad M_{\mathrm{col}}=\lceil(\log n)^{C_{\mathrm{col}}}\rceil.
\end{equation}
A uniform coloring $\chi:R_G\to[V]$ assigns independent uniform colors to the reservoir vertices. A tuple $(r_1,\ldots,r_V)$ is \emph{colorful} if its colors are distinct; in particular, it is injective. Let $\Phi_w^\chi(i)$ be the sum in \eqref{eq:low-exact-signature} or \eqref{eq:high-exact-signature} restricted to colorful tuples, with its original normalization, and write $\Phi^\chi(i)=(\Phi_w^\chi(i))_{w\in\mathcal W}$.

For $M$ independent colorings, define
\begin{equation}\label{eq:color-coded-signatures}
 \widetilde\Phi_w^G(i)=\frac{1}{Mp_{\mathrm{col}}}\sum_{\ell=1}^{M}\Phi_w^{\chi_\ell}(i),
 \qquad
 \widetilde S(i,j)=\sum_{w\in\mathcal W}\widetilde\Phi_w^A(i)\widetilde\Phi_w^B(j).
\end{equation}
The algorithm uses $M=M_{\mathrm{col}}$, independently for the two graphs and afresh for each seed guess; within a graph and seed guess, all words and queries share the same colorings. Every injective tuple is colorful with probability $p_{\mathrm{col}}$, so, conditionally on the graphs and label sets, $\E_\chi[\widetilde\Phi_w^G(i)]=\Phi_w^G(i)$ and $\E_\chi[\widetilde S(i,j)]=S(i,j)$.

\subsection{Dynamic programs}\label{ss:color-dynamic-programs}

Fix a word $w$ and a coloring $\chi$ of $R_G$. The messages below depend on $G$, $w$ and $\chi$, which we suppress from the notation. For $t\in[V]$, $Q\subseteq[V]$ and $r\in R_G$, the message $F_t(Q,r)$ is zero unless $|Q|=V-t+1$ and $\chi(r)\in Q$; otherwise it is defined by a backward recursion in $t$. In the high-dimensional regime the messages are scalars:
\begin{align}
 F_V(\{\chi(r)\},r)&=1, \notag\\
 F_t(Q,r)&=Z^G_{s^G_{w_t}r}\sum_{q\in R_G}Z^G_{rq}F_{t+1}(Q\setminus\{\chi(r)\},q),\quad t<V.
 \label{eq:color-graph-recursion}
\end{align}
In the intermediate regime they are vectors in $\bbR^d$:
\begin{align}
 F_V(\{\chi(r)\},r)&=f_G(r), \notag\\
 F_t(Q,r)&=H_3(f_G(r))[v_{w_t}^G]\sum_{q\in R_G}F_{t+1}(Q\setminus\{\chi(r)\},q),\quad t<V.
 \label{eq:color-coordinate-recursion}
\end{align}
For $|Q|=V-t+1$ and $\chi(r)\in Q$, let $\mathcal R_t(Q,r)$ be the set of tuples $(r_t,\ldots,r_V)$ of vertices of $R_G$ with $r_t=r$ and $\{\chi(r_t),\ldots,\chi(r_V)\}=Q$. The next lemma identifies the messages with colorful path sums.

\begin{lem}[Colorful messages]\label{lem:colorful-messages}
Let $t\in[V]$, let $Q\subseteq[V]$ with $|Q|=V-t+1$, and let $r\in R_G$ with $\chi(r)\in Q$. Then, in the high-dimensional and in the intermediate regime, respectively,
\begin{align}
 F_t(Q,r)&=\sum_{(r_t,\ldots,r_V)\in\mathcal R_t(Q,r)}\prod_{k=t}^{V-1}Z^G_{s^G_{w_k}r_k}Z^G_{r_kr_{k+1}}, \notag\\
 F_t(Q,r)&=\sum_{(r_t,\ldots,r_V)\in\mathcal R_t(Q,r)}\Big(\prod_{k=t}^{V-1}H_3(f_G(r_k))[v_{w_k}^G]\Big)f_G(r_V),
 \label{eq:colorful-messages}
\end{align}
with the matrix factors ordered by increasing $k$ from left to right. Consequently, for every $i\in I_G$,
\begin{equation}\label{eq:colorful-signature-dp}
 \Phi_w^\chi(i)=\frac{1}{\sqrt{(N)_V}}\sum_{r\in R_G}Z^G_{ir}F_1([V],r),
 \quad\text{respectively}\quad
 \Phi_w^\chi(i)=\frac{1}{\sqrt{(N)_Vd^V}}f_G(i)^\top\sum_{r\in R_G}F_1([V],r).
\end{equation}
\end{lem}

\begin{proof}
The terminal identities hold by definition. Deleting the first entry of a colorful suffix removes precisely the color $\chi(r)$, so induction in \eqref{eq:color-graph-recursion} and \eqref{eq:color-coordinate-recursion} gives \eqref{eq:colorful-messages}, with the level-$t$ factor on the left. Summing over the first vertex then gives \eqref{eq:colorful-signature-dp}.
\end{proof}

The query enters only through the final contraction in \eqref{eq:colorful-signature-dp}. We therefore average $F_1([V],r)$, respectively $\sum_rF_1([V],r)$, over the colorings before evaluating the queries. In the intermediate regime,
\begin{equation}\label{eq:cubic-matrix-vector}
 H_3(z)[v]y=(z^\top v)z(z^\top y)-\frac{\|z\|^2}{d+2}\big(z(v^\top y)+v(z^\top y)+(z^\top v)y\big)
\end{equation}
costs $O(d)$ operations without forming a matrix. The suffix sum $\sum_qF_{t+1}(Q',q)$ is computed once for each $Q'$, and only two consecutive message levels need to be stored. Algorithms~\ref{alg:color-graph} and~\ref{alg:color-coordinate} give the procedures for one graph and seed guess. Their accumulators estimate the coefficients in \eqref{eq:high-path-coefficient} and \eqref{eq:low-path-direction}; words are processed in lexicographic order.

\begin{algorithm}[htbp]
\caption{Color coding for signatures in the high-dimensional regime}\label{alg:color-graph}
\begin{algorithmic}[1]
\State \textbf{Input:} $n,N,V,C_{\mathrm{col}},\mathcal W,R_G,I_G,(s_a^G)_{a\in[K]}$, and $Z^G$.
\State Set $M\gets\lceil(\log n)^{C_{\mathrm{col}}}\rceil$ and $p_{\mathrm{col}}\gets V!/V^V$.
\State Set $\widetilde v_w(r)\gets0$ for every $w\in\mathcal W$ and $r\in R_G$.
\For{$\ell=1,\ldots,M$}
    \State Sample $\chi(r)$ uniformly from $[V]$, independently for every $r\in R_G$.
    \ForAll{$w=(w_1,\ldots,w_{V-1})\in\mathcal W$}
        \State Set all messages $F_t(Q,r)\gets0$.
        \State Set $F_V(\{\chi(r)\},r)\gets1$ for every $r\in R_G$.
        \For{$t=V-1,V-2,\ldots,1$}
            \ForAll{$Q'\subseteq[V]$ with $|Q'|=V-t$}
                \ForAll{$r\in R_G$ with $\chi(r)\notin Q'$}
                    \State $F_t(Q'\cup\{\chi(r)\},r)\gets
                    Z^G_{s^G_{w_t}r}\displaystyle\sum_{q\in R_G}Z^G_{rq}F_{t+1}(Q',q)$.
                \EndFor
            \EndFor
        \EndFor
        \State Set $\widetilde v_w(r)\gets\widetilde v_w(r)+F_1([V],r)$ for every $r\in R_G$.
    \EndFor
\EndFor
\ForAll{$w\in\mathcal W$}
    \State Set $\widetilde v_w(r)\gets\widetilde v_w(r)/(Mp_{\mathrm{col}}\sqrt{(N)_V})$ for every $r\in R_G$.
    \State Set $\widetilde\Phi_w^G(i)\gets\displaystyle\sum_{r\in R_G}Z^G_{ir}\widetilde v_w(r)$ for every $i\in I_G$.
\EndFor
\State \textbf{Output:} $\widetilde\Phi^G(i)=(\widetilde\Phi_w^G(i))_{w\in\mathcal W}$ for all $i\in I_G$.
\end{algorithmic}
\end{algorithm}

\begin{algorithm}[htbp]
\caption{Color coding for signatures in the intermediate regime}\label{alg:color-coordinate}
\begin{algorithmic}[1]
\State \textbf{Input:} $n,d,N,V,C_{\mathrm{col}},\mathcal W,R_G,I_G,f_G$, and $(v_a^G)_{a\in[K]}$.
\State Set $M\gets\lceil(\log n)^{C_{\mathrm{col}}}\rceil$ and $p_{\mathrm{col}}\gets V!/V^V$.
\State Set $\widetilde b_w\gets0\in\bbR^d$ for every $w\in\mathcal W$.
\For{$\ell=1,\ldots,M$}
    \State Sample $\chi(r)$ uniformly from $[V]$, independently for every $r\in R_G$.
    \ForAll{$w=(w_1,\ldots,w_{V-1})\in\mathcal W$}
        \State Set all messages $F_t(Q,r)\gets0\in\bbR^d$.
        \State Set $F_V(\{\chi(r)\},r)\gets f_G(r)$ for every $r\in R_G$.
        \For{$t=V-1,V-2,\ldots,1$}
            \ForAll{$Q'\subseteq[V]$ with $|Q'|=V-t$}
                \State Set $y\gets\displaystyle\sum_{q\in R_G}F_{t+1}(Q',q)$.
                \ForAll{$r\in R_G$ with $\chi(r)\notin Q'$}
                    \State Set $F_t(Q'\cup\{\chi(r)\},r)\gets H_3(f_G(r))[v_{w_t}^G]y$ using \eqref{eq:cubic-matrix-vector}.
                \EndFor
            \EndFor
        \EndFor
        \State Set $\widetilde b_w\gets\widetilde b_w+\displaystyle\sum_{r\in R_G}F_1([V],r)$.
    \EndFor
\EndFor
\ForAll{$w\in\mathcal W$}
    \State Set $\widetilde b_w\gets\widetilde b_w/\big(Mp_{\mathrm{col}}\sqrt{(N)_Vd^{V-1}}\big)$.
    \State Set $\widetilde\Phi_w^G(i)\gets d^{-1/2}f_G(i)^\top\widetilde b_w$ for every $i\in I_G$.
\EndFor
\State \textbf{Output:} $\widetilde\Phi^G(i)=(\widetilde\Phi_w^G(i))_{w\in\mathcal W}$ for all $i\in I_G$.
\end{algorithmic}
\end{algorithm}

\subsection{Accuracy}\label{ss:color-accuracy}\label{ss:color-mean-square}

Using Lemma~\ref{lem:numerical-restriction}, we bound the mean-square error for kernels with values in a Hilbert space, without a centering assumption. Markov's inequality then gives simultaneous accuracy outside $o(n)$ labels.

\begin{lem}[Mean-square bound for color coding]\label{lem:color-kernel-approximation}
Let $X=(X_1,\ldots,X_N)$ have independent coordinates with a common law, and let $g$ be a square-integrable symmetric kernel of $V$ variables with values in a real Hilbert space. Let $\chi_1,\ldots,\chi_M$, where $M\ge1$, be independent uniform colorings of $[N]$ with $V$ colors, independent of $X$. For a coloring $\chi$, let
\[
 U=\sum_{\substack{S\subseteq[N]\\|S|=V}}g(X_S),\quad
 U_\chi=\sum_{\substack{S\subseteq[N]\\|S|=V}}\mathbf1\{\chi\text{ is injective on }S\}g(X_S),\quad
 \widetilde U=\frac{1}{Mp_{\mathrm{col}}}\sum_{\ell=1}^MU_{\chi_\ell}.
\]
Then
\[
 \E[\widetilde U\mid X]=U,
 \quad
 \E[\|\widetilde U-U\|^2]\le\frac{p_{\mathrm{col}}^{-1}-1}{M}\E[\|U\|^2].
\]
The same conclusions hold conditionally on a $\sigma$-field $\mathcal H$ if $g$ is $\mathcal H$-measurable, the $X_i$ are conditionally independent given $\mathcal H$ with a common conditional law, and the colorings are independent of $\mathcal H\vee\sigma(X)$.
\end{lem}

\begin{proof}
Unbiasedness follows from $\P_\chi[\chi\text{ is injective on }S]=p_{\mathrm{col}}$. For fixed $\chi$, apply \eqref{eq:weighted-kernel-restriction} to the weights $\mathbf1\{\chi\text{ is injective on }S\}$ and average over $\chi$ to obtain $\E\|U_\chi\|^2\le p_{\mathrm{col}}\E\|U\|^2$. Conditional independence of the colorings now gives
\[
 \E\|\widetilde U-U\|^2
 =\frac1M\left(p_{\mathrm{col}}^{-2}\E\|U_\chi\|^2-\E\|U\|^2\right)
 \le\frac{p_{\mathrm{col}}^{-1}-1}{M}\E\|U\|^2.
\]
The conditional version follows by the same application of the conditional restriction inequality.
\end{proof}

To apply the lemma, relabel $B$ so that $\pi=\mathrm{id}$ and fix a graph $G$, an admissible seed guess and $i\in I_G$. Let $\varphi_i(y_1,\ldots,y_V)\in\bbR^J$ be the vector of normalized summands in the exact signature, with $u_{r_t}^G$ replaced by $y_t$, and symmetrize it by setting $g_i(y_1,\ldots,y_V)=\sum_{\sigma\in\mathfrak S_V}\varphi_i(y_{\sigma(1)},\ldots,y_{\sigma(V)})$. The identities $G_{xy}-p=h_0(u_x^G,u_y^G)$ and $f_G(r)=f_G(u_r^G)$ make this a well-defined kernel. With $u_S=(u_r^G)_{r\in S}$,
\[
 \Phi^G(i)=\sum_{\substack{S\subseteq R_G\\|S|=V}}g_i(u_S),\qquad
 \Phi^\chi(i)=\sum_{\substack{S\subseteq R_G\\|S|=V}}\mathbf1\{\chi\text{ is injective on }S\}g_i(u_S).
\]
Let $\mathcal H_i$ be generated by the label sets and the latent vectors of $G$ at $\{i\}\cup T_G\cup S_G\cup D_G$. The bounded kernel $g_i$ is $\mathcal H_i$-measurable, while the reservoir points are conditionally independent with law $\mu$, since these labels lie outside $R_G$. Thus Lemma~\ref{lem:color-kernel-approximation}, applied conditionally on $\mathcal H_i$, gives
\begin{equation}\label{eq:color-signature-mean-square}
 \E[\|\widetilde\Phi^G(i)-\Phi^G(i)\|^2\mid\mathcal H_i]
 \le\frac{p_{\mathrm{col}}^{-1}-1}{M}\E[\|\Phi^G(i)\|^2\mid\mathcal H_i].
\end{equation}

For the remaining analysis, use the constants of Proposition~\ref{prop:true-seed-candidate}, the true seed guess, and $\pi=\mathrm{id}$. Let $\mathcal L,\mathcal F_G,\mathcal A$ be as in \S\ref{ss:idealized-scores}: they specify the label sets, the exposed latent vectors at $T_G\cup S_G\cup D_G$, and admissibility, respectively. Then $\mathcal H_i=\mathcal F_G\vee\sigma(u_i^G)$. In the high-dimensional regime, let $\mathcal E_G$ be the whole sample space, as is $\mathcal A$. In the intermediate regime, let $\mathcal E_G\in\mathcal F_G$ be the event of Lemma~\ref{lem:seeded-local}; its hypotheses and those of Lemma~\ref{lem:seeded-path-transfer} hold by the parameter choices in the proof of Lemma~\ref{lem:direction-transfer}. On $\mathcal E_G\subseteq\mathcal E_{\mathrm{ref},G}$ the spectral checks pass, and on $\mathcal A\cap\mathcal E_G$ Lemma~\ref{lem:approx-cB-circ} gives \eqref{eq:full-reservoir-second-moment}. Let $\mathcal I$ be as in Remark~\ref{rmk:choice-cI}. Probabilities below include the colorings.

\begin{lem}[Accuracy outside a small exceptional set]\label{lem:color-accuracy-good-queries}
There is a constant $\kappa_{\mathrm{col}}\ge1$ such that the following holds for all sufficiently large $n$.
\begin{enumerate}[label=(\roman*)]
\item On $\mathcal A$, for every $G\in\{A,B\}$, every query $i\in I_G$ and every number $M\ge1$ of colorings,
\[
 \E[\mathbf1_{\mathcal E_G}\|\Phi^G(i)\|^2\mid\mathcal L]\le(\log n)^{\kappa_{\mathrm{col}}},
 \quad
 \E[\mathbf1_{\mathcal E_G}\|\widetilde\Phi^G(i)-\Phi^G(i)\|^2\mid\mathcal L]\le\frac{(\log n)^{\kappa_{\mathrm{col}}}}{M}.
\]
\item Let $M=M_{\mathrm{col}}$ with $C_{\mathrm{col}}\ge2\kappa_{\mathrm{col}}+2$ in \eqref{eq:coloring-budget}. With probability $1-o(1)$, the true seed guess is admissible, its preprocessing succeeds, and there is a set $\mathcal X_{\mathrm{col}}\subseteq\mathcal I$ with $|\mathcal X_{\mathrm{col}}|\le n/\sqrt{\log n}$ such that
\begin{equation}\label{eq:color-accuracy-good-queries}
 \max_{\substack{G\in\{A,B\}\\i\in\mathcal I\setminus\mathcal X_{\mathrm{col}}}}\|\Phi^G(i)\|\le(\log n)^{(\kappa_{\mathrm{col}}+1)/2},
 \quad
 \max_{\substack{G\in\{A,B\}\\i\in\mathcal I\setminus\mathcal X_{\mathrm{col}}}}\|\widetilde\Phi^G(i)-\Phi^G(i)\|\le(\log n)^{-(\kappa_{\mathrm{col}}+1)/2}.
\end{equation}
\end{enumerate}
\end{lem}

\begin{proof}
In the high-dimensional regime, Lemma~\ref{lem:moment-compare-err}(iii) and \eqref{eq:direct-moment-error} give $\E[\|\Phi^G(i)\|^2\mid\mathcal L]\le2J$ for every query. In the intermediate regime, condition on $\mathcal F_G$ on $\mathcal A\cap\mathcal E_G$. The first bound in \eqref{eq:full-reservoir-second-moment} gives
\begin{equation}\label{eq:color-full-reservoir-directions}
 \frac1d\sum_{w\in\mathcal W}\E[\|b_w^G\|^2\mid\mathcal F_G]\le C^VV^CJ.
\end{equation}
After also exposing the reservoir, the query point still has law $\mu$. On $\mathcal E_G$, \eqref{eq:reference-event-covariance} and $\delta\le1$ imply
\[
 \E_u[\widetilde f_G(u)\widetilde f_G(u)^\top]\preceq4I_d.
\]
Using $\Phi_w^G(i)=d^{-1/2}\widetilde f_G(u_i^G)^\top O_Gb_w^G$ and \eqref{eq:color-full-reservoir-directions} therefore gives
\[
 \E[\|\Phi^G(i)\|^2\mid\mathcal F_G]\le4C^VV^CJ.
\]

Both bounds are polylogarithmic, because $V=\lceil C_0\log\log n\rceil$ and $J=K^{V-1}$. Multiplying \eqref{eq:color-signature-mean-square} by the $\mathcal F_G$-measurable indicator $\mathbf1_{\mathcal E_G}$ and conditioning down to $\mathcal L$ proves (i), after enlarging $\kappa_{\mathrm{col}}$, since $p_{\mathrm{col}}^{-1}\le e^V\le e(\log n)^{C_0}$.

Let $\mathcal X_{\mathrm{col}}$ consist of the labels in $\mathcal I$ violating either bound in \eqref{eq:color-accuracy-good-queries} for at least one graph. By (i) and Markov's inequality, on $\mathcal A$ each violation, intersected with its event $\mathcal E_G$, has conditional probability at most $1/\log n$: for the error bound this uses $M\ge(\log n)^{2\kappa_{\mathrm{col}}+2}$. Summing over the two bounds and the two graphs gives
\[
 \E[\mathbf1_{\mathcal E_A\cap\mathcal E_B}|\mathcal X_{\mathrm{col}}|\mid\mathcal L]\le\frac{4n}{\log n},\qquad
 \P[\mathcal E_A\cap\mathcal E_B\cap\{|\mathcal X_{\mathrm{col}}|>n/\sqrt{\log n}\}\mid\mathcal L]\le\frac4{\sqrt{\log n}}.
\]
Finally, $\P[\mathcal A^c]\le2Km/n=o(1)$ by \S\ref{ss:preprocessing}, and $\P[\mathcal A\cap\mathcal E_G^c]=o(1)$ by Lemma~\ref{lem:seeded-local}. This proves (ii), including successful preprocessing.
\end{proof}

\subsection{Recovery and running time}\label{ss:color-recovery}

Let $\widetilde\pi_{S_A,S_B}$ be the candidate obtained from the scores $\widetilde S(i,j)$ in \eqref{eq:color-coded-signatures} by the threshold $\theta_n$ of \eqref{eq:def-theta-n} and the maximal matching and completion rules of \S\ref{ss:candidate-matching}.

\begin{ppn}[Accuracy of the color-coded candidate]\label{prop:color-coded-candidate}
Let $b,L,L_0,K,C_0$ be as in Proposition~\ref{prop:true-seed-candidate}, and let $M=M_{\mathrm{col}}$ be given by \eqref{eq:coloring-budget} with $C_{\mathrm{col}}\ge2\kappa_{\mathrm{col}}+2$, where $\kappa_{\mathrm{col}}$ is the constant of Lemma~\ref{lem:color-accuracy-good-queries}. Uniformly over $p\in[p_0,1-p_0]$, $d\ge(\log n)^L$ and $\pi$, with probability $1-o(1)$ the true seed guess is admissible, its preprocessing succeeds, and
\[
 |\{i\in[n]:\widetilde\pi_{S_A,S_B}(i)\ne\pi(i)\}|\le6cn+o(n).
\]
\end{ppn}

\begin{proof}
Relabel $B$ so that $\pi=\mathrm{id}$. Remark~\ref{rmk:coordinate-score-separation}, or the proof in \S\ref{ss:high-proof}, gives with probability $1-o(1)$ a set $\mathcal G\subseteq\mathcal I$, with $|\mathcal I\setminus\mathcal G|=o(n)$, on which
\[
 S(i,i)>5Jt_n/8,\qquad S(i,j)<3Jt_n/8\quad(i\ne j).
\]
In the high-dimensional regime, include the $o(n)$ labels lacking the stronger true-pair bound in the exceptional set.

On the event of Lemma~\ref{lem:color-accuracy-good-queries}(ii), expand the difference of the two inner products defining $\widetilde S$ and $S$ and apply Cauchy--Schwarz to obtain
\[
 \max_{i,j\in\mathcal I\setminus\mathcal X_{\mathrm{col}}}|\widetilde S(i,j)-S(i,j)|
 \le2+(\log n)^{-(\kappa_{\mathrm{col}}+1)}\le3.
\]
Since $Jt_n^2\ge(\log n)^4$ and $t_n\le1$, we have $3<Jt_n/8$ for large $n$. Thus on $\mathcal G'=\mathcal G\setminus\mathcal X_{\mathrm{col}}$ every true pair exceeds $\theta_n=Jt_n/2$ and every false pair lies below it. Lemma~\ref{lem:threshold-matching} and \eqref{eq:physical-query-exclusions} bound the errors of the completed maximal matching by
\[
 3|[n]\setminus\mathcal G'|
 \le3\big(|[n]\setminus\mathcal I|+|\mathcal I\setminus\mathcal G|+|\mathcal X_{\mathrm{col}}|\big)
 \le6cn+o(n).
\]
Admissibility and successful preprocessing also follow from Lemma~\ref{lem:color-accuracy-good-queries}(ii). Undoing the relabeling proves the proposition.
\end{proof}

\begin{proof}[Proof of Proposition~\ref{prop:true-seed-candidate-color-coding}]
Accuracy is Proposition~\ref{prop:color-coded-candidate}. Fix a seed guess and a graph. Computing the estimated seed directions costs $O(Kmd)=O(n^2)$ in the intermediate regime. For one word and coloring, Algorithms~\ref{alg:color-graph} and~\ref{alg:color-coordinate} cost $O(V2^VN^2)$ and $O(V2^VNd)$, respectively: each suffix color set requires a matrix-vector multiplication in the first case and one suffix sum followed by $O(d)$ updates in the second. As $N\le n$ and $d\le n$ in the intermediate regime, all messages cost $O(MJV2^Vn^2)$. Normalization and query evaluation cost $O(nNJ)$, respectively $O(ndJ)$. Since $M,J,V,2^V$ are polylogarithmic, the total is $O(n^{2+o(1)})$ from the preprocessing outputs. The remaining costs are counted in Proposition~\ref{prop:running-time}.
\end{proof}

Finally, we combine the candidates with selection and refinement. Let $\varepsilon=\varepsilon(p_0,\rho)$ be the constant of Proposition~\ref{prop:rounding-finishing}, which is proved in \S\ref{s:seeded-graph-matching}.

\begin{cor}[Recovery by the implemented algorithm]\label{cor:color-coded-recovery}
Choose $0<c<\min\{1/10,\varepsilon/6\}$, the constants $b,L,L_0,K,C_0$ of Proposition~\ref{prop:true-seed-candidate} with $L\ge14$, and $C_{\mathrm{col}}\ge2\kappa_{\mathrm{col}}+2$. Then Algorithm~\ref{alg:rgg-matching} uses $O(n^{K+2+o(1)}+n^3)$ arithmetic operations, and its output satisfies $\P[\widehat\pi=\pi]=1-o(1)$, uniformly over $p\in[p_0,1-p_0]$, $d\ge(\log n)^L$ and $\pi$.
\end{cor}

\begin{proof}[Proof of Corollary~\ref{cor:color-coded-recovery} and Theorem~\ref{thm:main-alg}]
Proposition~\ref{prop:running-time} gives the stated cost. With probability $1-o(1)$, Proposition~\ref{prop:color-coded-candidate} ensures that the list contains a candidate with at most $6cn+o(n)\le\varepsilon n$ errors. Proposition~\ref{prop:rounding-finishing} then returns $\pi$, since $d\ge(\log n)^{14}$. Its event holds simultaneously for every candidate list, so it applies to the list computed from the graphs and colorings. Finally, $O(n^{K+2+o(1)}+n^3)\subseteq O(n^{K+3})$, proving Theorem~\ref{thm:main-alg} with $C$ equal to $3$ plus the larger of the two regime-specific values of $K$.
\end{proof}
\section{Selection and refinement}\label{s:seeded-graph-matching}

The seed enumeration produces a list $\mathcal C\subseteq\mathfrak S_n$ with $|\mathcal C|\le n^K$. We select a candidate by adjacency overlap and then correct its remaining errors using neighborhood scores. The estimates below hold simultaneously for all permutations, so the list may depend on the observed graphs and the colorings.

If $\mathcal C$ is empty, return the identity. Otherwise, select
\[
 F_\tau=\frac1{n^2}\sum_{1\le i<j\le n}(A_{ij}-p)(B_{\tau(i),\tau(j)}-p),\qquad
 \tau_*\in\operatorname*{argmax}_{\tau\in\mathcal C}F_\tau,
\]
breaking ties lexicographically. For a permutation $\tau$, define
\[
 Q_{ij}(\tau)=\sum_{k=1}^n(A_{ik}-p)(B_{j,\tau(k)}-p),\qquad i,j\in[n].
\]
The centered diagonal entries equal $-p$. Refinement keeps $\tau_*$ fixed and sets $\widehat\pi(i)=\min\operatorname*{argmax}_jQ_{ij}(\tau_*)$ for every vertex, including the guessed seeds. It returns $\widehat\pi$ if these labels are distinct, and the identity otherwise. Selection and refinement cost $O(|\mathcal C|n^2+n^3)$ arithmetic operations.

\begin{algorithm}[htbp]
\caption{Selection and refinement}\label{alg:seeded-graph-matching}
\begin{algorithmic}[1]
\State \textbf{Input:} $A,B\in\{0,1\}^{n\times n}$, $p$, and a list $\mathcal C\subseteq\mathfrak S_n$.
\If{$\mathcal C=\varnothing$}
    \State \Return $\mathrm{id}$.
\EndIf
\ForAll{$\tau\in\mathcal C$}
    \State $F_\tau\gets n^{-2}\displaystyle\sum_{1\le i<j\le n}
                  (A_{ij}-p)(B_{\tau(i),\tau(j)}-p)$.
\EndFor
\State Let $\tau_*$ be the lexicographically smallest maximizer of $F_\tau$ over $\tau\in\mathcal C$.
\For{$i=1,\ldots,n$}
    \For{$j=1,\ldots,n$}
        \State $Q_{ij}(\tau_*)\gets\displaystyle\sum_{k=1}^n
                       (A_{ik}-p)(B_{j,\tau_*(k)}-p)$.
    \EndFor
    \State $\widehat\pi(i)\gets\min\operatorname*{argmax}_{j\in[n]}Q_{ij}(\tau_*)$.
\EndFor
\If{$|\{\widehat\pi(i):i\in[n]\}|<n$}
    \State \Return $\mathrm{id}$.
\EndIf
\State \Return $\widehat\pi$.
\end{algorithmic}
\end{algorithm}

\subsection{Uniform estimates}\label{ss:selection-estimates}

Write $q_\pi(\tau)=|\{i:\tau(i)\ne\pi(i)\}|$, and define
\[
 \gamma_d=\E[(A_{12}-p)(B_{\pi(1),\pi(2)}-p)],\qquad
 \gamma_*=\inf_{q\in[p_0,1-p_0]}q(1-q)r(q,\rho).
\]
Then $p_0(1-p_0)r_{\min}\le\gamma_*\le1/4$. For a standard bivariate normal pair $(X,Y)$ of correlation $s$, write $\psi_p(s)=\P[X\ge z_p,Y\ge z_p]-p^2$. By \eqref{eq:r-p-rho}, $\psi_p(\rho^2)=p(1-p)r(p,\rho)\ge\gamma_*$. Mehler's formula \cite[Eq.~18.18.28]{dlmf} and Parseval's identity for the centered threshold indicator give
\begin{equation}\label{eq:psi-bound}
 |\psi_p(s)|\le p(1-p)|s|\le |s|/4,\qquad s\in[-1,1].
\end{equation}

\begin{lem}[Cap correlation]\label{lem:cap-correlation}
Let $u,v\in\mathbb S^{d-1}$ and let $(u^A,u^B)$ be a $\rho$-pair. Uniformly in $u,v$,
\begin{equation}\label{eq:finishing-cap-correlation}
 \E[h_0(u,u^A)h_0(v,u^B)]=\psi_p(\rho\langle u,v\rangle)+O(d^{-1/2}).
\end{equation}
In particular, $\gamma_d=p(1-p)r(p,\rho)+O(d^{-1/2})$, and $\gamma_d\ge\gamma_*/2$ for $d\ge C$.
\end{lem}

\begin{proof}
Let $U$ be a uniform orthonormal two-frame. Set $a=(\rho,\sqrt{1-\rho^2})^\top$ and represent the pair by $u^A=Ue_1$, $u^B=Ua$. By rotation invariance, $u,v$ may be taken in the first two coordinates. Equation~\eqref{eq:projected-frame-entropy} with $k=2$ and Pinsker's inequality compare the corresponding block of $\sqrt dU$ with four independent standard normals in total variation by $O(d^{-1/2})$. The two resulting linear forms have unit variances and correlation $\rho\langle u,v\rangle$. Moreover, $|\sqrt d\tau_{p,d}-z_p|\le C/d$ by the proof of Lemma~\ref{lem:local-tv}. Bounded Gaussian marginal densities make the cost of replacing the two thresholds $O(d^{-1})$. Each cap indicator has mean $p$, proving \eqref{eq:finishing-cap-correlation}.

For $\gamma_d$, condition on $(u_1^A,u_1^B)$ and apply this formula to the independent pair $(u_2^A,u_2^B)$. The identity $\langle u_1^A,u_1^B\rangle=\rho$ gives the remaining assertions.
\end{proof}

\begin{lem}[Neighborhood separation]\label{lem:neighborhood-separation}
Let $d\ge(\log n)^{14}$. With probability $1-o(1)$,
\begin{equation}\label{eq:finishing-neighborhood-separation}
 \min_{i\in[n]}Q_{i,\pi(i)}(\pi)\ge\frac34\gamma_*n,\qquad
 \max_{\substack{i,j\in[n]\\j\ne\pi(i)}}Q_{ij}(\pi)\le\frac14\gamma_*n.
\end{equation}
\end{lem}

\begin{proof}
Relabel $B$ so that $\pi=\mathrm{id}$. Conditional on $\sigma(u_i^A,u_i^B,u_j^A,u_j^B)$, the summands with $k\notin\{i,j\}$ are independent, bounded by one, and have mean $\psi_p(\rho\langle u_i^A,u_j^B\rangle)+O(d^{-1/2})$ by Lemma~\ref{lem:cap-correlation}. Hoeffding's inequality \cite[Theorem~2.2.6]{vershynin2018high} and a union bound give, with probability at least $1-2/n$, simultaneously for all $i,j$,
\[
 \left|Q_{ij}(\mathrm{id})-(n-|\{i,j\}|)\psi_p(\rho\langle u_i^A,u_j^B\rangle)\right|
 \le\sqrt{6n\log n}+\frac{Cn}{\sqrt d}+2=o(n).
\]
For $i=j$, the main term is $(n-1)\psi_p(\rho^2)\ge(n-1)\gamma_*$. For $i\ne j$, independence and the spherical cap bound \cite[Section~3.4]{vershynin2018high} imply $\max_{i\ne j}|\langle u_i^A,u_j^B\rangle|\le C\sqrt{\log n/d}$ with probability $1-O(n^{-1})$. By \eqref{eq:psi-bound}, all false-pair main terms are then $o(n)$. The stated margins follow.
\end{proof}

To control the overlaps of all $n!$ permutations, we smooth the cap kernel and apply Gaussian concentration.

\begin{lem}[Uniform overlap]\label{lem:finishing-certificate}
Let $d\ge(\log n)^{14}$. With probability $1-o(1)$,
\begin{equation}\label{eq:finishing-overlap-uniform}
 \sup_{\tau\in\mathfrak S_n}\left|F_\tau-\frac{\gamma_d}{2}\left(1-\frac{q_\pi(\tau)}n\right)^2\right|
 \le\frac C{\log n}.
\end{equation}
\end{lem}

\begin{proof}
Relabel $B$ so that $\pi=\mathrm{id}$. Two centered edge indicators in different graphs are uncorrelated unless their unordered label pairs coincide: for pairs sharing one label, condition on its latent pair and use the zero conditional means. If $t_2(\tau)$ counts the two-cycles of $\tau$, this gives
\begin{equation}\label{eq:expected-overlap}
 \E F_\tau=\frac{\gamma_d}{n^2}\left\{\binom{n-q_{\mathrm{id}}(\tau)}2+t_2(\tau)\right\}
 =\frac{\gamma_d}{2}\left(1-\frac{q_{\mathrm{id}}(\tau)}n\right)^2+O(n^{-1}),
\end{equation}
uniformly in $\tau$, since $t_2(\tau)\le q_{\mathrm{id}}(\tau)/2$ and $|\gamma_d|\le1/4$.

Let $\omega_n=(\log n)^{-4}$ and replace $h_0$ by
\[
 \bar h_0(u,w)=\min\left\{1,\max\left\{0,\frac12+\frac{\sqrt d(\langle u,w\rangle-\tau_{p,d})}{2\omega_n}\right\}\right\}-p.
\]
This kernel is bounded by one and has Lipschitz constant $\sqrt d/(2\omega_n)$ as a function of the inner product. Write $\bar F_\tau$ for the resulting overlap, and let $\nu_G$ count unordered pairs in $G$ with $|\langle u_i^G,u_j^G\rangle-\tau_{p,d}|<\omega_n/\sqrt d$. Since a permutation preserves the set of unordered pairs,
\begin{equation}\label{eq:smoothing-error}
 \sup_\tau|F_\tau-\bar F_\tau|\le\frac{\nu_A+\nu_B}{n^2}.
\end{equation}
The inner-product density is at most $C\sqrt d$, so the right side has mean $O(\omega_n)$. Markov's inequality bounds it by $1/\log n$ with probability $1-O((\log n)^{-3})$, while $|\E F_\tau-\E\bar F_\tau|\le C\omega_n$ uniformly in $\tau$.

Represent the latent pairs as $(u_i^A,u_i^B)=(U_ie_1,U_ia)$, where $a=(\rho,\sqrt{1-\rho^2})^\top$, $U_i=\Theta_i(\Theta_i^\top\Theta_i)^{-1/2}$ and the $\Theta_i$ are independent $d\times2$ standard Gaussian matrices; their polar factors are uniform orthonormal two-frames by rotation invariance. Let $\boldsymbol u^G$ have rows $(u_i^G)^\top$, and let $\mathcal E_\Theta$ be the event
\[
 \min_i s_{\min}(\Theta_i)\ge\sqrt d/2,\qquad
 \max_{G\in\{A,B\}}\|\boldsymbol u^G\|_{\mathrm{op}}\le C\sqrt{1+n/d}.
\]
The sub-Gaussian matrix bound \cite[Theorem~4.6.1]{vershynin2018high}, using the spherical sub-Gaussian estimate in \cite[Section~3.4]{vershynin2018high}, gives $\P[\mathcal E_\Theta^c]\le Cne^{-c_1d}$.

Write $\Theta=(\Theta_1,\ldots,\Theta_n)$ for the full Gaussian input, and mark quantities from a second input with primes. The perturbation bound for rectangular polar factors \cite[Theorem~2]{li1995polar} gives
\begin{equation}\label{eq:polar-lipschitz}
 \|U_i-U_i'\|_{\mathrm F}\le\frac{4\|\Theta_i-\Theta_i'\|_{\mathrm F}}{s_{\min}(\Theta_i)+s_{\min}(\Theta_i')}.
\end{equation}
For two inputs in $\mathcal E_\Theta$, the cap Lipschitz bound, Cauchy--Schwarz over the edges, and $\|XX^\top-YY^\top\|_{\mathrm F}\le(\|X\|_{\mathrm{op}}+\|Y\|_{\mathrm{op}})\|X-Y\|_{\mathrm F}$ imply
\begin{align*}
 |\bar F_\tau-\bar F_\tau'|
 &\le\frac{C\sqrt d}{\omega_nn}\sum_{G\in\{A,B\}}\|\boldsymbol u^G(\boldsymbol u^G)^\top-\boldsymbol u^{G\prime}(\boldsymbol u^{G\prime})^\top\|_{\mathrm F}\\
 &\le\Lambda_n\|\Theta-\Theta'\|_{\mathrm F},\qquad
 \Lambda_n=\frac{C\sqrt{n+d}}{\omega_nn\sqrt d}.
\end{align*}
Here \eqref{eq:polar-lipschitz} controls the two clouds simultaneously, and the bound is uniform in $\tau$.

By the McShane extension theorem \cite{mcshane1934extension}, followed by clipping to $[-1,1]$, $\bar F_\tau|_{\mathcal E_\Theta}$ has a globally $\Lambda_n$-Lipschitz extension $\widetilde F_\tau$. It satisfies $|\E\widetilde F_\tau-\E\bar F_\tau|\le2\P[\mathcal E_\Theta^c]$. Gaussian concentration \cite[Theorem~5.2.2]{vershynin2018high} yields
\[
 \P\left[|\widetilde F_\tau-\E\widetilde F_\tau|>\frac1{\log n}\right]
 \le2\exp\left(-\frac{c_1n^2d}{(n+d)(\log n)^{10}}\right)
 \le2e^{-c_2n(\log n)^4},
\]
where the last inequality uses $d\ge(\log n)^{14}$ and sufficiently large $n$. A union bound over $n!$ permutations, intersection with $\mathcal E_\Theta$, and \eqref{eq:smoothing-error} show that $\sup_\tau|F_\tau-\E F_\tau|\le C/\log n$ with probability $1-o(1)$. Equation~\eqref{eq:expected-overlap} finishes the proof.
\end{proof}

Let
\begin{equation}\label{eq:selection-event}
 \mathcal E_{\mathrm{sel}}=\{\eqref{eq:finishing-neighborhood-separation}\text{ and }\eqref{eq:finishing-overlap-uniform}\text{ hold}\},\qquad
 \P[\mathcal E_{\mathrm{sel}}]=1-o(1)
\end{equation}
for $d\ge(\log n)^{14}$. This event depends only on the observed graphs and is uniform over candidate lists.

\subsection{Correctness}\label{ss:selection-correctness}

\begin{ppn}[Correctness of selection and refinement]\label{prop:selection-rounding-correctness}
Let $d\ge(\log n)^{14}$ and let $n$ be sufficiently large. On $\mathcal E_{\mathrm{sel}}$:
\begin{enumerate}
 \item For every nonempty list $\mathcal C\subseteq\mathfrak S_n$ and $\tau_0\in\mathcal C$ with $q_\pi(\tau_0)\le n/2$, the selected candidate satisfies $q_\pi(\tau_*)\le q_\pi(\tau_0)+Cn/\log n$.
 \item If $q_\pi(\tau)<\gamma_*n/4$, then $\pi(i)$ uniquely maximizes $j\mapsto Q_{ij}(\tau)$ for every $i$.
\end{enumerate}
\end{ppn}

\begin{proof}
Set $y=1-q_\pi(\tau_*)/n$ and $y_0=1-q_\pi(\tau_0)/n\ge1/2$. Since $F_{\tau_*}\ge F_{\tau_0}$, \eqref{eq:finishing-overlap-uniform} gives $\gamma_d(y_0^2-y^2)/2\le2C/\log n$. If $y<y_0$, divide by $y_0+y\ge1/2$ and use $\gamma_d\ge\gamma_*/2$; otherwise the first assertion is immediate.

For refinement, changing $\pi$ to $\tau$ changes at most $q_\pi(\tau)$ summands of each neighborhood score, each by at most one. Hence
\[
 |Q_{ij}(\tau)-Q_{ij}(\pi)|\le q_\pi(\tau).
\]
By \eqref{eq:finishing-neighborhood-separation}, for $j\ne\pi(i)$ we obtain $Q_{i,\pi(i)}(\tau)-Q_{ij}(\tau)\ge\gamma_*n/2-2q_\pi(\tau)>0$, proving the second assertion.
\end{proof}

\begin{proof}[Proof of Proposition~\ref{prop:rounding-finishing}]
Take $\varepsilon=\gamma_*/8\le1/32$. On $\mathcal E_{\mathrm{sel}}$, a candidate with at most $\varepsilon n$ errors gives, by Proposition~\ref{prop:selection-rounding-correctness}, a selected permutation with at most $\gamma_*n/8+Cn/\log n<\gamma_*n/4$ errors. The same proposition then makes every refinement maximizer equal to $\pi(i)$, so the algorithm returns $\pi$. The event has probability $1-o(1)$ by \eqref{eq:selection-event} and holds simultaneously for all candidate lists.
\end{proof}
\section{Lower bound in the low-degree framework}\label{s:low-degree}

Let $\bbP=\bbP_{n,d,p,\rho}$ be the correlated law of Definition~\ref{def:correlated-RGG}, with $\pi$ uniform on $\mathfrak S_n$ and independent of the latent pairs, and let $\bbQ=\bbQ_{n,d,p}$ be the law of two independent spherical random geometric graphs. For polynomials of total degree at most $D$ in the upper triangular entries of both graphs, define
\[
 \adv_{\le D}(\bbP,\bbQ)=\sup_{\substack{\deg f\le D\\\operatorname{Var}_\bbQ(f)>0}}
 \frac{|\E_\bbP[f]-\E_\bbQ[f]|}{\sqrt{\operatorname{Var}_\bbQ(f)}}.
\]
Since $\bbQ$ is the product of the marginals of $\bbP$ on a finite space, $\bbP\ll\bbQ$. If $\Lambda^{\le D}$ is the $L^2(\bbQ)$ projection of $\Lambda=\de\bbP/\de\bbQ$ onto these polynomials, then $\|\Lambda^{\le D}\|_{L^2(\bbQ)}^2=1+\adv_{\le D}(\bbP,\bbQ)^2$; see \cite{kunisky2019notes}. Bounded low-degree advantage is a widely used benchmark for computational difficulty, although its interpretation as a general algorithmic lower bound requires a conjecture with known counterexamples \cite{buhai2025quasi,mao2026polynomial,wein2025computational}.

We prove Theorem~\ref{thm:main-hardness} in the following quantitative form.

\begin{thm}[Low-degree testing bound]\label{thm:low-degree-testing}
There are absolute constants $c_1,C>0$ such that, for every $n\ge2$, $d\ge3$, integer $D\ge1$, $p\in(0,1)$ and $\rho\in[0,1]$,
\begin{equation}\label{eq:low-degree-quantitative-bound}
 \eta_D:=D\left(\rho^2+\frac1d\right)\le c_1
 \quad\Longrightarrow\quad \adv_{\le D}(\bbP,\bbQ)\le C\eta_D.
\end{equation}
Consequently, if $D\rho^2\to0$ and $D/d\to0$, then $\adv_{\le D}(\bbP,\bbQ)=o(1)$ and $\|\Lambda^{\le D}\|_{L^2(\bbQ)}^2=1+o(1)$.
\end{thm}

The argument transfers each degree-$D$ statistic to Gram matrices of $D$ latent pairs, expands their $\chi^2$ divergence in invariant spherical harmonics, and bounds the resulting series. Notation in this section is local.

\subsection{Reduction to Gram matrices}\label{ss:reduce-small-Gram}

For $\boldsymbol u=(u_i)_{i\le N}$, write $\boldsymbol u_S=(u_i)_{i\in S}$. A function is \emph{invariant} if it is unchanged by applying the same $O\in O(d)$ to every point. A symmetric kernel $h\in L^2(\mu^{\otimes m})$ is \emph{canonical} if integrating any argument against $\mu$ gives zero almost everywhere; constants are canonical kernels of order zero. For $N\ge m$, set
\[
 \mathcal U_{N,m}h(\boldsymbol u)=\binom Nm^{-1/2}\sum_{\substack{S\subseteq[N]\\|S|=m}}h(\boldsymbol u_S),
 \qquad \mathcal U_{N,0}h=h.
\]
Define the fixed-angle channel by $(K_\rho g)(u)=\E[g(u_1^B)\mid u_1^A=u]$. Its tensor power $K_\rho^{\otimes m}$ conditions independently on $m$ latent pairs, and $K_\rho^{\otimes0}$ is the identity.

\begin{lem}\label{lem:canonical-u-stat}
For canonical kernels $h,g$ of orders $m,k\le N$ and $\boldsymbol u\sim\mu^{\otimes N}$,
\begin{equation}\label{eq:canonical-row-isometry}
 \E[\mathcal U_{N,m}h(\boldsymbol u)\mathcal U_{N,k}g(\boldsymbol u)]
 =\begin{cases}\langle h,g\rangle_{L^2(\mu^{\otimes m})},&m=k,\\0,&m\ne k.\end{cases}
\end{equation}
For independent $\rho$-pairs $(u_i^A,u_i^B)_{i\le N}$,
\begin{equation}\label{eq:canonical-paired-row-moment}
 \E[\mathcal U_{N,m}h(\boldsymbol u^A)\mathcal U_{N,k}g(\boldsymbol u^B)]
 =\begin{cases}\langle h,K_\rho^{\otimes m}g\rangle_{L^2(\mu^{\otimes m})},&m=k,\\0,&m\ne k.\end{cases}
\end{equation}
\end{lem}

\begin{proof}
The first identity is the standard orthogonality of canonical components in the Hoeffding decomposition \cite[Section~2]{efron1981jackknife}, with our normalization. For the paired identity, expand over supports $S,T$. If $S\ne T$, an index in their symmetric difference occurs in only one factor; its latent vector is independent of the other variables, and integration against $\mu$ annihilates that factor. For $S=T$, conditioning on $\boldsymbol u_S^A$ gives $\langle h,K_\rho^{\otimes m}g\rangle$. The $\binom Nm$ equal-support terms cancel the normalization.
\end{proof}

We also use the classical Hoeffding decomposition \cite[Section~2]{efron1981jackknife}. If a bounded, Borel, symmetric and invariant function $H(\boldsymbol u)$ is a sum of functions depending on at most $v$ points each, then
\begin{equation}\label{eq:hoeffding-decomposition}
 H=\sum_{m=0}^{\min(v,N)}\mathcal U_{N,m}\left(\binom Nm^{1/2}h_m\right),
\end{equation}
where $h_0=\E H$ and each $h_m$ is bounded, symmetric, canonical and invariant. Explicitly, its component on a support $S$ is $\sum_{T\subseteq S}(-1)^{|S|-|T|}\E[H\mid\boldsymbol u_T]$, with conditional expectation taken by integrating the remaining points against $\mu$; this also shows that the components vanish for $|S|>v$ and preserve invariance.

For $r\ge1$, let $\Gamma_r(\boldsymbol u)=(\langle u_i,u_j\rangle)_{i,j\le r}$, and let $\nu_r$ be its law under $\mu^{\otimes r}$. Write $\mathsf P_r$ for the joint law of
\[
 A^{\gram}=\Gamma_r(u_1^A,\ldots,u_r^A),\qquad
 B^{\gram}=\Gamma_r(u_1^B,\ldots,u_r^B)
\]
from $r$ independent $\rho$-pairs, and $\mathsf Q_r=\nu_r\otimes\nu_r$.

\begin{lem}\label{lem:gram-invariant-functions}
Every bounded, Borel, invariant function $\psi$ on $(\mathbb S^{d-1})^r$ has the form $\psi=\phi\circ\Gamma_r$ for a bounded Borel function $\phi$ on the range of $\Gamma_r$. If $\psi$ is continuous, so is $\phi$.
\end{lem}

\begin{proof}
Equal Gram matrices determine tuples related by an orthogonal map, so the fibers of $\Gamma_r$ are precisely the orbits. Compactness makes $\Gamma_r$ a quotient map, proving the continuous assertion. For the Borel assertion, average functions over $O(d)$ with Haar probability measure. Averages of continuous functions factor continuously through $\Gamma_r$, and the class whose averages factor measurably is closed under bounded pointwise limits. The functional monotone class theorem gives the assertion for all bounded Borel functions; averaging leaves invariant functions unchanged.
\end{proof}

The following is the spherical counterpart of \cite[Lemma~3.8]{li2026computational}.

\begin{lem}\label{lem:low-degree-gram-reduction}
For every $n\ge2$, $d\ge3$ and integer $D\ge1$,
\[
 \adv_{\le D}(\bbP,\bbQ)^2\le\chi^2(\mathsf P_D\|\mathsf Q_D).
\]
\end{lem}

\begin{proof}
Fix $f$ of degree at most $D$. We construct a bounded Borel function $\Phi$ of two $D\times D$ Gram matrices such that
\begin{equation}\label{eq:transferred-gram-statistic}
 \E_{\mathsf P_D}\Phi=\E_\bbP f-\E_\bbQ f,\qquad
 \E_{\mathsf Q_D}\Phi=0,\qquad
 \E_{\mathsf Q_D}\Phi^2\le\operatorname{Var}_\bbQ(f).
\end{equation}
Average $f-\E_\bbQ f$ over independent relabelings of the two graphs, obtaining $\bar f$. Both laws are invariant under these relabelings, and averaging contracts the $L^2(\bbQ)$ norm. If $\bbP_{\mathrm{id}}$ denotes the correlated law with $\pi=\mathrm{id}$, then
\begin{equation}\label{eq:symmetrized-statistic}
 \E_{\bbP_{\mathrm{id}}}\bar f=\E_\bbP f-\E_\bbQ f,\qquad
 \E_\bbQ\bar f=0,\qquad
 \E_\bbQ\bar f^2\le\operatorname{Var}_\bbQ(f).
\end{equation}
Here the first identity uses the uniform hidden permutation and the relabeling invariance of $\bar f$.

The entries are binary, so $f$ may be multilinear. A monomial with $a$ edges of $A$ and $b$ edges of $B$, $a+b\le D$, uses at most $2a$ points in the first cloud and $2b$ in the second. Its separate relabeling average is a product of symmetric invariant functions. Expanding both factors by \eqref{eq:hoeffding-decomposition} gives
\[
 \bar f=\sum_{m+k\le2D}G_{m,k},\qquad
 G_{m,k}=\sum_{j\in J_{m,k}}\mathcal U_{n,m}h_j(\boldsymbol u^A)\mathcal U_{n,k}g_j(\boldsymbol u^B),
\]
with finite index sets and bounded, symmetric, canonical, invariant kernels; only $m,k\le n$ occur. Under $\bbQ$, these groups are orthogonal by \eqref{eq:canonical-row-isometry}, and the constant group is zero. Under $\bbP_{\mathrm{id}}$, groups of unequal orders have mean zero by \eqref{eq:canonical-paired-row-moment}. Thus $\widetilde f=\sum_{m=1}^{\min(D,n)}G_{m,m}$ satisfies
\begin{equation}\label{eq:equal-order-part}
 \E_{\bbP_{\mathrm{id}}}\widetilde f=\E_{\bbP_{\mathrm{id}}}\bar f,\qquad
 \E_\bbQ\widetilde f=0,\qquad
 \E_\bbQ\widetilde f^2\le\E_\bbQ\bar f^2.
\end{equation}

Replace every $\mathcal U_{n,m}$ in $\widetilde f$ by $\mathcal U_{D,m}$, keeping the kernels and evaluating them on clouds of $D$ points. By Lemma~\ref{lem:gram-invariant-functions}, the resulting statistic is a bounded Borel function $\Phi(A^{\gram},B^{\gram})$. The identities in Lemma~\ref{lem:canonical-u-stat} do not depend on the cloud size, so this replacement preserves the means under the paired and independent laws and the second moment under the independent law. Equations~\eqref{eq:symmetrized-statistic} and \eqref{eq:equal-order-part} prove \eqref{eq:transferred-gram-statistic}.

If the divergence is infinite the result is immediate. Otherwise Cauchy--Schwarz gives
\[
 (\E_\bbP f-\E_\bbQ f)^2
 =\left(\E_{\mathsf Q_D}\left[\Phi\left(\frac{\de\mathsf P_D}{\de\mathsf Q_D}-1\right)\right]\right)^2
 \le\chi^2(\mathsf P_D\|\mathsf Q_D)\operatorname{Var}_\bbQ(f).
\]
Taking the supremum over $f$ concludes the proof.
\end{proof}

\subsection{Invariant harmonic expansion}\label{ss:explicit-chi-square-calculation}

Let $\mathcal H_\ell$ be the spherical harmonics of degree $\ell$, with $\mathcal H_0$ the constants, and let $P_\ell^{(d)}(1)=1$ normalize the zonal polynomials as in \S\ref{ss:intermediate-spherical-estimates}. The classical zonal uniqueness theorem \cite[Lemma~7.1]{dai2013spherical} states that
\[
 \{h\in\mathcal H_\ell:h\circ O=h\text{ whenever }Ou=u\}
 =\bbR P_\ell^{(d)}(\langle u,\cdot\rangle).
\]
The Funk--Hecke formula \cite[Theorem~2.9]{dai2013spherical}, applied to normalized bands shrinking to the latitude $\langle u,v\rangle=\rho$, gives
\begin{equation}\label{eq:fixed-angle-harmonic-multiplier}
 K_\rho|_{\mathcal H_\ell}=P_\ell^{(d)}(\rho)I.
\end{equation}
For $\boldsymbol\ell\in\bbN_0^r$, set $\mathcal H_{\boldsymbol\ell}=\bigotimes_{i=1}^r\mathcal H_{\ell_i}$ and $m_{\boldsymbol\ell}=\dim\mathcal H_{\boldsymbol\ell}^{O(d)}$, where $O(d)$ acts simultaneously on all arguments. Only these invariant subspaces represent functions of the Gram matrix.

\begin{lem}\label{lem:gram-invariant-channel}
For $d\ge3$ and $r\ge1$, with both sides allowed to be infinite,
\[
 1+\chi^2(\mathsf P_r\|\mathsf Q_r)
 =\sum_{\boldsymbol\ell\in\bbN_0^r}m_{\boldsymbol\ell}\prod_{i=1}^rP_{\ell_i}^{(d)}(\rho)^2.
\]
\end{lem}

\begin{proof}
Choose orthonormal bases of $\mathcal H_{\boldsymbol\ell}^{O(d)}$. By Lemma~\ref{lem:gram-invariant-functions}, they have the form $\phi_j\circ\Gamma_r$ for continuous functions $\phi_j$, with $\phi_0=1$. Together they give an orthonormal basis of $L^2(\nu_r)$. If $\boldsymbol\ell(j)$ is the multidegree and $\kappa_j=\prod_iP_{\ell_i(j)}^{(d)}(\rho)$, then \eqref{eq:fixed-angle-harmonic-multiplier} gives
\begin{equation}\label{eq:gram-channel-basis-moments}
 \E_{\mathsf P_r}[\phi_j(A^{\gram})\phi_k(B^{\gram})]
 =\langle\phi_j\circ\Gamma_r,K_\rho^{\otimes r}(\phi_k\circ\Gamma_r)\rangle
 =\kappa_j\mathbf1\{j=k\}.
\end{equation}
Thus any $L^2$ density has squared norm at least $\sum_j\kappa_j^2$ by Bessel's inequality, which also treats the case of a divergent series.

If $\sum_j\kappa_j^2<\infty$, let $\Lambda_r=\sum_j\kappa_j\phi_j\otimes\phi_j$ in $L^2(\mathsf Q_r)$. A polynomial in Gram entries pulls back to an invariant polynomial in the point coordinates. Its harmonic decomposition is finite, and each component remains invariant, so it is a finite linear combination of the $\phi_j$. Equation~\eqref{eq:gram-channel-basis-moments} therefore makes $\mathsf P_r$ and $\Lambda_r\mathsf Q_r$ agree on all polynomial tests in the two Gram matrices. Polynomial density on the compact Gram range, by Stone--Weierstrass, identifies these finite signed measures. Consequently $\Lambda_r=\de\mathsf P_r/\de\mathsf Q_r$, and Parseval gives $1+\chi^2=\sum_j\kappa_j^2$. Grouping by multidegree proves the formula.
\end{proof}

\subsection{Proof of Theorem~\ref{thm:low-degree-testing}}\label{ss:conclude-low-degree-lower-bound}

For $\boldsymbol\ell\in\bbN_0^r$, let $\mathcal M(\boldsymbol\ell)$ be the loopless undirected multigraphs on $[r]$ with degree sequence $\boldsymbol\ell$; parallel edges are allowed.

\begin{lem}\label{lem:gram-multigraph-multiplicity}
For every $d\ge3$ and $\boldsymbol\ell\in\bbN_0^r$, $m_{\boldsymbol\ell}\le|\mathcal M(\boldsymbol\ell)|$.
\end{lem}

\begin{proof}
Identify each $\mathcal H_{\ell_i}$ with the symmetric trace-free tensors of order $\ell_i$. The first fundamental theorem for the orthogonal group \cite[Theorem~3.1]{lehrer2012second} states that invariant tensors of even total order are spanned by pairwise contractions, and those of odd order vanish. Project these tensors onto $\mathcal H_{\boldsymbol\ell}$; the projection commutes with $O(d)$, so their images span its invariant subspace. A pairing within one block contracts a trace and has zero projection. Every remaining pairing defines a loopless multigraph of degree sequence $\boldsymbol\ell$. Pairings defining the same multigraph differ only by permutations within the symmetric blocks and hence have the same projection. This proves the bound.
\end{proof}

For $0\le\rho<1$, let $Z,Z'$ be independent $\operatorname{Beta}(1/2,(d-2)/2)$ variables, and set $X=\rho^2+(1-\rho^2)Z$, $X'=\rho^2+(1-\rho^2)Z'$ and $Y=\sqrt{XX'}$.

\begin{lem}\label{lem:gram-positive-envelope}
For $d\ge3$, $0\le\rho<1$ and every integer $\ell\ge0$, $P_\ell^{(d)}(\rho)^2\le\E[Y^\ell]$.
\end{lem}

\begin{proof}
The Laplace representation \cite[Eq.~18.10.4]{dlmf} gives $P_\ell^{(d)}(\rho)=\E[(\rho+\mathrm i\sqrt{1-\rho^2}\,\omega_1)^\ell]$, where $\omega$ is uniform on $\mathbb S^{d-2}$. Since $\omega_1^2\sim\operatorname{Beta}(1/2,(d-2)/2)$, the squared modulus inside the expectation has law $X$. The triangle inequality and independence give $P_\ell^{(d)}(\rho)^2\le(\E X^{\ell/2})^2=\E Y^\ell$.
\end{proof}

\begin{proof}[Proof of Theorem~\ref{thm:low-degree-testing}]
We prove
\begin{equation}\label{eq:small-cloud-gram-divergence}
 \eta:=r(\rho^2+d^{-1})\le\tfrac14
 \quad\Longrightarrow\quad \chi^2(\mathsf P_r\|\mathsf Q_r)\le11\eta^2.
\end{equation}
Let $q=\rho^2$, $b=(d-2)/2$ and $s=(r-1)/2$. The assumption implies $d\ge4r\ge2r+2$ and
\begin{equation}\label{eq:low-degree-parameter-range}
 q\le\tfrac14,\qquad b-s=(d-r-1)/2\ge d/4.
\end{equation}

Let $Y_1,\ldots,Y_r$ be independent copies of $Y$. Lemmas~\ref{lem:gram-invariant-channel}--\ref{lem:gram-positive-envelope} bound the channel series by
\[
 \sum_{\boldsymbol\ell}|\mathcal M(\boldsymbol\ell)|\E\prod_iY_i^{\ell_i}.
\]
Indexing multigraphs by edge multiplicities $k_{ij}\ge0$ and summing their geometric series gives, by Tonelli's theorem,
\[
 1+\chi^2(\mathsf P_r\|\mathsf Q_r)
 \le\E\prod_{i<j}(1-Y_iY_j)^{-1}.
\]
The inequality $1-xy\ge\sqrt{(1-x^2)(1-y^2)}$ for $x,y\in[0,1)$, followed by independence, yields
\begin{equation}\label{eq:gram-one-variable}
 1+\chi^2(\mathsf P_r\|\mathsf Q_r)\le F_r^r,\qquad F_r=\E(1-XX')^{-s}.
\end{equation}
Write $(s)_k=s(s+1)\cdots(s+k-1)$ for the rising factorial. Expanding the nonnegative power series and using $\E X^k\le\E X$ for $k\ge1$ gives
\begin{equation}\label{eq:gram-series-bound}
 F_r-1=\sum_{k\ge1}\frac{(s)_k}{k!}(\E X^k)^2
 \le\E X\big(\E(1-X)^{-s}-1\big).
\end{equation}

Since $1-X=(1-q)(1-Z)$ and $s<b$, the beta integral gives
\[
 \E(1-X)^{-s}=(1-q)^{-s}\frac{B(1/2,b-s)}{B(1/2,b)},\qquad
 B(x,y)=\frac{\Gamma(x)\Gamma(y)}{\Gamma(x+y)}.
\]
For $a(t)=\log\Gamma(t+1/2)-\log\Gamma(t)$, the digamma integral formula \cite[Eq.~5.9.16]{dlmf} gives $a'(t)\le1/t$ for $t>0$. Consequently $\log[B(1/2,b-s)/B(1/2,b)]\le s/(b-s)$. Combining this with $-\log(1-q)\le2q$ and \eqref{eq:low-degree-parameter-range}, we obtain
\[
 \log\E(1-X)^{-s}\le2sq+\frac{s}{b-s}\le rq+\frac{2r}d\le2\eta.
\]
Also $\E X=q+(1-q)/(d-1)\le2(q+d^{-1})$. Equations~\eqref{eq:gram-one-variable}--\eqref{eq:gram-series-bound} now imply
\[
 \log(1+\chi^2(\mathsf P_r\|\mathsf Q_r))
 \le r(F_r-1)\le r\E X(e^{2\eta}-1)
 \le4\eta^2e^{2\eta}\le7\eta^2.
\]
Thus $\chi^2\le e^{7\eta^2}-1\le7e^{7/16}\eta^2\le11\eta^2$, proving \eqref{eq:small-cloud-gram-divergence}. Apply Lemma~\ref{lem:low-degree-gram-reduction} with $r=D$ to obtain \eqref{eq:low-degree-quantitative-bound} with $c_1=1/4$ and $C=4$. The asymptotic assertions follow from $\eta_D\to0$ and the likelihood-ratio identity above.
\end{proof}
\bibliographystyle{alpha}
\bibliography{algorithm}

@article{alon1995color,
  title={Color-coding},
  author={Alon, Noga and Yuster, Raphael and Zwick, Uri},
  journal={Journal of the ACM},
  volume={42},
  number={4},
  pages={844--856},
  year={1995},
  publisher={ACM New York, NY, USA}
}

@article{ding2025low,
  title={Low-degree hardness of detection for correlated {Erd{\H{o}}s--R{\'e}nyi} graphs},
  author={Ding, Jian and Du, Hang and Li, Zhangsong},
  journal={The Annals of Statistics},
  volume={53},
  number={5},
  pages={1833--1856},
  year={2025},
  publisher={Institute of Mathematical Statistics}
}

@article{ding2025polynomial2,
  title={A polynomial-time iterative algorithm for random graph matching with nonvanishing correlation},
  author={Ding, Jian and Li, Zhangsong},
  journal={Mathematics of Operations Research},
  year={2025},
  publisher={INFORMS}
}

@article{ding2025polynomial,
  title={A polynomial time iterative algorithm for matching {G}aussian matrices with non-vanishing correlation},
  author={Ding, Jian and Li, Zhangsong},
  journal={Foundations of Computational Mathematics},
  volume={25},
  number={4},
  pages={1287--1344},
  year={2025},
  publisher={Springer}
}

@article{li2026computational,
  title={A Computational Transition for Detecting Multivariate Shuffled Linear Regression by Low-Degree Polynomials},
  author={Li, Zhangsong},
  journal={IEEE Transactions on Information Theory},
  year={2026},
  publisher={IEEE}
}

@inproceedings{mao2023random,
  title={Random graph matching at {O}tter's threshold via counting chandeliers},
  author={Mao, Cheng and Wu, Yihong and Xu, Jiaming and Yu, Sophie H},
  booktitle={Proceedings of the 55th Annual ACM symposium on theory of computing (STOC)},
  pages={1345--1356},
  year={2023},
  organization={ACM}
}

@article{mao2024testing,
  title={Testing network correlation efficiently via counting trees},
  author={Mao, Cheng and Wu, Yihong and Xu, Jiaming and Yu, Sophie H},
  journal={The Annals of Statistics},
  volume={52},
  number={6},
  pages={2483--2505},
  year={2024},
  publisher={Institute of Mathematical Statistics}
}

@article{niu2026high,
  title={High-Dimensional {P}rocrustes Matching via Tree Counts},
  author={Niu, Xiaochun and Schramm, Tselil and Xu, Jiaming},
  journal={arXiv preprint arXiv:2607.08538},
  year={2026}
}

@article{tropp2015introduction,
  title={An introduction to matrix concentration inequalities},
  author={Tropp, Joel A},
  journal={Foundations and Trends{\textregistered} in Machine Learning},
  volume={8},
  number={1-2},
  pages={1--230},
  year={2015},
  publisher={Emerald Publishing Limited}
}

@inproceedings{wang2022random,
  title={Random graph matching in geometric models: the case of complete graphs},
  author={Wang, Haoyu and Wu, Yihong and Xu, Jiaming and Yolou, Israel},
  booktitle={Proceedings of the 35th Conference on Learning Theory (COLT)},
  pages={3441--3488},
  year={2022},
  organization={PMLR}
}

@article{wu2022settling,
  title={Settling the sharp reconstruction thresholds of random graph matching},
  author={Wu, Yihong and Xu, Jiaming and Yu, Sophie H.},
  journal={IEEE Transactions on Information Theory},
  volume={68},
  number={8},
  pages={5391--5417},
  year={2022},
  publisher={IEEE}
}

@inproceedings{babai2016graph,
  title={Graph isomorphism in quasipolynomial time},
  author={Babai, L{\'a}szl{\'o}},
  booktitle={Proceedings of the forty-eighth annual ACM symposium on Theory of Computing (STOC)},
  pages={684--697},
  year={2016}
}

@article{makarychev2014maximum,
  title={Maximum quadratic assignment problem: Reduction from maximum label cover and {LP}-based approximation algorithm},
  author={Makarychev, Konstantin and Manokaran, Rajsekar and Sviridenko, Maxim},
  journal={ACM Transactions on Algorithms},
  volume={10},
  number={4},
  pages={1--18},
  year={2014},
  publisher={ACM New York, NY, USA}
}

@article{singh2008global,
  title={Global alignment of multiple protein interaction networks with application to functional orthology detection},
  author={Singh, Rohit and Xu, Jinbo and Berger, Bonnie},
  journal={Proceedings of the National Academy of Sciences},
  volume={105},
  number={35},
  pages={12763--12768},
  year={2008},
  publisher={National Academy of Sciences}
}

@inproceedings{narayanan2008robust,
  title={Robust de-anonymization of large sparse datasets},
  author={Narayanan, Arvind and Shmatikov, Vitaly},
  booktitle={2008 IEEE Symposium on Security and Privacy (SP 2008)},
  pages={111--125},
  year={2008},
  organization={IEEE}
}

@inproceedings{narayanan2009anonymizing,
  title={De-anonymizing social networks},
  author={Narayanan, Arvind and Shmatikov, Vitaly},
  booktitle={2009 30th IEEE symposium on security and privacy},
  pages={173--187},
  year={2009},
  organization={IEEE}
}

@inproceedings{haghighi2005robust,
  title={Robust textual inference via graph matching},
  author={Haghighi, Aria and Ng, Andrew Y and Manning, Christopher D},
  booktitle={Proceedings of Human Language Technology Conference and Conference on Empirical Methods in Natural Language Processing (EMNLP)},
  pages={387--394},
  year={2005}
}

@inproceedings{pedarsani2011privacy,
  title={On the privacy of anonymized networks},
  author={Pedarsani, Pedram and Grossglauser, Matthias},
  booktitle={Proceedings of the 17th ACM SIGKDD international conference on Knowledge discovery and data mining},
  pages={1235--1243},
  year={2011}
}

@article{cullina2016improved,
  title={Improved achievability and converse bounds for {Erd{\H{o}}s--R{\'e}nyi} graph matching},
  author={Cullina, Daniel and Kiyavash, Negar},
  journal={ACM SIGMETRICS performance evaluation review},
  volume={44},
  number={1},
  pages={63--72},
  year={2016},
  publisher={ACM New York, NY, USA}
}

@article{cullina2017exact,
  title={Exact alignment recovery for correlated {E}rd{\H{o}}s--{R}{\'e}nyi graphs},
  author={Cullina, Daniel and Kiyavash, Negar},
  journal={arXiv preprint arXiv:1711.06783},
  year={2017}
}

@article{hall2023partial,
  title={Partial recovery in the graph alignment problem},
  author={Hall, Georgina and Massouli{\'e}, Laurent},
  journal={Operations Research},
  volume={71},
  number={1},
  pages={259--272},
  year={2023},
  publisher={INFORMS}
}

@inproceedings{ganassali2021impossibility,
  title={Impossibility of partial recovery in the graph alignment problem},
  author={Ganassali, Luca and Massouli{\'e}, Laurent and Lelarge, Marc},
  booktitle={Conference on Learning Theory},
  pages={2080--2102},
  year={2021},
  organization={PMLR}
}

@article{wu2023testing,
  title={Testing correlation of unlabeled random graphs},
  author={Wu, Yihong and Xu, Jiaming and Yu, Sophie H},
  journal={The Annals of Applied Probability},
  volume={33},
  number={4},
  pages={2519--2558},
  year={2023},
  publisher={Institute of Mathematical Statistics}
}

@article{ding2023detection,
  title={Detection threshold for correlated {Erd{\H{o}}s--R{\'e}nyi} graphs via densest subgraph},
  author={Ding, Jian and Du, Hang},
  journal={IEEE Transactions on Information Theory},
  volume={69},
  number={8},
  pages={5289--5298},
  year={2023},
  publisher={IEEE}
}

@article{ding2023matching,
  title={Matching recovery threshold for correlated random graphs},
  author={Ding, Jian and Du, Hang},
  journal={The Annals of Statistics},
  volume={51},
  number={4},
  pages={1718--1743},
  year={2023},
  publisher={Institute of Mathematical Statistics}
}

@article{barak2019nearly,
  title={{(Nearly)} efficient algorithms for the graph matching problem on correlated random graphs},
  author={Barak, Boaz and Chou, Chi-Ning and Lei, Zhixian and Schramm, Tselil and Sheng, Yueqi},
  journal={Advances in Neural Information Processing Systems},
  volume={32},
  year={2019}
}

@article{ding2021efficient,
  title={Efficient random graph matching via degree profiles},
  author={Ding, Jian and Ma, Zongming and Wu, Yihong and Xu, Jiaming},
  journal={Probability Theory and Related Fields},
  volume={179},
  number={1},
  pages={29--115},
  year={2021},
  publisher={Springer}
}

@article{fan2023spectral,
  title={Spectral graph matching and regularized quadratic relaxations {I}: Algorithm and {G}aussian analysis},
  author={Fan, Zhou and Mao, Cheng and Wu, Yihong and Xu, Jiaming},
  journal={Foundations of Computational Mathematics},
  volume={23},
  number={5},
  pages={1511--1565},
  year={2023},
  publisher={Springer}
}

@article{fan2023spectral2,
  title={Spectral graph matching and regularized quadratic relaxations {II}: {Erd{\H{o}}s--R{\'e}nyi} graphs and universality},
  author={Fan, Zhou and Mao, Cheng and Wu, Yihong and Xu, Jiaming},
  journal={Foundations of Computational Mathematics},
  volume={23},
  number={5},
  pages={1567--1617},
  year={2023},
  publisher={Springer}
}

@article{mossel2020seeded,
  title={Seeded graph matching via large neighborhood statistics},
  author={Mossel, Elchanan and Xu, Jiaming},
  journal={Random Structures \& Algorithms},
  volume={57},
  number={3},
  pages={570--611},
  year={2020},
  publisher={Wiley Online Library}
}

@inproceedings{ganassali2020tree,
  title={From tree matching to sparse graph alignment},
  author={Ganassali, Luca and Massouli{\'e}, Laurent},
  booktitle={Conference on Learning Theory},
  pages={1633--1665},
  year={2020},
  organization={PMLR}
}

@article{ganassali2024correlation,
  title={Correlation detection in trees for planted graph alignment},
  author={Ganassali, Luca and Lelarge, Marc and Massouli{\'e}, Laurent},
  journal={The Annals of Applied Probability},
  volume={34},
  number={3},
  pages={2799--2843},
  year={2024},
  publisher={Institute of Mathematical Statistics}
}

@article{ganassali2024statistical,
  title={Statistical limits of correlation detection in trees},
  author={Ganassali, Luca and Massouli{\'e}, Laurent and Semerjian, Guilhem},
  journal={The Annals of Applied Probability},
  volume={34},
  number={4},
  pages={3701--3734},
  year={2024},
  publisher={Institute of Mathematical Statistics}
}

@article{mao2023exact,
  title={Exact matching of random graphs with constant correlation},
  author={Mao, Cheng and Rudelson, Mark and Tikhomirov, Konstantin},
  journal={Probability Theory and Related Fields},
  volume={186},
  number={1},
  pages={327--389},
  year={2023},
  publisher={Springer}
}

@article{racz2022correlated,
  title={Correlated randomly growing graphs},
  author={R{\'a}cz, Mikl{\'o}s Z and Sridhar, Anirudh},
  journal={The Annals of Applied Probability},
  volume={32},
  number={2},
  pages={1058--1111},
  year={2022},
  publisher={Institute of Mathematical Statistics}
}

@article{baumler2026correlated,
  title={Correlated uniform attachment trees},
  author={B{\"a}umler, Johannes and R{\'a}cz, Mikl{\'o}s Z and Ross, Nathan and Sridhar, Anirudh},
  journal={arXiv preprint arXiv:2606.02472},
  year={2026}
}

@article{racz2021correlated,
  title={Correlated stochastic block models: Exact graph matching with applications to recovering communities},
  author={R{\'a}cz, Mikl{\'o}s Z. and Sridhar, Anirudh},
  journal={Advances in Neural Information Processing Systems},
  volume={34},
  pages={22259--22273},
  year={2021}
}

@inproceedings{gaudio2022exact,
  title={Exact community recovery in correlated stochastic block models},
  author={Gaudio, Julia and R{\'a}cz, Mikl{\'o}s Z. and Sridhar, Anirudh},
  booktitle={Conference on Learning Theory},
  pages={2183--2241},
  year={2022},
  organization={PMLR}
}

@article{gong2024umeyama,
  title={The {U}meyama algorithm for matching correlated {G}aussian geometric models in the low-dimensional regime},
  author={Gong, Shuyang and Li, Zhangsong},
  journal={arXiv preprint arXiv:2402.15095},
  year={2024}
}

@article{even2024aligning,
  title={Aligning embeddings and geometric random graphs: Informational results and computational approaches for the {Procrustes--Wasserstein} problem},
  author={Even, Mathieu and Ganassali, Luca and Maier, Jakob and Massouli{\'e}, Laurent},
  journal={Advances in Neural Information Processing Systems},
  volume={37},
  pages={70730--70764},
  year={2024}
}

@article{chen2026computational,
  title={A computational transition for detecting correlated stochastic block models by low-degree polynomials},
  author={Chen, Guanyi and Ding, Jian and Gong, Shuyang and Li, Zhangsong},
  journal={The Annals of Statistics},
  volume={54},
  number={1},
  pages={226--251},
  year={2026},
  publisher={Institute of Mathematical Statistics}
}

@article{chen2025detecting,
  title={Detecting Correlation Efficiently in Stochastic Block Models: Breaking {O}tter's Threshold in the Entire Supercritical Regime},
  author={Chen, Guanyi and Ding, Jian and Gong, Shuyang and Li, Zhangsong},
  journal={arXiv preprint arXiv:2503.06464},
  year={2025}
}

@article{umeyama1988eigendecomposition,
  title={An eigendecomposition approach to weighted graph matching problems},
  author={Umeyama, Shinji},
  journal={IEEE Transactions on Pattern Analysis and Machine Intelligence},
  volume={10},
  number={5},
  pages={695--703},
  year={1988},
  publisher={IEEE}
}

@inproceedings{kothari2026smooth,
  title={Smooth trade-off for tensor {PCA} via sharp bounds for {K}ikuchi matrices},
  author={Kothari, Pravesh K and Xu, Jeff},
  booktitle={Proceedings of the 2026 Annual ACM-SIAM Symposium on Discrete Algorithms (SODA)},
  pages={2617--2632},
  year={2026},
  organization={SIAM}
}

@article{li2025smooth,
  title={A Smooth Computational Transition in Tensor {PCA}},
  author={Li, Zhangsong},
  journal={arXiv preprint arXiv:2509.09904},
  year={2025}
}

@article{ding2024subexponential,
  title={Subexponential-time algorithms for sparse {PCA}},
  author={Ding, Yunzi and Kunisky, Dmitriy and Wein, Alexander S and Bandeira, Afonso S},
  journal={Foundations of Computational Mathematics},
  volume={24},
  number={3},
  pages={865--914},
  year={2024},
  publisher={Springer}
}

@article{huang2025testing,
  title={Testing correlation in graphs by counting bounded degree motifs},
  author={Huang, Dong and Yang, Pengkun},
  journal={arXiv preprint arXiv:2510.25289},
  year={2025}
}

@inproceedings{hopkins2017efficient,
  title={Efficient {Bayesian} estimation from few samples: community detection and related problems},
  author={Hopkins, Samuel B and Steurer, David},
  booktitle={2017 IEEE 58th annual symposium on foundations of computer science (FOCS)},
  pages={379--390},
  year={2017},
  organization={IEEE}
}

@article{fernandez2026spectral,
  title={Spectral Concentration and Recovery in Sparse High-Dimensional Random Geometric Graphs},
  author={Fernandez V, Manuel and Zhu, Yizhe},
  journal={arXiv preprint arXiv:2607.14304},
  year={2026}
}

@article{huang2025information,
  title = 	 {Information-Theoretic Thresholds for the Alignments of Partially Correlated Graphs},
  author =       {Huang, Dong and Song, Xianwen and Yang, Pengkun},
  journal = 	 {IEEE Transactions on Information Theory},
  year = 	 {2025},
  volume = {71},
  number = {12},
  pages = {9674--9697},
  publisher = {IEEE}
}

@article{huang2026information,
  title={Information-theoretic and computational limits of correlation detection under graph sampling},
  author={Huang, Dong and Yang, Pengkun},
  journal={arXiv preprint arXiv:2601.13966},
  year={2026}
}

@unpublished{huang2026matching,
  author = {Huang, Dong and Jiao, Ankang and Li, Shuangping and Li, Zhangsong},
  title  = {Graph Matching in Random Geometric Graphs {II}: Information-theoretic limits},
  year = {2026},
  note   = {Manuscript}
}

@inproceedings{kunisky2019notes,
  title={Notes on computational hardness of hypothesis testing: Predictions using the low-degree likelihood ratio},
  author={Kunisky, Dmitriy and Wein, Alexander S and Bandeira, Afonso S},
  booktitle={Mathematical Analysis, its Applications and Computation},
  series={Springer Proceedings in Mathematics \& Statistics},
  volume={385},
  pages={1--50},
  year={2022},
  publisher={Springer},
  doi={10.1007/978-3-030-97127-4_1}
}

@book{penrose2003random,
  title={Random geometric graphs},
  author={Penrose, Mathew},
  volume={5},
  year={2003},
  publisher={Oxford University Press}
}

@article{dall2002random,
  title={Random geometric graphs},
  author={Dall, Jesper and Christensen, Michael},
  journal={Physical Review E},
  volume={66},
  number={1},
  pages={016121},
  year={2002},
  publisher={APS}
}

@article{bubeck2016testing,
  title={Testing for high-dimensional geometry in random graphs},
  author={Bubeck, S{\'e}bastien and Ding, Jian and Eldan, Ronen and R{\'a}cz, Mikl{\'o}s Z},
  journal={Random Structures \& Algorithms},
  volume={49},
  number={3},
  pages={503--532},
  year={2016},
  publisher={Wiley Online Library}
}

@article{liu2022testing,
  title={Testing thresholds for high-dimensional sparse random geometric graphs},
  author={Liu, Siqi and Mohanty, Sidhanth and Schramm, Tselil and Yang, Elizabeth},
  journal={SIAM Journal on Computing},
  pages={STOC22{\char45}125--STOC22{\char45}181},
  year={2024}
}

@article{mao2026random,
  title={Random geometric graphs with smooth kernels: sharp detection threshold and a spectral conjecture},
  author={Mao, Cheng and Wu, Yihong and Xu, Jiaming},
  journal={arXiv preprint arXiv:2602.14998},
  year={2026}
}

@article{du2026resolution,
  title={Resolution of the Detection Threshold Conjecture for Random Geometric Graphs in the $ d> n $ Regime},
  author={Du, Hang and Mao, Cheng and Sun, Nike and Wu, Yihong and Xu, Jiaming},
  journal={arXiv preprint arXiv:2607.02013},
  year={2026}
}

@article{brennan2020phase,
  title={Phase transitions for detecting latent geometry in random graphs},
  author={Brennan, Matthew and Bresler, Guy and Nagaraj, Dheeraj},
  journal={Probability Theory and Related Fields},
  volume={178},
  number={3},
  pages={1215--1289},
  year={2020},
  publisher={Springer}
}

@phdthesis{hopkins2018statistical,
  title={Statistical inference and the sum of squares method},
  author={Hopkins, Samuel},
  year={2018},
  school={Cornell University}
}

@inproceedings{buhai2025quasi,
  title={The quasi-polynomial low-degree conjecture is false},
  author={Buhai, Rares-Darius and Hsieh, Jun-Ting and Jain, Aayush and Kothari, Pravesh K},
  booktitle={2025 IEEE 66th Annual Symposium on Foundations of Computer Science (FOCS)},
  pages={2577--2590},
  year={2025},
  organization={IEEE}
}

@article{mao2026polynomial,
  title={The Polynomial-Time Low-Degree Conjecture is False},
  author={Mao, Songtao},
  journal={arXiv preprint arXiv:2607.20318},
  year={2026}
}

@article{wein2025computational,
  title={Computational complexity of statistics: New insights from low-degree polynomials},
  author={Wein, Alexander S},
  journal={arXiv preprint arXiv:2506.10748},
  year={2025}
}

@article{otter1948number,
  title={The number of trees},
  author={Otter, Richard},
  journal={Annals of Mathematics},
  volume={49},
  number={3},
  pages={583--599},
  year={1948},
  publisher={JSTOR}
}

@inproceedings{mao2024impossibility,
  title={Impossibility of latent inner product recovery via rate distortion},
  author={Mao, Cheng and Zhang, Shenduo},
  booktitle={2024 60th Annual Allerton Conference on Communication, Control, and Computing},
  pages={01--08},
  year={2024},
  organization={IEEE}
}

@inproceedings{valdivia2019latent,
  title={Latent distance estimation for random geometric graphs},
  author={Araya Valdivia, Ernesto and De Castro, Yohann},
  booktitle={Advances in Neural Information Processing Systems},
  pages={8721--8731},
  year={2019}
}

@article{li2023spectral,
  title={Spectral clustering in the {G}aussian mixture block model},
  author={Li, Shuangping and Schramm, Tselil},
  journal={arXiv preprint arXiv:2305.00979},
  year={2023}
}

@article{eldan2022community,
  title={Community detection and percolation of information in a geometric setting},
  author={Eldan, Ronen and Mikulincer, Dan and Pieters, Hester},
  journal={Combinatorics, Probability and Computing},
  volume={31},
  number={6},
  pages={1048--1069},
  year={2022},
  publisher={Cambridge University Press}
}

@article{mossel2018proof,
  title={A proof of the block model threshold conjecture},
  author={Mossel, Elchanan and Neeman, Joe and Sly, Allan},
  journal={Combinatorica},
  volume={38},
  number={3},
  pages={665--708},
  year={2018},
  publisher={Springer}
}

@inproceedings{bordenave2015non,
  title={Non-backtracking spectrum of random graphs: community detection and non-regular {Ramanujan} graphs},
  author={Bordenave, Charles and Lelarge, Marc and Massouli{\'e}, Laurent},
  booktitle={2015 IEEE 56th Annual Symposium on Foundations of Computer Science},
  pages={1347--1357},
  year={2015},
  organization={IEEE}
}

@article{krzakala2013spectral,
  title={Spectral redemption in clustering sparse networks},
  author={Krzakala, Florent and Moore, Cristopher and Mossel, Elchanan and Neeman, Joe and Sly, Allan and Zdeborov{\'a}, Lenka and Zhang, Pan},
  journal={Proceedings of the National Academy of Sciences},
  volume={110},
  number={52},
  pages={20935--20940},
  year={2013},
  publisher={National Academy of Sciences}
}

@article{montanaro2012hypercontractive,
  author        = {Montanaro, Ashley},
  title         = {Some applications of hypercontractive inequalities
                   in quantum information theory},
  journal       = {Journal of Mathematical Physics},
  volume        = {53},
  number        = {12},
  pages         = {122206},
  year          = {2012},
  eprint        = {1208.0161},
  archivePrefix = {arXiv},
  primaryClass  = {quant-ph}
}

@article{gotze2023higher,
  author  = {G{\"o}tze, Friedrich and Sambale, Holger},
  title   = {Higher order concentration on {Stiefel} and {Grassmann} manifolds},
  journal = {Electronic Journal of Probability},
  volume  = {28},
  number  = {79},
  pages   = {1--30},
  year    = {2023},
}

@book{vershynin2018high,
  author    = {Vershynin, Roman},
  title     = {High-Dimensional Probability: An Introduction with Applications in Data Science},
  series    = {Cambridge Series in Statistical and Probabilistic Mathematics},
  volume    = {47},
  publisher = {Cambridge University Press},
  year      = {2018},
}

@article{dai2013spherical,
  author  = {Dai, Feng and Xu, Yuan},
  title   = {Spherical Harmonics},
  journal = {arXiv preprint arXiv:1304.2585},
  year    = {2013},
  url     = {https://arxiv.org/abs/1304.2585}
}

@article{vignat2008wick,
  author  = {Vignat, Christophe and Bhatnagar, Shalabh},
  title   = {An extension of {Wick}'s theorem},
  journal = {Statistics \& Probability Letters},
  volume  = {78},
  number  = {15},
  pages   = {2404--2407},
  year    = {2008},
  url     = {https://arxiv.org/abs/0709.1999}
}

@book{horn1991topics,
  author    = {Horn, Roger A. and Johnson, Charles R.},
  title     = {Topics in Matrix Analysis},
  publisher = {Cambridge University Press},
  address   = {Cambridge},
  year      = {1991}
}

@misc{dlmf,
  key          = {DLMF},
  title        = {{NIST Digital Library of Mathematical Functions}},
  howpublished = {\url{https://dlmf.nist.gov/}},
  note         = {National Institute of Standards and Technology}
}

@article{ding2025efficiently,
  title={Efficiently matching random inhomogeneous graphs via degree profiles},
  author={Ding, Jian and Fei, Yumou and Wang, Yuanzheng},
  journal={The Annals of Statistics},
  volume={53},
  number={4},
  pages={1808--1832},
  year={2025},
  publisher={Institute of Mathematical Statistics}
}

@article{vassaux2026graph,
  title={Graph alignment in sparse inhomogeneous models via self-overlap},
  author={Vassaux, Louis},
  journal={arXiv preprint arXiv:2607.14948},
  year={2026}
}

@article{du2026sharp,
  title         = {Sharp detection thresholds for random geometric graphs with polynomial average degrees},
  author        = {Du, Hang and Mao, Cheng and Sun, Nike and Wu, Yihong and Xu, Jiaming},
  journal = {arXiv preprint arXiv:2610.08517},
  year          = {2026}
}

@article{efron1981jackknife,
  author  = {Efron, Bradley and Stein, Charles},
  title   = {The Jackknife Estimate of Variance},
  journal = {The Annals of Statistics},
  volume  = {9},
  number  = {3},
  pages   = {586--596},
  year    = {1981},
  doi     = {10.1214/aos/1176345462}
}

@article{lehrer2012second,
  author  = {Lehrer, Gustav and Zhang, Ruibin},
  title   = {The second fundamental theorem of invariant theory for the orthogonal group},
  journal = {Annals of Mathematics},
  volume  = {176},
  number  = {3},
  pages   = {2031--2054},
  year    = {2012},
  url     = {https://arxiv.org/abs/1102.3221}
}

@article{li1995polar,
  author  = {Li, Ren-Cang},
  title   = {New Perturbation Bounds for the Unitary Polar Factor},
  journal = {SIAM Journal on Matrix Analysis and Applications},
  volume  = {16},
  number  = {1},
  pages   = {327--332},
  year    = {1995},
  doi     = {10.1137/S0895479893256359}
}

@article{mcshane1934extension,
  author  = {McShane, Edward J.},
  title   = {Extension of range of functions},
  journal = {Bulletin of the American Mathematical Society},
  volume  = {40},
  number  = {12},
  pages   = {837--842},
  year    = {1934},
  doi     = {10.1090/S0002-9904-1934-05978-0}
}

@article{chen2015improved,
  author = {Chen, Wei-Kuo and Dafnis, Nikos and Paouris, Grigoris},
  title = {Improved {H}\"older and reverse {H}\"older inequalities for {Gaussian} random vectors},
  journal = {Advances in Mathematics},
  volume = {280},
  pages = {643--689},
  year = {2015},
  eprint = {1306.2410},
  archivePrefix = {arXiv},
  url = {https://arxiv.org/abs/1306.2410}
}

@book{meckes2019random,
  author = {Meckes, Elizabeth S.},
  title = {The Random Matrix Theory of the Classical Compact Groups},
  series = {Cambridge Tracts in Mathematics},
  volume = {218},
  publisher = {Cambridge University Press},
  year = {2019},
  url = {https://case.edu/artsci/math/mwmeckes/elizabeth/Haar_book.pdf}
}

@article{hsu2012tail,
  author = {Hsu, Daniel and Kakade, Sham M. and Zhang, Tong},
  title = {A tail inequality for quadratic forms of subgaussian random vectors},
  journal = {Electronic Communications in Probability},
  volume = {17},
  number = {52},
  pages = {1--6},
  year = {2012},
  doi = {10.1214/ECP.v17-2079},
  url = {https://www.cs.columbia.edu/~djhsu/papers/quadratic-ecp.pdf}
}

@article{finner1992generalization,
  author = {Finner, Helmut},
  title = {A generalization of {H}\"older's inequality and some probability inequalities},
  journal = {The Annals of Probability},
  volume = {20},
  number = {4},
  pages = {1893--1901},
  year = {1992},
  doi = {10.1214/aop/1176989534}
}

@book{bhatia1997matrix,
  author = {Bhatia, Rajendra},
  title = {Matrix Analysis},
  series = {Graduate Texts in Mathematics},
  volume = {169},
  publisher = {Springer},
  address = {New York},
  year = {1997},
  doi = {10.1007/978-1-4612-0653-8}
}

\end{document}